\documentclass[10pt,oneside]{article}
\usepackage[T1]{fontenc}
\usepackage[english]{babel}

\usepackage{graphicx,accents}

\newcommand{\cev}[1]{\reflectbox{\ensuremath{\vec{\reflectbox{\ensuremath{#1}}}}}}

\usepackage[cal=cm]{mathalfa}

\usepackage[T1]{fontenc}
\usepackage{titlesec}

\usepackage{hhline}

\titleformat
{\chapter} 
[display] 
{\bfseries\Large\itshape} 
{Story No. \ \thechapter} 
{0.5ex} 
{
    \rule{\textwidth}{1pt}
    \vspace{1ex}
    \centering
} 
[
\vspace{-0.5ex}%
\rule{\textwidth}{0.3pt}
] 

\titleformat{\section}[wrap]
{\normalfont\bfseries}
{\thesection.}{0.5em}{}

\titlespacing{\section}{12pc}{1.5ex plus .1ex minus .2ex}{1pc}

\usepackage{tocloft} 

\usepackage{titlesec}
\titleformat{\subsection}[runin]
       {\normalfont\bfseries}
       {\thesubsection}
       {0.5em}
       {}
       [.]

\makeatletter
\@ifundefined{dddot}{}{}
\usepackage{mathdots}  
\makeatother

\makeatletter
\@ifundefined{ddddot}{}{}
\usepackage{yhmath}
\makeatother

\usepackage{geometry} 
\usepackage{xcolor} 
\usepackage{tikz} 
\usepackage[draft=false]{hyperref} 
\usepackage{mathtools}
\usepackage{amsmath,amsthm,amssymb}
\usepackage[all,cmtip]{xy}

\usepackage{fullpage}

\usepackage[capitalise]{cleveref}  
\newcommand{\clevertheorem}[3]{%
	\newtheorem{#1}[thm]{#2}
	\crefname{#1}{#2}{#3}
}
\numberwithin{equation}{section} 
\numberwithin{figure}{section} 

\theoremstyle{plain} 
\newtheorem{thm}{Theorem}[subsection]
\crefname{thm}{Theorem}{Theorems}
\newtheorem*{thm*}{Theorem}
\clevertheorem{proposition}{Proposition}{Propositions}
\clevertheorem{prop}{Proposition}{Propositions}
\newtheorem*{prop*}{Proposition}
\clevertheorem{theorem}{Theorem}{Theorems}
\clevertheorem{lemma}{Lemma}{Lemmas}
\clevertheorem{corollary}{Corollary}{Corollaries}
\clevertheorem{cor}{Corollary}{Corollaries}
\clevertheorem{conj}{Conjecture}{Conjectures}
\clevertheorem{questn}{Question}{Questions}

\theoremstyle{definition} 
\clevertheorem{definition}{Definition}{Definitions}
\clevertheorem{assumption}{Assumption}{Assumptions}
\clevertheorem{defn}{Definition}{Definitions}
\clevertheorem{notn}{Notation}{Notations}
\clevertheorem{conv}{Convention}{Conventions}

\clevertheorem{remark}{Remark}{Remarks}
\clevertheorem{rmk}{Remark}{Remarks}
\clevertheorem{warn}{Warning}{Warnings}
\clevertheorem{example}{Example}{Examples}
\clevertheorem{summ}{Summary}{Summaries}

\usepackage{amsmath,amsthm,amssymb,amsfonts,extarrows}
\usepackage{enumerate,amscd,amsxtra,MnSymbol}
\usepackage{mathrsfs}
\usepackage{color}
\usepackage[cmtip,all]{xy}
\usepackage{eucal}
\usepackage{bbold}
\usepackage{upgreek}
\usepackage{paralist}
\usepackage{tikz-cd}
\usepackage{dsfont}
\usepackage{bbm}
\usepackage[bbgreekl]{mathbbol}

\DeclareMathSymbol\bbDelta \mathord{bbold}{"01}
\DeclareMathSymbol\bDelta \mathord{bbold}{"01}

\newtheorem{remark*}{Remark}
\newtheorem{convention}[thm]{Convention}
\newtheorem{construction}[thm]{Construction}

\newtheorem{notation}[thm]{Notation}

\newcommand{\bD}{{\mathbb D}}

\renewcommand{\P}{{\mathbb P}}

\newcommand{\bN}{{\mathbb N}}

\newcommand{\bS}{{\mathbb S}}

\newcommand{\mA}{{\mathcal A}}
\newcommand{\mB}{{\mathcal B}}
\newcommand{\mC}{{\mathcal C}}
\newcommand{\mD}{{\mathcal D}}
\newcommand{\mE}{{\mathcal E}}
\newcommand{\mF}{{\mathcal F}}

\newcommand{\mL}{{\mathcal L}}
\newcommand{\mM}{{\mathcal M}}
\newcommand{\mN}{{\mathcal N}}
\newcommand{\mO}{{\mathcal O}}
\newcommand{\mP}{{\mathcal P}}

\newcommand{\mR}{{\mathcal R}}
\newcommand{\mS}{{\mathcal S}}

\newcommand{\mV}{{\mathcal V}}
\newcommand{\mW}{{\mathcal W}}

\newcommand{\mY}{{\mathcal Y}}

\newcommand{\A}{A}
\newcommand{\B}{B}
\newcommand{\C}{C}

\newcommand{\E}{{E}}
\newcommand{\F}{{F}}
\newcommand{\G}{{G}}

\renewcommand{\L}{{\mathrm L}}

\newcommand{\N}{{\mathrm N}}
\renewcommand{\P}{{P}}
\newcommand{\Q}{{Q}}
\newcommand{\R}{{\mathrm R}}
\newcommand{\rS}{{S}}

\newcommand{\V}{{V}}

\newcommand{\X}{X}
\newcommand{\Y}{Y}

\newcommand{\bj}{{j}}
\newcommand{\bi}{{i}}
\newcommand{\m}{{m}}
\newcommand{\bk}{{k}}

\newcommand{\g}{{g}}
\newcommand{\n}{{n}}

\newcommand{\nd}{\mathrm{nd}}

\newcommand{\op}{\mathrm{op}}

\newcommand{\colim}{\mathrm{colim}}
\newcommand{\Mod}{{\mathrm{Mod}}}

\newcommand{\rev}{{\mathrm{rev}}}

\newcommand{\ot}{\otimes}

\newcommand{\co}{\mathrm{co}}

\newcommand{\univ}{\mathrm{univ}}
\newcommand{\strict}{\mathrm{strict}}

\renewcommand{\S}{\mathcal{S}}

\newcommand{\id}{\mathrm{id}}
\newcommand{\Cat}{\mathrm{Cat}}

\newcommand{\Set}{\mathrm{Set}}

\newcommand{\Alg}{\mathrm{Alg}}

\newcommand{\Fun}{\mathrm{Fun}}

\newcommand{\lax}{{\mathrm{lax}}}
\newcommand{\oplax}{{\mathrm{oplax}}}

\newcommand{\Bord}{\mathrm{Bord}}

\newcommand{\tu}{{\mathbb 1}}
\newcommand{\coop}{\mathrm{coop}}

\newcommand{\Map}{{\mathrm{Map}}}

\newcommand{\bZ}{{\mathbb{Z}}}

\newcommand{\Mor}{{\mathrm{Mor}}}

\newcommand{\PrL}{\mathrm{Pr^L}}
\newcommand{\PrR}{\mathrm{Pr^R}}

\newcommand{\cop}{\mathrm{cop}}

\newcommand{\CAT}{\mathrm{CAT}}

\newcommand{\RMor}{\mathrm{RMor}}

\newcommand{\Ho}{\mathrm{Ho}}
\newcommand{\sk}{\mathrm{sk}}

\newcommand{\cube}{{\,\vline\negmedspace\square}}

\newcommand{\scat}{\mathcal{C}\mathit{at}}

\newcommand{\fcat}{\mathfrak{Cat}}

\author{David Gepner and Hadrian Heine}
\title{Homotopy Posets, Postnikov Towers, and\\
Hypercompletions of $\infty$-Categories}

\date{}

\begin{document}

\maketitle

\begin{abstract}
    We show that basic homotopical notions such as homotopy sets and groups, connected and truncated maps, cellular constructions and skeleta, etc., extend to the setting of $(\infty,\infty)$-categories, as well as to presentable categories enriched in $(\infty,\infty)$-categories under the Gray tensor product.
    The homotopy posets of an $(\infty,\infty)$-category are indexed by boundaries of categorical disks; in particular, there is a fundamental poset for each pair of objects, which we regard as a oriented point where the source and target objects have opposite orientation.
    In contrast to the situation in topology, weakly contractible geometric building blocks such as oriented polytopes typically have nontrivial homotopy posets.
    
    The homotopy posets assemble to form an oriented analogue of the long exact sequence of a fibration and form the layers of a categorical Postnikov tower, which converges for any $(\infty,n)$-category but not for general $(\infty,\infty)$-categories.
    We show that the full subcategory consisting of the Postnikov complete $(\infty,\infty)$-categories is obtained by inverting the Postnikov equivalences and canonically identifies with the limit of the categories of $(\infty,n)$-categories taken along the truncation functors.
    We also study truncated morphisms in general oriented categories and connected morphisms in presentable oriented categories. 
\end{abstract}

\tableofcontents

\vspace{.5cm}
\section{Introduction}

\subsection{Categorical symmetries}
Heuristically, $(\infty,1)$-category theory is ordinary $1$-category theory with sets replaced by the more refined notion of $\infty$-groupoid.\footnote{Also called space, homotopy type, or anima in the literature.}
The point is that systemically replacing set theory with homotopy theory in certain algebraic and geometric contexts is often useful both in theory and in practice.
Combined with the computational techniques of algebraic topology, such as homotopy groups, filtrations, spectral sequences, etc., this often results in a rich interplay between homotopical and classical mathematics.

One standard filtration, from which many important spectral sequences can be derived, is the Postnikov tower $\{\tau_{\leq n}X\}$ of a homotopy type $X$.
The Postnikov tower of $X$ allows for its reconstruction in terms of the algebraic data of its homotopy groups $\{\pi_n X\}$ together with the (nonabelian) cohomological data needed to classify $\tau_{\leq n} X$ as a fibration over $\tau_{\leq n-1} X$ with fiber $\mathrm{B}^n\pi_n X$.\footnote{The zero truncation $\tau_{\leq 0} X\cong\pi_0 X$ is the underlying set of path components, which splits $X$ as a disjoint union of connected components and allows us to assume without loss of generality that $X$ is connected.}
The layers of the Postnikov tower of $X$, the Eilenberg--MacLane spaces $\mathrm{B}^n\pi_n X$, are spaces with only a single homotopy group $\pi_n X$, shifted up into homotopical degree $n$.
Thinking of a group as encoding a set of symmetries, 
we see that automorphisms of objects (or morphisms, or higher morphisms) of $X$ are precisely its internal symmetries, graded by dimension, and can be strung together with the appropriate twisting data to recover $X$ itself.

Nevertheless, for various applications, especially in topological quantum field theory and other areas related to mathematical physics and representation theory, it seems useful to be able to encode so-called {\em categorical} symmetries, and not simply {\em groupoidal symmetries}, meaning symmetries which are inherently invertible.
Categorical symmetries are inherently antisymmetric, as they encode relations which are directed and thus not generally invertible.
It seems useful to think of these as {\em oriented} symmetries, in which the orientation can always be reversed to obtain the inverse symmetry, though at the expense of passing to some mirror image version of the original structure.
A higher category $X$ need not be equivalent to its opposite, and when it is, at least in a suitably coherent way, we obtain a category with duality.
Higher groupoids have this structure, where objects are self-dual and the dual of a morphism is its inverse.

Already in the groupoidal case, it is necessary for both theoretical and practical reasons to not work in the truncated setting of $n$-groupoid for some finite $n$, but to pass to the limit as $n$ goes to infinity.
Likewise, and for the same reasons, we cannot work in the truncated setting of $n$-categories,\footnote{By which we will always mean $(\infty,n)$-category, and not ordinary or strict $n$-category, unless otherwise stated.}
but we must again pass to the limit as $n$ goes to infinity.
This begs the question: to what extent are $\infty$-categories themselves a suitable replacement for $(\infty,0)$-categories, at least as far as category theory and its applications to algebra, arithmetic, geometry, mathematical physics, representation theory, etc., are concerned?

\subsection{Higher homotopy theory}
One of the main goals of this paper is to show how many of the elementary concepts of homotopy theory are specializations of notions which exist on the level of $\infty$-categories, or more precisely, oriented categories, which are categories enriched in $\infty$-categories but with respect to the Gray monoidal structure.
The most basic invariants in homotopy theory, the homotopy groups, arise as equivalence classes in morphism spaces.
The fundamental group is the space of loops modulo homotopy, while the higher homotopy groups are equivalence classes of iterated loops.
Thus the most fundamental structures are the loop functor $\Omega$ (together with its left adjoint companion, the suspension functor $\Sigma$) and the homotopy set functor $\pi_0$.

Unfortunately, the categorical suspension functor is not a functor of $\infty$-categories.
This is because it is incompatible with the cartesian monoidal structure on $\infty\Cat$.
However, it is a functor of antioriented categories (that is, categories enriched with respect to the reverse Gray tensor product), as it is compatible with the Gray tensor product, as evidenced by the {\em suspension formula}, i.e. the pushout square
\[
\xymatrix{
X \boxtimes \partial\bD^1\ar[r]\ar[d] & \partial\bD^1\ar[d]\\
X \boxtimes \bD^1\ar[r] & S(X)},
\]
where $\boxtimes:\infty\Cat\times\infty\Cat\to\infty\Cat$ denotes the Gray tensor product.
Hence we should regard its right adjoint, the morphism object, which fits into the pullback square
\[
\xymatrix{
\Mor_X(s,t)\ar[r]\ar[d] & X^{\bD^1}\ar[d]\\
\ast\ar[r]^{(s,t)} & X^{\partial\bD^1}},
\]
as a functor of antioriented categories as well.
The same holds for the reduced suspension $\Sigma$ and its right adjoint the endomorphism object construction $\Omega$, which we might think of as oriented loops.

Since the basic categorical operations of bipointed suspension, its adjoint, the morphism object construction, are defined using the Gray tensor product and not the cartesian product, it should come as no surprise that any higher categorical refinement of homotopy can only exist in the oriented setting.
Fortunately, it turns out that many of the basic homotopical notions and constructions have meaningful oriented analogues, and we hope and expect that in time, most of the apparatus of abstract homotopy theory will be usefully imported into the oriented world so as to have a substantial supply of theoretical tools, computational techniques, and results at our disposal.




\subsection{Oriented categories}
As outlined above, in order for the basic operations to be sufficiently functorial, we must view them as (anti)oriented categories; that is, categories enriched in the (reverse) Gray tensor monoidal structure.
However, this is also consistent with some geometric intuition concerning these basic objects, which we tend to think of as complexes consisting of $n$-cells, but where the cells are oriented, or directed, from source to target.
We already depict higher categorical diagrams in this way, and ideally we might be able to formalize this geometry via a sort of lax higher topos theory, which provides rules for how limits and colimits interact, and hopefully resulting in a ``place where one can do mathematics''.
To simplify terminology and notation, we will henceforth refer to $(\infty,n)$-categories simply as $n$-categories, and use the term $(n,n)$-category or strict $n$-category when we wish to refer to stricter variants.
In particular, the category of $\infty$-categories, the limit of the categories of $n$-categories, where the limit is taken along the core functors, which (roughly speaking, and depending on whether or not one is imposing univalence) discard the top level of (noninvertible) morphisms.

The category of $\infty$-categories is not only enriched in itself via the cartesian product, but it is also enriched in itself via the lax or oplax Gray tensor product, two related but nonidentical monoidal structures, each of which is the reverse of the other.
An oriented category is simply a category enriched in the category of $\infty$-categories, but viewed as a monoidal category with respect to the lax Gray tensor product.
This is a fundamental notion because the cartesian enrichment of the category of $\infty$-categories is not very useful in practice, as it is incompatible with categorical dimension in a certain precise sense which renders it impossible for natural operations, like the categorical suspension, to be compatible with the cartesian enrichment.

Hence it is the oriented category of $\infty$-categories which is the proper analogue (in lax higher mathematics) of the category of spaces (in higher mathematics), or the category of sets (in classical mathematics).
In keeping with this analogy and our geometric intuition for higher categories, we refer to these objects as oriented spaces.
Oriented categories have underlying categories, but they do not have underlying $\infty$-categories.
This is because an oriented category only satisfies lax version of the interchange law.


\subsection{Homotopy posets}
There are potentially many higher categorical analogues of the notion of homotopy set, or group, depending on which aspects one would like to generalize.
We shall take the viewpoint that perhaps the most important property of the homotopy sets is that they are sets, i.e. discrete invariants.
Higher categorically, it is no longer reasonable to expect further algebraic structure (like a multiplication) on homotopy sets, as we wish characterize images of general morphisms, not just endomorphisms or, in the case of spaces, automorphisms.

In the oriented setting, the set of components $\pi_0$ is the quotient of the set of equivalence classes objects by the equivalence relation where $x\sim y$ whenever there are morphisms $x\to y$ and $y\to x$.
This naturally has the structure of a poset, where the equivalence class of an object $x$ is less than or equal to that of an object $y$ precisely if there is a morphism $x\to y$.
If $x\neq y$, this implies that there is no morphism $y\to x$.
This reduces to the usual set of path components in a space, since all morphisms are invertible, and consequently the poset structure is trivial.

In ordinary homotopy theory, the topological $n$-simplex is contractable.
In particular, the homotopy groups of the topological $n$-simplex are all trivial.
This ceases to be the case categorically.
Specifically, the $n$-simplex $\Delta^n$, viewed as a category, is not contractible.
Rather, viewing $\Delta^n$ as a poset, or $(0,1)$-category, the canonical map $\Delta^n\to\tau_{\leq 0}\Delta^n\cong\pi_0\Delta^n$ is an equivalence.
Of course, the higher homotopy posets still vanish, as there is no higher categorical structure on $\Delta^n$ itself.
But this is because, at least from the perspective of $\infty$-category theory, the $n$-simplex itself most naturally arises as the homotopy $1$-category of the oriented $n$-simplex $\bDelta^n$, which is by far the more fundamental object.
Its homotopy posets are quite complex, and are nontrivial in all dimensions $m\leq n$.

\subsection{Strong and weak equivalences}
One of the more useful notions in homotopy theory is that of an $n$-equivalence.
Recall that a map $f:Y\to X$ is an $n$-equivalence if for all basepoints $y\in Y$, the induced map of sets $\pi_m(Y,y)\to\pi_m(X,f(y))$ is an isomorphism for $m<n$ and an epimorphism for $m<n$, or equivalently, the homotopy fibers $Z=Y\times_X \{x\}$ at all basepoints $x\in X$ are $n$-connective, in the sense that the homotopy sets $\pi_m(Z,z)$ based at all $z\in Z$ vanish for $m<n$.

The problem with simple discrete invariants is that there is little reason to expect them to form a conservative family.
In some sense this is also true in ordinary homotopy theory, as the homotopy sets are not functors from spaces to sets, as they rely on a choice of basepoint.
Likewise, the homotopy posets of an $\infty$-category $X$ are indexed by oriented basepoints in all dimensions, which is to say maps of the form $\partial\bD^n\to X$.
However, even with all possible choices of basepoint, the homotopy posets fail to be jointly conservative.
It is simply too weak of a notion to be useful, at least outside of a natural and identifiable class of $\infty$-categories on which they are conservative, which fortunately includes natural combinatorial examples such as Steiner $\infty$-categories.

Nevertheless, there is a useful notion of weak $n$-equivalence, which in the limit recovers the notion of Postnikov equivalence, i.e. induces an equivalence on all truncations.
There is a subtlety here in that there are competing notions of $n$-truncation, depending on whether or not one wishes to reflect to the category of $n$-categories, $(n,n)$-categories, or $(n,n+1)$-categories (stated in order of obviousness and perhaps reverse order of utility).
Specifically, a functor $f:Y\to X$ is a 
weak $n$-equivalence if its $(n,n)$-categorical reflection is an equivalence and a Postnikov $n$-equivalence if its $(n,n+1)$-categorical reflection is an equivalence.

The difference between the $(n,n)$-categorical and $(n,n+1)$-categorical reflection of an $\infty$-category $X$ is precisely the difference between its various $n$-dimensional homotopy sets and $n$-dimensional homotopy posets.
The latter additionally remembers the partial ordering on the highest dimension morphism objects, which are now only sets by definition of $(n,n)$-category, given by the existence of an $(n+1)$-morphism with specified source and target $n$-morphism.
The fact that this is a partial ordering is due to the fact that the existence of an $(n+1)$-morphism in both directions already force the two $n$-morphisms to be equal.

\subsection{Connected and truncated functors}
Just as in ordinary homotopy theory, there are notions of connected and truncated objects in $\infty\Cat$.
It is convenient to formulate these notions relatively, in terms of connected and truncated functors.
Fortunately, these correspond to the familiar categorical concepts of full and faithful, and indeed there is a hierarchy of these notions indexed by the natural numbers, as well as $\infty$ and $-1$ (or even $-2$ depending on the connected versus connective convention).
These notions are orthogonal to one another, and determine a factorization system on the oriented category of oriented spaces, and therefore on any presentable oriented category.

While these families of truncation endofunctors exist for any presentable oriented category, their left orthogonal functors, the connective covers, only seem to exist in presentable oriented categories which arise as oriented pullback preserving localizations of oriented presheaf categories.
This is because the morphism object functor decreases connectivity by one, because its left adjoint suspension functor increases connectivity by one.
This occurs provided the localization functor commutes with oriented pullbacks.

We deduce many of our results in this paper from abstract properties of factorization systems in the setting of enriched categories.
Thus we begin with some useful results on factorization systems in general, after which we specialize to the case of enrichment in the Gray tensor product.
From this it is relatively straightforward to construct the hierarchy of connected and truncated factorization systems on the  oriented category of $\infty$-categories.
This is later used to study truncated objects in general oriented categories, as well as connected objects under the additional assumption that the left orthogonal of the class of truncated morphisms is stable under oriented basechange.
The $n$-connective and $n$-faithful factorization system on $\infty\Cat$ has also been studied by Loubaton, see \cite{loubaton2024categorical} and \cite{loubaton2025effectivity}. 

\subsection{Connections}
In homotopy theory, two points in a space are equivalent if they are connected by a path.
This is because a path $f$ from a point $A$ to a point $B$ in $X$ admits an inverse path $g$ from $B$ to $A$.
Hence, the set $\pi_0 X$ of equivalence classes in $X$ is the set of path components, meaning that we identify objects when they are connected.

Higher categorically, it is not enough to have a single oriented path $f$ connecting two objects $A$ and $B$ in an $\infty$-category $X$, since there is no reason for $f$ to be invertible.
Thus, part of the data of an equivalence between $A$ and $B$ would also involve a path $g$ from $B$ to $A$.
While this is not the full data of an equivalence, it is still a useful notion, which we call a $0$-connection.

In other words, a zero connection is the data of two objects and two oppositely oriented arrows between them.
More precisely, the walking $0$-connection is the $1$-category obtained by gluing $\bD^1$ and $\bD^{1\op}$ along their common boundary $\partial\bD^1$, which can be depicted as follows:
\[
\begin{tikzcd}
A \arrow[r, bend left=30, "f"] & B \arrow[l, bend left=30, "g"]
\end{tikzcd}
\]
Again, if we wish to specify an equivalence between $A$ and $B$, two oriented paths $f$ and $g$ are not enough, we must also identify $\id_A$ with $gf$ and $\id_B$ with $fg$.
This involves the data of an oriented paths $\alpha$ from $\id_A$ to $gf$ and $\beta$ from $gf$ to $\id_B$, as well as oriented paths $\gamma$ from $\id_B$ to $fg$ and $\delta$ from $fg$ to $\id_B$.
Altogether, this is the data of a $1$-connection between $A$ and $B$, which may be depicted as follows:
\[
\begin{tikzcd}
A \arrow[r, bend left=30, "f"] & B \arrow[l, bend left=30, "g"]
\end{tikzcd}
\qquad
\begin{tikzcd}
\id_A \arrow[r, bend left=30, "\alpha"] & gf \arrow[l, bend left=30, "\beta"]
\end{tikzcd}
\qquad
\begin{tikzcd}
\id_B \arrow[r, bend left=30, "\gamma"] & fg \arrow[l, bend left=30, "\delta"]
\end{tikzcd}
\]
The data of an equivalence between the objects $A$ and $B$ of $X$ continues indefinitely, though we can truncate at categorical level $n+1$ to obtain a notion of $n$-connection. See \cref{connection} for details.

However, there is a difference between an equivalence and an $\infty$-connection.
We say that an $\infty$-category $X$ is hypercomplete if every $\infty$-connection in $X$ and all iterated morphism $\infty$-categories of $X$ are equivalences.
This agrees with the standard notion of hypercompletion in topos theory as those objects which are local with respect to all $\infty$-connected maps.

\subsection{Postnikov and hyper completions}
The Postnikov tower is an important tool in homotopy theory since it describes a homotopy type $X$ as an inverse limit its $n$-truncations $\tau_{\leq n} X$.
Moreover, the fiber of the map $\tau_{\leq n}X\to\tau_{\leq n-1} X$ over a point $x$ is the Eilenberg-MacLane space $\mathrm{B}^n\pi_n(X,x)$, so that the layers of the Postnikov tower are extremely elementary objects, at least from the perspective of homotopy groups.

One natural candidate for an analogue of the Postnikov tower in higher category theory is the dimension tower, which compares $X$ to the limit of its $n$-categorical truncations
\[
X\to\lim\{\cdots\to\tau_n X\to\tau_{n-1} X\to\cdots\to\tau_0 X\}.
\]
While this is a useful tower for many purposes, it has two major drawbacks.
First, and most importantly, the layers of this tower are not particularly simple objects, but rather look like deloopings of $n$-tuply monoidal categories.
Secondly, this tower need not converge in general, though this is difficultly cannot be avoided.

The more useful tower is the categorical Postnikov tower, which is compares $X$ to the limit of its $(n,n+1)$-categorical reflections $\tau_{\leq n}X\in(n,n+1)\Cat$.
Inductively, an $(n,n+1)$-category is a category enriched in $(n-1,n)$-categories, where a $(0,1)$-category is by definition a poset.
In particular, an $(n-1,n)$-category is a special case of an $(n,n)$-category, and is a much more combinatorial sort of object than a general $n$-category, which according to our notational convention means an $(\infty,n)$-category.
For instance, a $(1,2)$-category is a $(1,1)$-category (an category in the classical sense of the term) in which the morphism sets are equipped with a partial ordering such that composition respects the partial ordering.

The resulting Postnikov tower
\[
X\to\lim\{\cdots\to\tau_{\leq n} X\to\tau_{\leq{n-1}} X\to\cdots\to\tau_{\leq 0} X\}
\]
has the same limit of that of the dimension tower, and therefore will not  converge in general.
This leads to the notion of the Postnikov completion of an $\infty$-category, which is the identity when restricted to $n$-categories, but is not in general.

The natural question which then arises is to characterize the full subcategory of $\infty\Cat$ consisting of the Postnikov complete objects.
This has been viewed as one of the major outstanding questions in the basic foundations of the subject.
Indeed, some authors prefer to define the category of $\infty$-categories ``coinductively'' as the limit of the tower
\[
\lim\{\cdots\overset{\tau_n}{\to} n\Cat\overset{\tau_{n-1}}{\to} (n-1)\Cat\overset{\tau_0}{\to}\cdots\to 0\Cat\}
\]
in which the maps are given by passing to the $n$-categorical truncation instead of the $n$-categorical core.

We show that the this limit can be identified with the full subcategory of $\infty\Cat$ spanned by the Postnikov complete $\infty$-categories, and thus that this alternative notion of $\infty$-category results in a reasonable notion which compares well to the standard notion.
Another way of saying this, more in keeping with type-theoretic approaches to the subject, is that the full subcategory of Postnikov complete $\infty$-categories is the localization of $\infty\Cat$ itself with respect to the so-called Postnikov equivalences.
This result has also been obtained recently by Ozornova--Rovelli--Walde \cite{ozornova2026core}.

There at least one other useful notion of completion in $\infty\Cat$, corresponding to the notion of hypercompletion in topos theory.
Namely, $\infty$-connected morphisms need not be equivalences in general, they become equivalences after passing to any truncation.
Thus, one can also invert $\infty$-connected morphisms in $\infty\Cat$, in order to obtain the full subcategory of hypercomplete $\infty$-categories.
Any Postnikov complete $\infty$-category is hypercomplete, but the converse does not hold in general.

In addition to the fact that any $n$-category is Postnikov complete, hence also hypercomplete, we also identify a natural class of (not necessarily truncated) $\infty$-categories which are necessarily Postnikov complete.
We call these $\infty$-categories weakly directed, though it contains are more readily identifiable subclass, the directed $\infty$-categories.
Roughly, an $\infty$-category $X$ is directed if, whenever there are morphisms $A\to B$ and $B\to A$, then both morphisms are equivalences, in $X$ as well as all its higher morphism $\infty$-categories.

\subsection{Filtrations and obstructions}
In homotopy theory, objects of the unstable homotopy category can be represented by cell complexes, in which the cells can be added in order of dimension.
This guarantees that any space can be filtered by skeleta, which inductively arise by attaching cells along their boundaries.
Moreover, up to homotopy, maps between cell complexes are again cellular, in the sense that they arise inductively by as maps between the respective $n$-skeleta.
This is a fundamental fact which is useful in both theory and computation.

In higher category theory, there are also natural notions of cell complex and $n$-skeleton.
Perhaps the most basic and obvious is the decomposition of an $\infty$-category $X$ as the colimit of its $n$-categorical cores; that is, the map $\colim\iota_n X\to X$ is always an equivalence.
Similarly, any functor $f:X\to Y$ is a colimit of $\iota_n f:\iota_n X\to\iota_n X$, since the composition $\iota_n X\to X\to Y$ factors uniquely through $\iota_n Y$ by the universal property of the $n$-core.
The disadvantage of this approach is that it is more difficult to understand the cofibers of the inclusions $\iota_{n-1} X\to\iota_n X$, which in topology is always a wedge of $n$-spheres.
This is important since it provides natural obstructions on the level of homotopy groups to extending a map $\sk_{n-1} X\to Y$ to a map $\sk_n X\to Y$.

We obtain a more useful cellular obstruction theory if we use a different 
skeletal filtration.
The idea is to mimic more closely the simplicial notion of skeleton, where the $n$-skeleton of a simplicial set $X$ is defined as the left Kan extension $\sk_n X=i_!i^*X$ of its restriction along the fully faithful inclusion $i:\Delta_{\leq n}\to\Delta$.
This works because the category of $\infty$-categories admits a presentation in terms of Street's oriented version of simplices, often referred to as the orientals, as we show in \cite{gepner2025oriented}.

For technical reasons, it is more useful to filter by orientals $\bDelta^m$ of dimension less than or equal to $n$ together with their codimension one truncations $\tau_{m-1}\bDelta^m$.
Moreover, in order to retain computability, and adhere closer to the simplicial analogue, we choose to approximate $\infty$-categories not by the full subcategory of the (truncated) orientals of dimension less than or equal to $n$, but by a certain subcategory of (roughly) atomic maps between them.
Without the dimension restriction, this is equivalent to the category $\Delta^+$ of simplices and their codimension $1$ truncations, which has been considered by e.g. \cite{Complicial}.
This category admits a filtration by dimension, and it is the category $\Delta^+_{\leq n}$ which controls the $n$-skeleton of an $\infty$-category $X$.

The advantage of this filtration is that the cofiber of the canonical map $\sk_{n-1}X\to\sk_n X$ is in fact a wedge of $n$-spheres, though due to the natural of the generators in degree $n$ (namely, $\bDelta^n$ and $tau_{n-1}\bDelta^n$), the wedge decomposes into a wedge of categorical $n$-spheres 
\[
\bDelta^n/\partial\bDelta^n\simeq\bD^n/\partial\bD^{n}\simeq\Sigma^{n-1}\mathrm{B}\bN
\]
and homotopical $n$-spheres
\[
\tau_{n-1}\bDelta^n/\partial\bDelta^n\simeq\bD^{n-1}/\partial\bD^n\simeq\Sigma^{n-1}\mathrm{B}\bZ,
\]
where $\Sigma$ denotes the reduced pointed categorical suspension, which agrees with the homotopical suspension whenever the pointed $\infty$-categories in question are actually pointed $\infty$-groupoids.

More precisely, the $n$-skeleton is obtained from the $n-1$-skeleton by attaching cells $\bD^n$ and $\tau_{n-1}\bD^n\simeq\bD^{n-1}$ along their boundaries $\partial\bD^n$.
In the former case, this amounts to adding new $n$-cells, and in the later case, this amounts to identifying existing $n-1$-cells.
In any case, this fact that map from the $n-1$-skeleton to the $n$-skeleton is cobasechanged from a disjoint union of elementary maps of the form $\partial\bD^n\to\bD^n$ and $\partial\bD^n\to\bD^{n-1}$ yields a straightforward obstruction theory to producing functors of $\infty$-categories cellularly.






\subsection{Main results}


In homotopy theory it is central to study homotopy types via
concepts of truncation and connectivity, which are measured by the homotopy groups and the Postnikov tower. It is goal of this work to extend these concepts from homotopy theory to higher category theory.

In homptopy theory one considers for every natural number $n \geq -2$ a notion of $n$-truncated and $n$-connected map of spaces,
which are the left and right class in a factorization system on the category of spaces.
This in turn induces a notion of $n$-truncated object in any category, and if additionally the category is presentable, one also obtains a notion of $n$-connected object (though it might be of limited utility if the category is not a topos).
A central in this paper is that the same hold in the oriented world.

\begin{theorem}[\cref{Grayfactor}]
Let $\mC$ be a presentable oriented category and $n \geq -2.$
There is an accessible oriented factorization system on $\mC$ whose left class precisely consists of the $n$-connected morphisms and whose right class precisely consists of the $n$-truncated morphisms.
Moreover, if $f^*:\mC\rightleftarrows\mD:f_*$ is an oriented adjunction such that $f^*$ preserves oriented pullbacks, then $f^*$ preserves $n$-connected and $n$-truncated morphisms. 
\end{theorem}

In general, the right adjoint $f_*$ of an oriented adjunction will always preserve $n$-truncated morphisms.
Dually, the left adjoint $f^*$ of an oriented adjunction will always preserve $n$-connected morphisms.
However, since truncation is such a fundamental operation, in practice it is sometimes useful to require that $f^*$ preserves {\em oriented} pullbacks in order for it to be compatible with truncation.

In the special case of $\infty\Cat^\univ$, we deduce that the classes of $n$-connected and $n$-truncated functors form a factorization system on the oriented category of $\infty$-categories.
In this case the compatibility with the oriented structure together with the antimonoidal involution $(-)^\op: \infty\Cat \to \infty\Cat$ is equivalent to say that $n$-connected functors are closed under the Gray tensor product.

\begin{corollary}[\cref{mainfact},\cref{Grayfacto}]
Let $n \geq -2$.
There is a factorization system on $\infty\Cat^\univ$, where the left class consists of the $n$-connected ($n+1$-connective) functors and the right class consists of the $n$-truncated ($n+1$-faithful) functors. This factorization system on $\infty\Cat^\univ$ is compatible with the Gray tensor product.
\end{corollary}


Every factorization system gives rise to a localization, where an object is local if and only if the morphism to the final object belongs to the right class. The same holds for oriented localization systems:

\begin{corollary}\label{orientedpost} Let $\mC$ be a presentable oriented category and $n \geq -2.$
The embedding of the full oriented subcategory $$\tau_{\leq n}\mC \subset \mC $$
of $\n$-truncated objects in $\mC$ admits an oriented left adjoint $\tau_{\leq n}: \mC \to \tau_{\leq n}\mC.$
\end{corollary}

\cref{orientedpost} exhibits $\n$-truncated objects in any oriented category as the local objects of an oriented localization.
For instance, for $\infty\fcat^\univ$ the the $n$-truncated objects are precisely the $(n,n+1)$-categories, which therefore are the local objects of an oriented localization.
Since $n$-truncated objects are in particular $n+1$-truncated, 
for every presentable oriented category $\mC$ there is a nested sequence of full oriented subcategories $\tau_{\leq 0}(\mC) \subset \tau_{\leq 1}(\mC) \subset\cdots$. 
Consequently, \cref{orientedpost} provides a left adjoint oriented functor 
$$ \mC \to \lim_{n \geq 0} \tau_{\leq n} \mC$$
whose right adjoint sends every $(Y_0, ...) \in \lim_{n \geq 0} \tau_{\leq n} \mC$ to the limit of the tower
$$ ... \to Y_2 \to \tau_{\leq 1}(Y_2)\simeq Y_1 \to \tau_{\leq 0}(Y_1) \simeq Y_0.$$
The unit \begin{equation}\label{posttower}
X \to \lim_{n \geq 0} \tau_{\leq n}(X) \end{equation} of $X \in \mC$ approximates $X$ by a tower of increasing truncations of $X$.

For $\mC= \infty\fcat^\univ$ we study this tower as a higher categorical analogue of the Postnikov tower of an $\infty$-category.
This motivation comes from the fact that if $\mC=\mS$ is the category of spaces viewed as an oriented category, then the morphism (\ref{posttower}) is the equivalence to the limit of the Postnikov tower. More generally, for every presentable category viewed as an oriented category,
the morphism (\ref{posttower}) is the equivalence to the limit of the Postnikov tower.

Unlike in topology,
$\infty$-categories generally fail to be Postnikov complete, i.e. 
the morphism (\ref{posttower}) is generally not an equivalence for
$\mC= \infty\fcat^\univ$.
However we show that a large class of $\infty$-categories is Postnikov complete. We prove that the full subcategory of Postnikov complete $\infty$-categories is a localization and compute the $\infty$-category of Postnikov complete $\infty$-categories:

\begin{theorem}[\cref{Postnikov}]
For every $\infty$-category $X $ the morphism (\ref{posttower}) is an equivalence.
Consequently, the oriented functor
$$\tau_{\leq \infty}: \infty\fcat^\univ \to \lim_{n \geq 0}  \tau_{\leq n}\infty\fcat^\univ $$ 
is an oriented localization.
The underlying category of $\lim_{n \geq 0}  \tau_{\leq n}\infty\fcat^\univ$ identifies with the limit of the tower
$$\cdots\to \n\Cat \to\cdots\to 0\Cat $$
along the left adjoints of the natural embeddings.
\end{theorem}

A description of the limit of the tower of the tower of the $n$-categorical truncations is a fundamental question in higher category theory and has been taken by many authors to be a reasonable notion of $\infty$-category itself.
This shows that this is indeed the case, and that this formulation of $\infty$-category leads to a full subcategory of the standard notion of $\infty$-category.
This question and the ideas surrounding the notion of ``coinductive'' equivalences and the limit of the tower $\cdots\to n\Cat\to (n-1)\Cat\to\cdots$ along the truncation functors has been the subject of considerable work by a number of authors.
See for instance \cite{ara2020folk}, \cite{LAFONT}, as well as recent work of Ozornova--Rovelli--Walde \cite{ozornova2026core}, who obtain a similar characterization of the localization of $\infty\Cat$ with respect to the Postnikov equivalences.

We prove that the class of weakly directed $\infty$-categories, which contains all finite dimensional $\infty$-cateories and so in particular all compact $\infty$-categories, and also contains all loopfree $\infty$-categories and so all Steiner $\infty$-categories, is Postnikov-complete.

\begin{theorem}[\cref{weaklydirpost}]
Every weakly directed $\infty$-category is Postnikov-complete.
    
\end{theorem}






We prove that truncted and connected functors can be detected by oriented fibers. This fails for fibers unless we consider fibrations of $\infty$-categories.

\begin{theorem}[\cref{trunfiber}, \cref{fiberwiseeqq}]
Let $n \geq -2.$
Let $X \to S, \ Y \to S$ be functors and
$\phi: X \to Y$ a functor over $S.$
The following are equivalent:

\begin{enumerate}[\normalfont(1)]\setlength{\itemsep}{-2pt}
\item The functor $\phi: X \to Y$ is $n$-connected ($n$-truncated).
    
\item For every functor $T \to S$ the induced functor $T \overset{\to}{\times}_S \phi: T \overset{\to}{\times}_S X \to T \overset{\to}{\times}_S Y$ is $n$-connected ($n$-truncated).

\item For every object $s \in S$ the induced functor $\{ s\} \overset{\to}{\times}_S \phi: \{ s\}\overset{\to}{\times}_S X \to \{ s\} \overset{\to}{\times}_S Y$ is $n$-connected ($n$-truncated).
    
\end{enumerate}

If $X \to S, \ Y \to S$ are cocartesian fibrations and
$\phi: X \to Y$ a map of cocartesian fibrations over $S,$
then conditions (1) -(3) are moreover equivalent to the condition that for every object $s \in S$ the induced functor
$X_s \to Y_s$ on the fiber over $s$ is $n$-connected ($n$-truncated).

\end{theorem}

It is a crucial property in homotopy theory that $n$-connected and $n$-truncated maps of spaces can inductively defined via the notions of homotopy pullbacks. 
We obtain a similar characterization of $n$-connected and $n$-truncated functors of $\infty$-categories in terms of oriented pullbacks:

\begin{corollary}[\cref{orientedchar}]
Let $ n \geq -1.$

\begin{enumerate}[\normalfont(1)]\setlength{\itemsep}{-2pt}
\item A functor $ X \to Y $ is $n$-truncated 
if and only if the canonical functor
$$ X \overset{\to}{\times}_X X \to X \overset{\to}{\times}_Y X $$
is $n-1$-truncated. 


\item A functor $ X \to Y $ is $n$-connected
if and only if it is essentially surjective and the canonical functor
$$ X \overset{\to}{\times}_X X \to X \overset{\to}{\times}_Y X $$
is $n-1$-connected.


\end{enumerate}

\end{corollary}

The statement about $n$-truncation holds more generally in any oriented category with oriented pullbacks.

\begin{corollary}
Let $ n \geq -1$ and $\mC$ an oriented category that admits oriented pullbacks.
A morphism $ X \to Y $ in $\mC$ is $n$-truncated 
if and only if the canonical morphism
$$ X \overset{\to}{\times}_X X \to X \overset{\to}{\times}_Y X $$
is $n-1$-truncated. 

\end{corollary}

Another important property of $n$-connected and $n$-truncated maps of spaces is the compatibility with homotopy pullbacks.
Homotopy pullback preserve $n$-truncated maps for rather formal reasons since $n$-truncated maps belong to the right class of a factorization system. On the other hand homotopy pullback does not preserve $n$-connected maps but decrease connectivity by one,
a statement which is much less formal.
We prove that similar relationships hold in higher category theory after replacing homotopy pullbacks by oriented pullbacks:

\begin{theorem}[\cref{conpull}, \cref{faithpull}]
Let $\n \geq -2$ and 
\[
\xymatrix{
\mA \ar[r]^f\ar[d]^{\alpha} & \mC \ar[d]^{\gamma} \\
\mA' \ar[r]^{f'} & \mC',
}\qquad
\xymatrix{
\mB \ar[r]^g \ar[d]^{\beta} & \mC \ar[d]^{\gamma} \\
\mB' \ar[r]^{g'} & \mC'
}
\]
commutative squares of $\infty$-categories.

\begin{enumerate}[\normalfont(1)]\setlength{\itemsep}{-2pt}
    
\item If $\alpha, \beta, \gamma$ are $n$-truncated, the induced functor $$\mA \overset{\to}{\times}_\mC \mB \to \mA' \overset{\to}{\times}_{\mC'} \mB'$$
is $n$-truncated.

\item If $\alpha, \beta$ are $n$-connected and $\gamma$ is $n+1$-connected, the induced functor
$$ \mA \overset{\to}{\times}_\mC \mB \to \mA' \overset{\to}{\times}_{\mC'} \mB' $$
is $\n$-connected.

\end{enumerate}
\end{theorem}


We extend the concept of homotopy groups to $\infty$-categories.
Here is it crucial to consider not only endomorphisms of an object but all morphisms between two given objects, which is not a group anymore but a poset.
The main fundamental feature of homotopy groups, which makes homotopy groups computationally useful, is that the homotopy groups of a fibration are related to the homotopy groups of the base space and total space via a long exact sequence on homotopy groups.

We extend the long exact sequence on homotopy groups to higher category theory:
\begin{theorem}[\cref{longexact}]
Let $\phi: \mC \to \mD$ be a functor of $\infty$-categories, ${Z}:= (X_\bi,Y_\bi)_{\bi \geq 0} $ an oriented base point of $\mC$.
Let $\mF$ be the oriented right fiber of $\phi$ over $Y_0$.

We consider the oriented exact sequence of partially ordered sets of \cref{remseq}:
$$ ...\to \pi_2(\mD, \phi{Z}) \to \pi_1(\mF, \bar{{Z}}) \to \pi_1(\mC, {Z}) \to \pi_1(\mD,\phi{Z}) \to \pi_0(\mF, \bar{{Z}}) \to \pi_0(\mC, {Z}) \to \pi_0(\mD,\phi{Z}).$$

Let $\n \geq 0.$

\begin{enumerate}[\normalfont(1)]\setlength{\itemsep}{-2pt} 

\item An object of $\pi_\n(\mC, {Z})$ belongs to the image of the map $\pi_{\n}(\mF, \bar{{Z}}) \to \pi_\n(\mC, {Z})$
if and only if it belongs to the oriented right fiber of the map $\pi_\n(\mC, {Z}) \to \pi_\n(\mD, \phi{Z})$.

\vspace{1mm}
\item An object of $\pi_{\n}(\mF, \bar{{Z}})$ 
belongs to the image of the map $\pi_{\n+1}(\mD,\phi{Z}) \to \pi_{\n}(\mF, \bar{{Z}}) $
if and only if its image in $\pi_\n(\mC, {Z})$ is 
the image of $X_{\n+1}$ in $\tau_{\leq 1}(\Mor^n_\mC({Z}))$.

\vspace{1mm}
\item An object of $\pi_{\n+1}(\mD, \phi{Z})$ 
belongs to the image of the map $\pi_{\n+1}(\mC, {Z}) \to \pi_{\n+1}(\mD,\phi{Z}) $
if and only if its image in $\pi_{\n}(\mF, \bar{{Z}})$ is 
the image of $X_{\n+1}$ in $\tau_{\leq 1}(\Mor^n_\mF(\bar{{Z}}))$.

\end{enumerate}

\end{theorem}

Although homotopy posets do not detect equivalences between arbitrary $\infty$-categories, we prove that homotopy posets do detect equivalences between an important class of $\infty$-categories,
which contains all Steiner $\infty$-categories and so all $\infty$-categories of combinatorial shapes that generate the class of all
$\infty$-categories. We prove the following Whitehead theorem for directed $\infty$-categories:

\begin{theorem}[\cref{Whitehead}]
Let $0 \leq n \leq \infty.$
A functor $X \to Y$ of directed $\infty$-categories 
is an $n$-equivalence if and only if
for every $0 \leq m \leq n $ and $m-1$-dimensional oriented base point $ {Z}$ of $X$ the induced map of partially ordered sets $$ \pi_m(X, {Z}) \to \pi_m(Y, \phi {Z})$$ is an isomorphism.
 \end{theorem}

\begin{theorem}[\cref{fundamental}]
The fundamental poset $\pi_1(X,(A,B))$ of a Steiner $\infty$-category $X$ based at $(A,B):\partial\bD^1\to X$ is canonically isomorphic to the set of sequences of atomic morphisms $ A \to T_1 \to ... \to T_m \to B $ for $ m \geq 0 $ eqipped with a natural partial order.

\end{theorem}

\begin{corollary}[\cref{orientalcube}]
Let $n \geq 0.$
There are canonical equivalences

\begin{enumerate}[\normalfont(1)]\setlength{\itemsep}{-2pt}
\item  $$ \pi_0(\bDelta^n)= [n],$$
\item $$ \pi_1(\bDelta^n, (i,j))= (\bD^1)^{\times j-i-1} \cong \{ A \subset \{i+1,...,j-1 \}\}$$ for every
$0 \leq i < j \leq n.$

\item  $$ \pi_0(\cube^n)= (\bD^1)^{\times n},$$
\item $$ \pi_1(\cube^n, (\delta, \epsilon))= S^{|\epsilon-\delta|}$$
for every $\delta, \epsilon \in (\bD^1)^{\times n}.$
\end{enumerate}
    
\end{corollary}

\begin{example}(\cref{thetaposet})
Let $n,k \geq 1$ and $X_1,..., X_n$ be $\infty$-categories. Let ${Z}$ be an oriented base point of dimension $k-1$ of $S(X_1)\vee ... \vee  S(X_n),$ corresponding to a pair of objects $0 \leq i \leq j \leq n $ and an oriented base point $ {Z}'$ of dimension $k-1$ of $\Mor_{S(X_1)\vee ... \vee S(X_n)}(i,j) \simeq X_{i+1} \times ... \times X_{j}.$
For $\bi \leq \ell \leq \bj$ let ${Z}'_\ell$ be the image of the oriented base point ${Z}'$ in $X_\ell.$
Then $$ \pi_k(S(X_1)\vee ... \vee S(X_n),{Z}) = \pi_{k-1}(X_{i+1}, {Z}'_{i+1}) \times ... \times \pi_{k-1}(X_j, {Z}'_j). $$


\end{example}


An important feature of a space is that it can be approximated by a cellular complex. This guarantees that any space can be filtered by skeleta, which inductively arise by attaching cells along their boundaries. We prove that this also holds for $\infty$-categories.
We prove that every $\infty$-category can be filtered by skeleta, which inductively arise by attaching cells along their boundaries:

\begin{theorem}[\cref{skeleta}, \cref{skeletaconnect}]
Every $\infty$-category $X$ is the colimit of the following filtration by skeleta:
$$ \sk_0 X \to ... \to \sk_{n}X\to\sk_{n+1} \to ..., $$
where the $n$-th functor is $n$-connective.
For every $n\geq 0$ the canonical commutative square of $\infty$-categories
\[
\xymatrix{
\underset{\Map^\nd_{\infty\Cat}(\tau_{n-1}\bDelta^n,X)}{\coprod}\partial\bD^n\ar[r] \coprod \underset{\Map_{\infty\Cat}^\nd(\bDelta^n,X)\setminus \Map^\nd_{\infty\Cat}(\tau_{n-1}\bDelta^n,X)}{\coprod}\partial\bD^n\ar[r] \ar[d] & \sk_{n-1}(X) \ar[d]\\
\underset{\Map^\nd_{\infty\Cat}(\tau_{n-1}\bDelta^n,X)}{\coprod}\bD^{n-1} \coprod \underset{\Map_{\infty\Cat}(\bDelta^n,X) \setminus \Map^\nd_{\infty\Cat}(\tau_{n-1}\bDelta^n,X)}{\coprod}\bD^n \ar[r] & \sk_n(X)
}
\]
is a pushout square.
    
\end{theorem}

\begin{corollary}[\cref{wedgesphere}]
Let $n\geq 0$ and let $X$ be an $\infty$-category.
The cofiber of the canonical functor $$\sk_{n-1}X\to\sk_n X$$ is a wedge of $n$-dimensional categorical and homotopical spheres.
Specifically, each nondegenerate oriented $n$-simplex of $X$ contributes an $n$-dimensional sphere which is homotopical if it is a thin $n$-simplex (in the sense of Street and Verity et. al.) and categorical otherwise.
\end{corollary}



We apply the skeletal filtration of an $\infty$-category of \cref{skeleta} to develop obstruction theory of $\infty$-categories. We compute for every $n \geq 0$ the $0$-th poset of components of the $\infty$-category of extensions of a functor $\sk_n(X) \to Y$ to the $n+1$-th skeleton:

\begin{corollary}[\cref{obstruction}]
Let $X, Y$ be $\infty$-categories, $n \geq 0$
and $F:\sk_n(X) \to Y$ a functor.
There is a canonical equivalence 
$$ \tau_{\leq 0}(\Fun(\sk_{n+1}(X), Y) \times_{\Fun(\sk_{n}(X), Y)} \{ F \}) \simeq $$$$ \underset{\alpha \in \Map^\nd_{\infty\Cat}(\tau_{n-1}\bDelta^n,X)}{\prod} \pi_n(Y,F \circ \alpha_{| \partial\bD^n}) \times \underset{\alpha \in \Map_{\infty\Cat}^\nd(\bDelta^n,X)\setminus \Map^\nd_{\infty\Cat}(\tau_{n-1}\bDelta^n,X)}{\coprod} \pi'_n(Y,F \circ \alpha_{| \partial\bD^n}). $$
    
\end{corollary}


\subsection{Relation to other work}

Higher category has been developed extensively in the past few decades.
The first major development was the synthesis of homotopy theory and category theory. See for instance \cite{bergner2007three}, \cite{MR420609}, \cite{joyal2002quasi}, \cite{lurie.HTT}, \cite{rezk2001model}, \cite{simpson1997closed}, \cite{toen2005vers} for some foundational works in this area.

Our theory of truncation and connectivity of higher categories
crucially relies on the theory of oriented categories, categories enriched in the Gray tensor product of higher categories. The Gray tensor product is a refinement of the cartesian product which plays an analogous role in higher category theory and goes back to the work of Gray \cite{GRAY197663}. It was lifted to weak higher category theory by Campion \cite{campion2022cubesdenseinftyinftycategories}.

The combinatorics of homotopy posets of Steiner $\infty$-categories of combinatorial origin, like oriented cubes, oriented simplices
and polytopes, deeply connects to work in representation theory 
of Kapranov-Voevodsky \cite{kapranov1991combinatorial} and Manin-Schechtman \cite{manin1989arrangements}.

Our theory of skeletal filtations of higher categories
makes use of the complicial model of higher categories.
Street and Roberts suggested complicial sets as a model for strict higher categories, and a proof was later carried out by Verity
\cite{Verity2} who systematically developed the model of complicial sets \cite{VERITY1}. Loubaton adapted the approach of Verity to weak higher categories \cite{loubaton2024complicialmodelinftyomegacategories} and has written extensively about foundations from this perspective \cite{loubaton2024categorical}.

The gradual adaptation of homotopy theory in the higher categorical context has led to the development of categorical versions of spectra, which make essential use of the Gray enrichment.
Versions of the theory of categorical spectra have been studied in \cite{heine2025categorification}, \cite{Masuda}, \cite{stefanich2021higher}.

The Baez--Dolan cobordism hypothesis builds a bridge between topological quantum field theory and $n$-category theory.
Lurie's sketch proof \cite{Cob} of the cobordism hypothesis suggests that higher category theory is essential in the study of field theories and categorical symmetries.
Other work in this direction, and more generally applications of higher category theory to mathematical physics and representation theory, have been carried out in \cite{baez1995higher}, \cite{baez2011prehistory}, \cite{ferrer2024dagger}, \cite{gaiotto2019condensations}, \cite{liu2024braided}, and many others.




\subsection{Notation and terminology}

We fix a hierarchy of set-theoretic universes whose objects we call small, large, very large, etc.
We call a space (equivalently, $\infty$-groupoid) $X$ small, large, etc. if for any choice of basepoint and natural number $n$ its homotopy sets $\pi_n X$ are small, large, etc.
We call an $\infty$-category small, large, etc. if its maximal subspace and all its mapping spaces are.

We refer to (not necessarily univalent) weak $(\infty,\n)$-categories for $0 \leq \n \leq \infty$ simply as $\n$-categories, and we refer to (not necessarily univalent) weak $(\n,\n)$-categories as $(\n,\n)$-categories.
In particular, we refer to (not necessarily univalent) $(\infty,1)$-categories as 1-categories, or simply categories.
We will sometimes want to work strictly, which can be viewed as a basechange along the colimit-preserving symmetric monoidal functor $\mS\to\Set$.


\begin{notation}
We will make use of the following notation and terminology when discussing categories, in the sense of categories enriched in the monoidal category of $\infty$-groupoids under the cartesian product.
\begin{enumerate}[\normalfont(1)]\setlength{\itemsep}{-2pt}
\item We write $\mS$ for the category of small spaces, by which we mean $\infty$-groupoids, homotopy types, or anima, and $\Set$ for the category of small sets.
\item We write $\infty\Cat$ and $\infty\CAT$ for the (very) large categories of small and large $\infty$-categories, respectively.
\item
We write $\infty\Cat^\univ$ and $\infty\CAT^\univ$ for the (very) large categories of small and large univalent $\infty$-categories, respectively.
In general, throughout the paper, we will always assume that our $\infty$-categories are univalent, unless explicitly mentioned otherwise.

\item $\Delta$ for (a skeleton of) the category of finite, non-empty, partially ordered sets and order preserving maps, whose objects we denote by $[\n] = \{0 < ... < \n\}$ for $\n \geq 0$.
\item $\Map_{\mC}(A,B)$ for the space of maps (equivalently, $1$-morphisms) from $A$ to $B$ in $\mC$, for any category $\mC$ containing an ordered pair of objects $(A,B)\in\mC$.
\item
We often call a fully faithful functor $\mC \to \mD$ an embedding.
We call a functor $\mC \to \mD$ an inclusion if it induces an embedding on maximal subspaces and on all mapping spaces. The latter are exactly the monomorphisms in $\Cat$. Equivalently, a functor is an inclusion if for any category $\mB$ the induced map
$\Map_\Cat(\mB,\mC) \to \Map_\Cat(\mB,\mD)$ is an embedding.
\item Given a diagram $X\to Z\leftarrow Y$ in a category $\mC$, we write $X\underset{Z}{\prod} Y$ or $X\underset{Z}{\times} Y$ for the pullback, and given a diagram $X\leftarrow W\to Y$ in a category $\mC$, we write $X\underset{W}{\coprod} Y$ or $X\underset{Z}{+} Y$ for the pushout.
\item If $\mC$ and $\mD$ are categories and $\mC\to\mD$ is a left adjoint functor with right adjoint $\mD\to\mC$, we often write $\mC\rightleftarrows\mD$ for this adjunction, where the left adjoint is understood to be the top functor going from left to right.

\item $\emptyset $ for the initial object and $*$ for the final object
in any category.

\item Let $\mC$ be a category that admits a final object.
We write $\mC_*\simeq\mC_{\ast/}\subset\Fun([1],\mC)$ for the category of pointed objects in $\mC$, i.e. the full subcategory of $\Fun([1],\mC)$ consisting of those arrows in $\mC$ for which the source is a final object.

\item We write $\PrL$ and $\PrR$ for the subcategories of $\CAT$ spanned by the presentable categories and the left and right adjoint functors, respectively.
Recall that there is a canonical equivalence $\PrL \simeq (\PrR)^\op$
sending left to right adjoints.
\item We write $\otimes$ for the symmetric monoidal structure on $\PrL$ and $\PrR$.
More precisely, by \cite{lurie.higheralgebra}, $\PrL$ carries a closed symmetric monoidal structure such that the subcategory inclusion $\PrL \subset \CAT$ is a lax symmetric monoidal inclusion with respect to the the cartesian structure on $\CAT$.
\end{enumerate}
\end{notation}

\begin{notation}

Other authors have also considered the connected-truncated factorization systems we study in this paper.
Specifically, Liu--Mazel-Gee--Reuter--Stroppel--Wedrich, Loubaton, Ozornova--Rovelli--Walde write $n$-surjective for what we call $n$-connective, and call $n$-faithful.
Moreover Ozornova--Rovelli--Walde refer to the limit of the tower of right adjoints of the canonical embeddings from the category of $n$-categories into the category of $n+1$-categories as the category of right $\infty$-categories, which we call the category of univalent $\infty$-categories. 
Analogously, they refer to the limit of the tower of left adjoints of the canonical embeddings from the category of $n$-categories into the category of $n+1$-categories
as the category of left $\infty$-categories, which we identify with the category of Postnikov complete $\infty$-categories.

\end{notation}

\subsection*{Acknowledgements}
We thank Ben Antieau, Tim Campion, Fan Huang, Felix Loubaton, Naruki Masuda, Viktoriya Ozornova, Emily Riehl, Martina Rovelli, Markus Spitzweck, Dominic Verity, and Tashi Walde for interesting conversations related to the subject of this paper.
We thank the MPIM for their hospitality while much of this work was carried out. The first author acknowledges the support of the Simons Foundation.









\section{\mbox{Enriched factorization systems}}


\subsection{Enriched categories}

We first recall the notion of homotopy coherent enrichment, as defined and studied in, for instance, \cite{MR3345192}, \cite{heine2024higher}, \cite{heine2025equivalence}, \cite{HINICH2020107129}.
For every presentably monoidal category $\mV$ there is a presentable 2-category $${_{\mV}\Cat}$$ of $\mV$-enriched categories and $\mV$-enriched functors and a forgetful functor $$ \iota:{_{\mV}\Cat} \to \Cat$$ to the presentable 2-category $\Cat$ of categories,
which is an equivalence for $\mV= \mS$ the category of homotopy types. 

 


	


	



	

	
	

\begin{notation}Let $\mV$ be a presentably monoidal category.
A $\mV$-enriched category $\mC$ has an underlying category $\iota(\mC)$, for every objects $X,Y \in \iota(\mC)$ a morphism object 
$$\Mor_\mC(X, Y) \in \mV$$
and for every objects $X,Y,Z \in \iota(\mC)$ a composition morphism in $\mV:$
$$\Mor_\mC(Y, Z) \ot \Mor_\mC(X, Y) \to \Mor_\mC(X, Z).$$

We write $\X \in \mC$ for $\X \in \iota(\mC)$ and usually notationally identify $\mC$ with $\iota(\mC).$





	
\end{notation}



For every enriched category there is an opposite one:

\begin{notation}
There is an involution $$(-)^\circ: {_{\mV}\Cat} \simeq {_{\mV^\rev}}\Cat$$ forming the opposite enriched category.
For every $\mC \in {_{\mV}\Cat}$ and $X,Y \in \mC$
there are canonical equivalences
$ \iota(\mC^\circ) \simeq \iota(\mC)^\op$
and $$ \Mor_{\mC^\circ}(X,Y) \simeq \Mor_{\mC}(Y,X).$$
\end{notation}

There is a close relationship between enriched categories
and tensored and cotensored categories:

\begin{definition}Let $\mV$ be a presentably monoidal category, $\mC$ a $\mV$-enriched category and $X \in \mC, V \in \mV.$	 
	
\begin{enumerate}[\normalfont(1)]\setlength{\itemsep}{-2pt}
	
\item The tensor of $V$ and $X$ in $\mC$ is the object $V \ot X \in \mC $ such that there is a morphism
$V \to \Mor_\mC(X, V \ot X) $ in $\mV$ that induces for every $Y \in \mC$ an equivalence
$$ \Mor_\mC(V \ot X,Y) \to \Mor_\mV(V, \Mor_\mC(X,Y)). $$ 

\item The cotensor of $V$ and $X$ in $\mC$ is the object ${^V X} \in \mC $ that is the tensor of $V $ and $X$ in the opposite $\mV^\rev$-enriched category $\mC^\circ.$

\end{enumerate}
	
\end{definition}











Since the category of enriched categories forms a 2-category, there is a natural intrinsic notion of adjunction between enriched categories:

\begin{definition}Let $\mV$ be a presentably monoidal category.
A $\mV$-enriched functor $\mC \to \mD$ admits a left (right) adjoint if
it admits a left (right) adjoint in the 2-category $\mV \mathrm{-}\Cat.$

\end{definition}

The following is \cite[Remark 2.55.]{heine2024bienriched}:

\begin{proposition}

Let $\mV$ be a presentably monoidal category.

\begin{enumerate}[\normalfont(1)]\setlength{\itemsep}{-2pt}

\item A $\mV$-enriched functor $\phi: \mC \to \mD$ admits a $\mV$-enriched right adjoint if and only if for every $\Y \in \mD$ the $\mV^\rev$-enriched functor
$\Mor_\mD(\phi(-),\X): \mC^\circ \to \mV$ is representable.

\vspace{1mm}

\item A $\mV$-enriched functor $\phi: \mC \to \mD$ admits a $\mV$-enriched left adjoint if and only if the opposite $\mV^\rev$-enriched functor $\phi^\circ: \mC^\circ \to \mD^\circ$ admits a $\mV^\rev$-enriched right adjoint. By (1) this holds if and only if for every $\Y \in \mD$ the $\mV$-enriched functor
$\Mor_\mD(\Y,\phi(-)): \mC \to \mV$ is representable.

\end{enumerate}
\end{proposition}

The following is \cite[Lemma 2.77.]{heine2024bienriched}:

\begin{proposition}\label{adj}
Let $\mV$ be a presentably monoidal category.

\begin{enumerate}[\normalfont(1)]\setlength{\itemsep}{-2pt}
 
\item A $\mV$-enriched functor $\mC \to \mD$ admits a right adjoint if and only if it preserves tensors and the underlying functor admits a right adjoint.

\item A $\mV$-enriched functor $\mC \to \mD$ admits a left adjoint if and only if it preserves cotensors and the underlying functor admits a left adjoint.

\end{enumerate}

\end{proposition}

The concept of presentability has a natural analogue in enriched category theory:


\begin{notation}
Let $\PrL \subset \widehat{\Cat}$ be the subcategory of presentable categories and left adjoint functors.

\end{notation}

\begin{remark}
The category $\PrL$ carries a canonical closed symmetric monoidal structure such that the inclusion $$\PrL \subset \widehat{\Cat}$$ is lax symmetric monoidal, where $\widehat{\Cat}$ carries the cartesian structure \cite[Proposition 4.8.1.15.]{lurie.higheralgebra}.

\end{remark}

\begin{definition}\label{present} Let $\mV$ be a presentably monoidal category.
A $\mV$-enriched category $\mC$ is presentable if it admits tensors
and the underlying category $\iota(\mC)$ is presentable.
\end{definition}




\begin{notation}Let $\mV$ be a presentably monoidal category.
Let $$_{\mV}\PrL \subset {_{\mV}\Cat}$$ 
be the subcategory of presentable $\mV$-enriched categories and left adjoint
$\mV$-enriched functors.

\end{notation}

The following is \cite[Theorem 1.2.]{HEINE2023108941}:

\begin{theorem} Let $\mV$ be a presentably monoidal category.
There is a canonical equivalence $${_\mV\Mod(\PrL)} \simeq {_\mV \PrL}.$$

\end{theorem}





Next we introduce enriched slice categories.
For the next notation we use the cotensors in the presentable 2-category ${_\mV \Cat}$.

\begin{notation}\label{slices} Let $\mV$ be a presentably monoidal category whose tensor unit is final.
Let $\mC$ be a $\mV$-enriched category and $X \in \mC.$

Let $\mC_{/X} $ be the fiber over $X$ of the $\mV$-enriched functor $\mC^{\bD^1} \to \mC^{\{1\}}$ evaluating at $1$
in the presentable 2-category ${_\mV \Cat}.$
There is a forgetful $\mV$-enriched functor $\mC_{X/} \to \mC^{\bD^1} \to \mC^{\{0\}}$ evaluating at $0.$	
	
\end{notation}

The following is \cite[Remark 2.4.25.]{gepner2025oriented}:

\begin{remark}\label{bien} Let $\mV$ be a presentably monoidal category
whose tensor unit is final.
Let $\mC$ be a $\mV$-enriched category and $X \in \mC.$
Let $\alpha: X \to Y, \beta: X \to Z$ be morphisms in $\mC.$
There is a canonical equivalence 	
$$ \Mor_{\mC_{X/}}(Y,Z) \simeq \{\beta\} \times_{\Mor_{\mC}(X,Z)} \Mor_{\mC}(Y,Z).$$

\end{remark}

The following is \cite[Lemma 2.4.28.]{gepner2025oriented}:

\begin{lemma}\label{left adjoint slice} Let $\mV$ be a presentably monoidal category whose tensor unit is final.
Let $\mC$ be a $\mV$-enriched category and $X \in \mC.$
Every morphism $X \to Y$ in $\mC$ gives rise to a $\mV$-enriched functor
$ \mC_{Y/}  \to \mC_{X/} $ over $\mC.$
If $\mC$ admits conical pushouts, the latter admits a $\mV$-enriched left adjoint.

\end{lemma}


We will use the following terminology to describe compatible enrichments in two monoidal categories:

\begin{definition}\label{bienr}
Let $\mV, \mW$ be presentably monoidal categories.	 

\begin{enumerate}[\normalfont(1)]\setlength{\itemsep}{-2pt}
\item A left $\mV$-enriched category is a $\mV$-enriched category.

\item A left $\mV$-enriched functor is a $\mV$-enriched functor

\item A right $\mV$-enriched category is a $\mV^\rev$-enriched category.

\item A right $\mV$-enriched functor is a $\mV^\rev$-enriched functor.

\item A $\mV,\mW$-bienriched category is a $\mV \ot \mW^\rev$-enriched category.

\item A $\mV,\mW$-enriched functor is a $\mV \ot \mW^\rev$-enriched functor.

\end{enumerate}

\end{definition}




\begin{definition}Let $\mV, \mW$ be presentably monoidal categories, $\mC$ a 
$\mV,\mW$-enriched category and $X \in \mC, V \in \mV, W \in \mW.$	 
\begin{enumerate}[\normalfont(1)]\setlength{\itemsep}{-2pt}
\item The left (co)tensor of $V$ and $X$ in $\mC$ is the (co)tensor of
$V \ot \tu_\mW \in \mV \ot \mW^\rev$ and $X$ in $\mC.$

\item The right (co)tensor of $W$ and $X$ in $\mC$ is the (co)tensor of
$\tu_\mV \ot W \in \mV \ot \mW^\rev$ and $X$ in $\mC.$



\end{enumerate}

\end{definition}




Due to the next proposition we can apply \cref{slices} also to bienriched categories.
The following is \cite[Lemma 2.4.22.]{gepner2025oriented}:

\begin{proposition}
Let $\mV, \mW$ be presentably monoidal categories whose tensor unit is final.
Then the tensor unit of $\mV \ot \mW$ is final.

\end{proposition}

Next we introduce the enriched category of enriched presheaves \cite{heine2025equivalence}.

The next theorem follows from \cite[Proposition 4.11, Theorem 4.86]{heine2024bienriched}:

\begin{theorem}\label{psinho} Let $\mV, \mW$ be presentably monoidal categories and $\mN$ a $(\mV, \mW)$-bienriched category. 

\begin{enumerate}[\normalfont(1)]\setlength{\itemsep}{-2pt}
\item Let $\mM$ be a small left $\mV$-enriched category. The category $\Fun_{\mV}(\mM, \mN)$
refines to a right $\mW$-enriched category characterized by an 
equivalence
$$ \Fun_\mW(\mO,\Fun_{\mV}(\mM, \mN)) \to \Fun_{\mV,\mW}(\mM \boxtimes \mO,\mN)$$
natural in any right $\mW$-enriched category $\mO$.

\item Let $\mO$ be a small right $\mW$-enriched category. 
The category $\Fun_{\mW}(\mO, \mN)$ refines to a left $\mV$-enriched category characterized by an equivalence
$$ _\mV\Fun(\mM,\Fun_{\mW}(\mO, \mN)) \to \Fun_{\mV,\mW}(\mM \boxtimes \mO,\mN) $$
natural in any left $\mV$-enriched category $\mM$.

\item If $\mN$ is a category presentable bitensored over $\mV, \mW$, then $\Fun_{\mV}(\mM, {\mN}) $ a category presentable right tensored over $\mW$ and $\Fun_{\mW}(\mO, {\mN})$ is a category presentable left tensored over $\mV$.

\end{enumerate}

\end{theorem}
    
\begin{definition}

Let $\mV$ be a presentably monoidal category and $\mC$ a small $\mV$-enriched category.
The presentable left $\mV$-tensored category of $\mV$-enriched presheaves on $\mC$ is
$$ \mP_\mV(\mC):= \Fun_\mV(\mC^\op,\mV). $$
\end{definition}

The next theorem, which follows from \cite[Theorem 3.41, Theorem 4.70]{heine2024bienriched}, is an enriched version of the universal property of the category of presheaves as the free cocompletion under small colimits \cite[Theorem 5.1.5.6]{lurie.HTT}, \cite[Proposition 3.16.]{heine2026local}:

\begin{theorem}
\label{Yonedaext}
Let $\mV$ be a presentably monoidal category, $\mC$ a small $\mV$-enriched category and $\mD$ a presentable left $\mV$-tensored category.

\begin{enumerate}[\normalfont(1)]\setlength{\itemsep}{-2pt}
\item There is a $\mV$-enriched embedding $\iota_\mC : \mC \to \mP_\mV(\mC)$ that sends $X$ to $\L\Mor_\mC(-,X)$ and induces
for every $\mV$-enriched category $\mD$ an equivalence
$$ {_\mV\Fun^\L}(\mP_\mV(\mC),\mD)\to {_\mV\Fun}(\mC,\mD).$$


\item Let $F: \mC \to \mD$ be a $\mV$-enriched functor and
$\bar{F}: \mP_\mV(\mC) \to \mD $ the unique $\mV$-enriched left adjoint extension of $F$.
The $\mV$-enriched right adjoint $\mD \to \mP_\mV(\mC)$ of $\bar{F}$ sends $Y$ to $\L\Mor_\mD(-,Y) \circ F. $

\end{enumerate}

\end{theorem}

The following is the enriched Yoneda-lemma proven in 
\cite[Corollary 4.44]{heine2024bienriched}:

\begin{lemma}\label{enryoneda}
Let $\mV$ be a presentably monoidal category and $\mC$ a small $\mV$- enriched $\infty$-category. For every object $X \in \mC$ and $F \in \mP_\mV(\mC)$ the induced morphism $$\L\Mor_{\mP_\mV(\mC)}(\L\Mor_\mC(-,X),F) \to F(X) $$ is an equivalence.

\end{lemma}

Next we define enriched cocartesian fibrations.

\begin{definition}

Let $\mV$ be a presentably monoidal category
and $\phi: \mC \to \mD$ a $\mV$-enriched functor.
A morphism $X \to Y$ in $\mC$ is $\phi$-cocartesian if for every $Z \in \mC$ the following commutative square in $\mV$ is a pullback square:
$$\begin{xy}
\xymatrix{
\L\Mor_\mC(Y,Z) \ar[d]^{} \ar[r]
& \L\Mor_\mC(X,Z) \ar[d]
\\ 
\L\Mor_\mD(\phi(Y),\phi(Z)) \ar[r] & \L\Mor_\mD(\phi(X),\phi(Z)).}
\end{xy}$$
   
\end{definition}

\begin{definition}

Let $\mV$ be a presentably monoidal category.
A $\mV$-enriched functor $\phi: \mC \to \mD$ is a $\mV$-enriched cocartesian fibration if for every $Y \in \mC$ and morphism $\alpha: \phi(Y)\to Z$ in $\mD$ there is a $\phi$-cocartesian morphism $Y \to X $ in $\mC$ lying over $\alpha.$

    
\end{definition}


\subsection{Factorization systems on enriched categories}


\begin{definition}Let $\mV$ be a presentably monoidal category and $\mC$ a $\mV$-enriched category.
A morphism $f:A\to B$ has the {\em left lifting property} with respect to a morphism $g:C\to D$, or $g:C\to D$ has the {\em right lifting property} with respect to $f:A\to B$, if 
the commutative square
\[
\xymatrix{
\Mor_\mC(B,C)\ar[r]\ar[d] & \Mor_\mC(A,C)\ar[d]\\
\Mor_\mC(B,D)\ar[r] & \Mor_\mC(A,D)
}
\]
is a pullback square in $\mV$.
\end{definition}

\begin{notation}
We write $f\bot\, g$ if $f$ has the left lifting property with respect to $g$.
\end{notation}

\begin{notation}Let $S$ be a class of morphisms of $\mC$.
We write $$^\bot S=\{f\in\mC\,|\,\forall g\in S,f\bot\,g\} $$$$ S^\bot=\{g\in\mC\,|\,\forall f\in S, f\bot\,g\}.$$
\end{notation}

\begin{remark}
If the functor $\Mor_\mV(\ast,-):\mV\to\mS$ is conservative, then either of the above conditions is equivalent to asking that the space $\Map_{\mC_{A//D}}(B,C)$ is contractible.
\end{remark}

\begin{remark}

Let $\mV$ be a presentably monoidal category and $F: \mC \rightleftarrows \mD: G $ a $\mV$-enriched adjunction.
Let $f:A\to B$ be a morphism in $\mC$ and $g: C\to D$ a morphism in $\mD$. Then $F(f) $ has the left lifting property with respect to $g:C\to D$ if and only if $f$ has the left lifting property with respect to $G(g)$.
This holds because the commutative square
\[
\xymatrix{
\Mor_\mC(F(B),C)\ar[r]\ar[d] & \Mor_\mC(F(A),C)\ar[d]\\
\Mor_\mC(F(B),D)\ar[r] & \Mor_\mC(F(A),D)
}
\]
identifies with the commutative square
\[
\xymatrix{
\Mor_\mC(B,G(C))\ar[r]\ar[d] & \Mor_\mC(A,G(C))\ar[d]\\
\Mor_\mC(B,G(D))\ar[r] & \Mor_\mC(A,G(D)).
}
\]
    
\end{remark}

\begin{remark}\label{enrfactor0}

Let $\mV$ be a presentably monoidal category and $\mC$ a $\mV$-   
enriched category that admits cotensors.
A morphism $f:A\to B$ has the left lifting property with respect to a morphism $g:C\to D$
if and only if $f:A\to B$ has the left lifting property in the unenriched sense with respect to the morphism $g^\V:C^V\to D^V$ for every $V \in \mV.$

\end{remark}

\begin{definition}Let $\mV$ be a presentably monoidal category, $\mC$ a $\mV$-enriched category and $C $ an object of $\mC$.
A morphism $f:A\to B$ of $\mC$ is $\mV$-enriched local with respect to $C$, or that $C$ is $\mV$-enriched local with respect to $f$, if the induced morphism $\Mor_\mC(B,C)\to\Mor_\mC(A,C)$ in $\mV$ is an equivalence.
\end{definition}

\begin{lemma}\label{char}Let $\mV$ be a presentably monoidal category and $\mC$ a $\mV$-enriched category.
A morphism $f: X \to Y$ in $\mC$ has the left lifting property with respect to a morphism $g: A \to B$ in $\mC$ if and only if the morphism 
\[
\xymatrix{
X\ar[r]^f\ar[d]_f & Y\ar[d]^= \\
Y\ar[r]^= & Y.
}
\]
in $\Fun(\bD^1,\mC)$ 
is $\mV$-enriched local with respect to $g \in \Fun(\bD^1,\mC).$
Here $\Fun(\bD^1,\mC)$ is the $\mV$-enriched arrow category,
which is the cotensor of the $\mV$-enriched category $\mC$ by the category $\bD^1$.    
\end{lemma}

\begin{proof}

The canonical map
$$ \Mor_{\Fun(\bD^1,\mC)}(\id_\Y,g) \to \Mor_{\Fun(\bD^1,\mC)}(f,g) $$
identifies with the canonical map
$$ \Mor_{\mC}(\Y,A) \to \Mor_{\mC}(\X,A) \times_{\Mor_{\mC}(\X,B)}  \Mor_{\mC}(\Y,B) \simeq \Mor_{\Fun(\bD^1,\mC)}(f,g). $$
\end{proof}

\begin{definition}\label{factorization} Let $\mV$ be a presentably monoidal category.
A $\mV$-enriched factorization system $(\mC_L,\mC_R)$ on a $\mV$-enriched category $\mC$ consists of classes of morphisms $\mC_L$ and $\mC_R$ in (the underlying category of) $\mC$ such that
\begin{enumerate}[\normalfont(1)]\setlength{\itemsep}{-2pt}
\item The classes $\mC_L$ and $\mC_R$ are closed under retracts.
\item Every morphism $A\to C$ in $\mC$ factors as $A\to B\to C$, where $A\to B$ is in $\mC_L$ and $B\to C$ is in $\mC_R$.
\item For all morphisms $f:A\to B$ in $\mC_L$ and $g:C\to D$ in $\mC_R$, we have $f\bot\, g$ (that is, $f$ has the left lifting property with respect to $g$).
\end{enumerate}
\end{definition}

\begin{remark}
We have that $\mC_L=\,^{\bot}\mC_R$ and $\mC_R=\mC_L^\bot$.
\end{remark}

\begin{remark}\label{enrfactor}

Let $\mV$ be a presentably monoidal category, $\mC$ a $\mV$-enriched $\infty$-category and $\mC_L, \mC_R$ classes of morphisms in $\mC.$

If $(\mC_L,\mC_R)$ is a $\mV$-enriched factorization system on $\mC$, then $(\mC_L,\mC_R)$ is a factorization system on $\mC$ in the usual (non-enriched) sense.
If $\mV$ admits cotensors and the right class is closed under cotensors, then by \cref{enrfactor0} the converse holds: the pair $(\mC_L,\mC_R)$ is a $\mV$-enriched factorization system on $\mC$ if $(\mC_L,\mC_R)$ is a factorization system on $\mC$.
    
\end{remark}

\begin{theorem}\label{fact} Let $\mV$ be a presentably monoidal category and $\mC$ a $\mV$-enriched category.
A full $\mV$-enriched subcategory $\mR \subset \Fun(\bD^1,\mC)$ is the right class of a $\mV$-enriched factorization system on $\mC$ if and only if the following conditions hold:

\begin{enumerate}[\normalfont(1)]\setlength{\itemsep}{-2pt}

\item The $\mV$-enriched embedding $\mR \subset \Fun(\bD^1,\mC)$ admits a $\mV$-enriched left adjoint.

\item The unit is sent to an equivalence by evaluation at the target $ \Fun(\bD^1,\mC) \to \mC$.

\item Morphisms in $\mC$ corresponding to objects in $\mR$ are closed under composition.

If these conditions hold, a morphism $X \to Y$ in $\mC$ belongs to the left class if and only if
the morphism 
\[
\xymatrix{
X\ar[r]^f\ar[d]^f & Y\ar[d]^= \\
Y\ar[r]^= & Y.
}
\]
in $\Fun(\bD^1,\mC)$ 
is a local equivalence.

Moreover in this case a morphism in $\Fun(\bD^1,\mC)$ to a local object is a commutative square
\[
\xymatrix{
X\ar[r]^f \ar[d] & Y\ar[d]^g \\
Y\ar[r]^= & Y,
}
\]
where $f$ is in the left class and $g$ is in the right class.

\end{enumerate}
\end{theorem}

\begin{proof}

Let $\mR \subset \Fun(\bD^1,\mC)$ be a full $\mV$-enriched subcategory that is the right class of a factorization system on $\mC$.
The class of morphisms of $\mC$ corresponding to objects in $\mR$ is closed under composition.
Moreover every morphism $f: X \to Y$ in $\mC$ factors as $X \xrightarrow{i} Z \xrightarrow{p} Y$, where $i$ is in the left class and $p \in \mR$.
We obtain a morphism
\[
\xymatrix{
X\ar[r]^i\ar[d]_f & Z\ar[d]^p\\
Y\ar[r]^= & Y.
}
\]
in $\Fun(\bD^1,\mC).$ 
The morphism $i$ determines a morphism
\[
\xymatrix{
X\ar[r]^i\ar[d]^i & Z\ar[d]^=\\
Z\ar[r]^= & Z.
}
\]
in $\Fun(\bD^1,\mC).$ 
We prove that the first morphism is local with respect to objects of $\mR.$
Let $g: A \to B$ be an object of $\mR$.
The induced map
\[
\Mor_{\Fun(\bD^1,\mC)}(p,g) \to \Mor_{\Fun(\bD^1,\mC)}(f,g)
\]
is a morphism over $\Mor_\mC(Y,B).$ So we obtain a commutative square
\[
\xymatrix{
\Mor(p,g)\ar[r]\ar[d] & \Mor(\id_Z,g)\ar[d]\\
\Mor(f,g)\ar[r] & \Mor(i,g)
}
\]
which comes from a commutative square
\[
\xymatrix{
p & \id_Z\ar[l]\\
f\ar[u] & i\ar[l]\ar[u].
}
\]
This square is a pushout since the source and target squares are
\[
\xymatrix{
Z & Z\ar[l]_{\id_Z}\\
X\ar[u]^i & X\ar[l]^{\id_X}\ar[u]_i}\qquad
\xymatrix{
Y & Z\ar[l]_p\\
Y\ar[u]^{\id_Y} & Z\ar[l]^{p}\ar[u]_{\id_Z}}
\]
respectively.
Hence it suffices to see that the induced morphism $$\Mor_{\Fun(\bD^1,\mC)}(\id_Z,g) \to \Mor_{\Fun(\bD^1,\mC)}(i,g) $$ is an equivalence.
This follows from \cref{char} since $g$ has the right lifting property with respect to $i: X \to Z.$
The description of the morphisms of the left class follows from \cref{char}.

We prove that conditions (1), (2), (3) imply that $\mR \subset \Fun(\bD^1,\mC)$ is the right class of a factorization system on $\mC$.

Assume that (1) holds and let $f: X \to Y$ be a morphism in $\mC$. Then by (2) there is a morphism
\begin{equation}\label{morsq}
\xymatrix{
X\ar[r]^i\ar[d]_f & Z\ar[d]^p \\
Y\ar[r]^= & Y.
}
\end{equation}
in $\Fun(\bD^1,\mC),$ where $p \in \mR$, which is a local equivalence
for the localization $\mR \subset \Fun(\bD^1,\mC).$
It remains to see that $i: X \to Z$ has the left lifting property with respect to $p$.
By \cref{char} we have to see that the morphism
\begin{equation}\label{morsq2}
\xymatrix{
X\ar[r]^i\ar[d]^i & Z\ar[d]\\
Z\ar[r]^= & Z.
}
\end{equation}
in $\Fun(\bD^1,\mC)$ is a local equivalence for the localization
$\mR \subset \Fun(\bD^1,\mC).$

The $\mV$-enriched functor $\Fun(\bD^1,\mC) \to \mC$ evaluating at the target is a $\mV$-enriched cocartesian fibration. This guarantess that (1), (2) imply that for every $\X \in \mR$ the embedding $\mR_\X \subset \mC_{/\X}$ admits a left adjoint, the functor $\mC_{/X} \to \mC$ preserves local equivalences and local objects and for every morphism $X \to X'$ in $\mC$ the forgetful functor
$\mC_{/\X} \to \mC_{/\X'}$ preserves local equivalences. Condition (3) implies that for every morphism $X \to X'$ corresponding to an object in $\mR$ the forgetful functor
$\mC_{/\X} \to \mC_{/\X'}$ preserves local objects.
The forgetful functor $\mC_{/\X} \to \mC_{/\X'}$ is conservative and therefore also detects local equivalences if $X \to X' $ corresponds to an object in $\mR$ since it preserves local equivalences and local objects.

We observe that the commutative square (\ref{morsq2}) is a morphism in $\mC_{/Z}$ whose image under the forgetful functor $\mC_{/Z} \to \mC_{/Y}$ is the morphism (\ref{morsq}), which is a local equivalence.
Thus also the commutative square (\ref{morsq2}) is a local equivalence
in $\mC_{/Z}$ and so in $\Fun(\bD^1,\mC).$   
\end{proof}


    


        
    

\begin{cor}\label{corfact} Let $\mV$ be a presentably monoidal category, $\mC$ a $\mV$-enriched category and $(\mL,\mR)$ a $\mV$-enriched factorization system on $\mC$.
\begin{enumerate}[\normalfont(1)]\setlength{\itemsep}{-2pt}
\item If $\mC$ admits a final object, the embedding of the full $\mV$-enriched subcategory $\mR_{/\ast}\subset\mC$ of morphisms $X\to\ast$ which lie in $\mR$ admits a $\mV$-enriched left adjoint.

\item If $\mC$ admits an initial object, the embedding of the full $\mV$-enriched subcategory $\mL_{\emptyset/}\subset\mC$ of morphisms $\emptyset\to X$ which lie in $\mL$ admits a $\mV$-enriched right adjoint.
\end{enumerate}\end{cor}

\begin{proposition}\label{secfac} \label{facfunc}
Let $\mV$ be a presentably monoidal category.

\begin{enumerate}[\normalfont(1)]\setlength{\itemsep}{-2pt}
\item Let $\mC \to \mD$ be a $\mV$-enriched cocartesian fibration whose fibers carry a $\mV$-enriched factorization system and whose fiber transports preserve the left and right class.
The category $\Fun^\mV_\mD(\mD,\mC)$ carries
a factorization system whose left class precisely consists of the morphisms that are sectionwise in the left class, and whose right class precisely consists of the morphisms that are sectionwise in the right class.

\item Let $\mC $ be a small $\mV$-enriched category, $\mD$ a $\mV,\mW$-bienriched category and $(\mL, \mR)$ a $\mV,\mW$-bienriched factorization system on $\mD.$
The right $\mW$-enriched category $\Fun^\mV(\mC,\mD)$ carries a right $\mW$-enriched factorization system whose left class precisely consists of the morphisms that are objectwise a morphism in $\mL$ and whose right class precisely consists of the morphisms that are objectwise a morphism in $\mR$.

\end{enumerate}

    
\end{proposition}

\begin{proof}

(1): We consider the $\mV$-enriched cocartesian fibration $\mC^{\bD^1} := \Fun(\bD^1, \mC) \times_{\Fun(\bD^1, \mD)} \mD\to \mD$
whose fiber over $X \in \mD$ is $\Fun(\bD^1, \mC_X).$
Let $\mR \subset \mC^{\bD^1}$ be the full $\mV$-enriched subcategory of objects that belong to the right class of some fiber.

By \cref{fact} the $\mV$-enriched embedding $\mR \subset \mC^{\bD^1}$ is fiberwise a $\mV$-enriched localization whose local equivalences are preserved by the fiber transports since the fiber transports preserve the left and right class.

By \cite[Remark 2.2.4.12.]{lurie.higheralgebra} this implies that the embedding $\mR \subset \mC^{\bD^1}$
admits a $\mV$-enriched left adjoint whose local equivalences are inverted by the functor $\mC \to \mD$, and so are inverted by evaluation at the target $\mC^{\bD^1} \to \mC$ since it holds fiberwise by description of local equivalences.
In particular, the embedding $$ \Fun^\mV_\mD(\mD, \mR) \subset \Fun^\mV_\mD(\mD,\mC^{\bD^1}) \simeq \Fun(\bD^1,\Fun^\mV_\mD(\mD,\mC)) $$
admits a left adjoint whose local objects and local morphisms are sectionwise.
So the local objects are the morphisms of sections of $\mC \to \mD$ that are sectionwise a local object in the fiber, i.e. a morphism in the right class. In particular, $\Fun_\mD^\mV(\mD, \mR)$ is closed under composition as a class of morphisms of $\Fun^\mV_\mD(\mD,\mC)$.
\cref{fact} implies that the full subcategory $ \Fun_\mD^\mV(\mD, \mR) \subset \Fun(\bD^1,\Fun^\mV_\mD(\mD,\mC))$ is the right class of a factorization system.

By \cref{fact} a morphism $X \to Y$ in $\Fun^\mV_\mD(\mD,\mC)$ belongs to the left class if and only if the morphism 
\[
\xymatrix{
X\ar[r]^f\ar[d]^f & Y\ar[d]^= \\
Y\ar[r]^= & Y.
}
\]
in $\Fun(\bD^1,\Fun^\mV_\mD(\mD,\mC))$ 
is a local equivalence, which is precisely a sectionwise local equivalence. This means that for every $D \in \mD$
the morphism 
\[
\xymatrix{
X(D)\ar[r]^{f_D} \ar[d]^{f_D} & Y(D) \ar[d]^= \\
Y(D) \ar[r]^= & Y(D).
}
\]
in $\Fun(\bD^1,\mC_D)$ is a local equivalence.
By \cref{fact} this is equivalent to say that $f_D$ is in the left class for every $D \in \mD.$

The proof of (2) is similar. (1) already implies that category $\Fun^\mV(\mC,\mD)$ carries a factorization system whose left class precisely consists of the morphisms that are objectwise a morphism in $\mL$ and whose right class precisely consists of the morphisms that are objectwise a morphism in $\mR$.

To prove (2) we use that there is a canonical right $\mW$-enriched equivalence 
$$ \Fun^\mV(\mD, \mR) \subset \Fun^\mV(\mD,\mC^{\bD^1}) \simeq \Fun(\bD^1,\Fun^\mV(\mD,\mC))$$ and perform the same but simpler arguments.
\end{proof}





\cref{secfac} gives the following:

\begin{corollary}\label{algfaccormon}

Let $\mV$ be a monoidal category and $(\mL, \mR)$ a factorization system on $\mV$ such that $\mL$ and $\mR$ are preserved by the tensor product.
Then $\Alg(\mV)$ inherits a factorization system and the forgetful functor $\Alg(\mV) \to \mV $ preserves the left and right class.
    
\end{corollary}

\subsection{Accessible enriched factorization systems}
We generalize the factorization systems on presentable categories studied in \cite{lurie.HTT} to the $\mV$-enriched context.

\begin{lemma}\label{accfac}Let $\mV$ be a presentably monoidal category and $\mC$ a presentable $\mV$-enriched category. Suppose given a small full subcategory
\[
\{i_\alpha:S_\alpha\to T_\alpha\}\subset\Fun(\bD^1,\mC).
\]
Let $\mR\subset\Fun(\bD^1,\mC)$ be the full subcategory of morphisms $p:X\to Y$ in $\mC$ such that each commutative square
\[
\xymatrix{
S_\alpha\ar[r]\ar[d] & X\ar[d]^p \\
T_\alpha\ar[r] & Y
}
\]
has a unique filler.
Then $\mR$ is the right class of a $\mV$-enriched factorization system on $\mC$.
\end{lemma}

\begin{proof}

By \cref{char} the full subcategory $\mR \subset\Fun(\bD^1,\mC)$ precisely consists of those morphisms which are $\mV$-enriched local with respect to the set of morphisms
\[
\xymatrix{
S_\alpha\ar[r]\ar[d] & T_\alpha\ar[d]\\
T_\alpha\ar[r] & T_\alpha.
}
\]
Since $\Fun(\bD^1,\mC)$ is a presentable $\mV$-enriched category, $\mR \subset\Fun(\bD^1,\mC)$ is an accessible $\mV$-enriched localization.
Since the generating local equivalences are inverted by evaluation at the target $\Fun(\bD^1,\mC) \to \mC$, every local equivalence is inverted by evaluation at the target $\Fun(\bD^1,\mC) \to \mC$.
The set of morphisms corresponding to objects of $\mR \subset\Fun(\bD^1,\mC)$ is closed under composition because
$\mR$ is defined via a right lifting property.
Hence the statement follows from \cref{fact}.
\end{proof}

\begin{definition}Let $\mV$ be a presentably monoidal category and $\mC$ a presentable $\mV$-enriched category.
A $\mV$-enriched factorization system $(\mL,\mR)$ on $\mC$ is {\em accessible} if it arises from the construction of \cref{accfac}.
That is, if there exists a small set $S$ of morphisms in $\mC$ such that the right class precisely consists of those morphisms in $\mC$ having the right lifting property with respect to all morphisms in $S$.
\end{definition}

\begin{corollary}\label{corsecfac}Let $\mV, \mW$ be presentably monoidal categories, $\mC $ a small $\mV$-enriched category, $\mD $ a presentable $\mV, \mW$-bienriched category and $(\mL, \mR)$ an accessible $\mV, \mW$-bienriched factorization system on $\mD.$
The right $\mW$-enriched category
$\Fun^\mV(\mC,\mD)$ carries an accessible right $\mW$-enriched factorization system whose left class precisely consists of the morphisms that are objectwise a morphism in $\mL$ and whose right class precisely consists of the morphisms that are objectwise a morphism in $\mR$.

    
\end{corollary}

\begin{proof}
Since $(\mL,\mR)$ is an accessible $\mV, \mW$-bienriched factorization system on $\mD$, there is a set $S:=\{i_\alpha:S_\alpha\to T_\alpha\}$ of morphisms of $\mC$ such that a morphism of $\mD$ belongs to $\mR$
if and only if it has the right lifting property with respect to $S.$

By \cite[Corollary 4.44.]{heine2024bienriched} for every $X \in \mC, Y \in \mD$ the
forgetful right $\mW$-enriched functor $\Fun^\mV(\mC,\mD) \to \mD$
evaluating at $X$ admits a right $\mW$-enriched left adjoint sending $Z \in \mD $ to $\Mor_\mC(X,-) \ot Z \in \Fun^\mV(\mC,\mD). $

Since the $\mV, \mW$-bienriched category $\mD$ is presentable, the right $\mW$-enriched category $\Fun^\mV(\mC,\mD) $ is presentable by \cref{psinho}. Thus by \cref{accfac} there is a right $\mW$-enriched factorization system on $\Fun^\mV(\mC,\mD) $ whose right class precisely consists of the morphisms that have the right lifting property with respect to the morphisms of the set
$$ T:=\{ \Mor_\mC(X,-) \ot \alpha \mid \alpha \in S, X \in \mC\}. $$
We observe that a morphism has the right lifting property with respect to the morphisms of $T$ if and only if it is objectwise in the right class.

The right class of the right $\mW$-enriched factorization system on $\Fun^\mV(\mC,\mD) $ coincides with the right class of the right $\mW$-enriched factorization system on $\Fun^\mV(\mC,\mD) $ of \cref{facfunc}. Hence also the left classes coincide.
Therefore the left class is objectwise by \cref{facfunc}.
\end{proof}

\begin{prop}\label{factlocal} Let $\mV$ be a presentably monoidal category, $\mC, \mC'$ presentable $\mV$-enriched categories and $F: \mC \rightleftarrows \mC': G$ a $\mV$-enriched adjunction. Let $(\mL,\mR)$ be an accessible $\mV$-enriched factorization system on $\mC$.

\begin{enumerate}[\normalfont(1)]\setlength{\itemsep}{-2pt}
\item Then $\mC'$ carries an accessible $\mV$-enriched factorization system
whose right class precisely consists of the morphisms whose image under $G$ belongs to $\mR.$

\item If every morphism in the right class of $\mC$ is the image under $G$ of a morphism in the right class of $\mC'$, then $F$ detects the left class.
\item
If the unit $\id \to G F$ is objectwise in the left class and the counit is an equivalence, then a morphism is in the left class of $\mC'$ if and only if it factors as a morphism in the left class of $\mC$ followed by a local equivalence.
    
\end{enumerate}
\end{prop}

\begin{proof}

(1): Since $(\mL,\mR)$ is an accessible $\mV$-enriched factorization system on $\mC$, there is a set $S:=\{i_\alpha:S_\alpha\to T_\alpha\}$ of morphisms of $\mC$ such that a morphism of $\mC$ belongs to $\mR$
if and only if it has the right lifting property with respect to $S.$
Then a morphism of $\mC'$ is sent by $G$ to a morphism of $\mR$ if and only if it has the right lifting property with respect to $F(S).$
So (1) follows immediately from \cref{accfac} and \cref{fact}.

(2): We assume that every morphism $X\to Y$ in the right class of $\mC$ is the image under $G$ of a morphism $X'\to Y'$ in the right class of $\mC'$. Let $A\to B$ be a morphism in $\mC$ such that $F(A)\to F(B)$ is in the left class.
Consider a commutative square of the form
\[
\xymatrix{
A\ar[r]\ar[d] & X\ar[d]\\
B\ar[r] & Y}
\]
in $\mC$, where the morphism $X\to Y$ is in the right class.
The morphism $X\to Y$ is of the form $G(X')\to G(Y')$ for some morphism $X'\to Y'$ in the right class of $\mC'$. So we may check that the morphism $A\to B$ is in the left class by finding a lift in the commutative square
\[
\xymatrix{
F(A)\ar[r]\ar[d] & X'\ar[d]\\
F(B)\ar[r] & Y'.}
\]
This morphism $A\to B$ is in the left class of $\mC$ if the morphism $F(A)\to F(B)$ is in the left class of $\mC'$.

We prove (3).
Suppose that the unit $\id \to G F$ is objectwise in the right class and the counit is an equivalence.
Hence we can identify $\mC' $ with a full subcategory of $\mC$ and $F:\mC\to\mC$ with the associated endofunctor.

Let $\mL''$ be the class of morphism that factor as a morphism in the left class of $\mC$ followed by a local equivalence.
Then every morphism $X \to Y$ of $\mC'$ factors in $\mC$ as
$X \to Z \to Y$, where $X \to Z$ is in $\mL$ and $Z \to Y$ is in $\mR.$
We factor the latter morphism as $$X \to Z \to F(Z) \to Y,$$ 
where the morphism $Z \to F(Z)$ is the unit.
Then $X \to Z \to F(Z)$ is in $\mL''$ and $ F(Z) \to Y $ 
is in $\mR'$ since $ Z \to F(Z) \to Y $ is in $\mR'$ and the unit is in $\mR'.$

It remains to see that every morphism of $\mL''$ has the left lifting property with respect to every morphism of $\mR'$.
This implies $\mL'' =\mL'.$
Let $A \to B$ belong to $\mR'$ and let $X \to Y$ belong to $\mL'$.
We consider the morphism $X \to Y \to Y',$ where $Y \to Y'$ is the unit.
We have to see that the morphism
\[
\Mor_\mC(Y',A)\to\Mor_\mC(X,A)\times_{\Mor_\mC(X,B)}\Mor_\mC(Y',B)
\]
is an equivalence.
The latter morphism identifies with the equivalence
\[
\Mor_\mC(Y,A)\to\Mor_\mC(X,A)\times_{\Mor_\mC(X,B)}\Mor_\mC(Y,B).
\]
\end{proof}

\begin{definition}

Let $\mV$ be a presentably monoidal category.
A factorization system $(\mL, \mR)$ on $\mV$ is monoidal if the tensor product of two morphisms of the left class is again in the left class.
\end{definition}

\begin{remark}

Let $\mV$ be a presentably monoidal category.
By adjointness a factorization system $(\mL, \mR)$ on $\mV$ is monoidal if and only if for every $V \in \mV$ the left and right internal hom out of $V$
preserves the right class.
    
\end{remark}

\begin{proposition}\label{enrprefact}

Let $\mV$ be a presentably monoidal category, $\mC$ a presentable $\mV$-enriched category and $(\mL,\mR)$ an accessible monoidal factorization system on $\mV$.
There is a $\mV$-enriched factorization system on $\mC$ whose right class precisely consists of the morphism $A \to B$ in $\mC$ such that
for every object $C \in \mC$ the induced morphism
$\L\Mor_\mC(C,A) \to \L\Mor_\mC(C,B)$ is in $\mR.$
  
\end{proposition}

\begin{proof}

Let $\kappa$ be a small regular cardinal such that $\mC$ is a 
$\mV$-enriched $\kappa$-presentable category.
By \cite[Proposition 5.7.]{heine2024higher} the restricted $\mV$-enriched Yoneda-embedding
$\mC \to \Fun_\mV((\mC^\kappa)^\circ, \mV) $ is a $\kappa$-accessible $\mV$-enriched embedding and admits a $\mV$-enriched left adjoint.
By \cref{corsecfac} the $\mV$-enriched category
$\Fun_\mV((\mC^\kappa)^\circ, \mV) $ carries a $\mV$-enriched factorization system whose right class precisely consists of the morphisms that are objectwise in the right class. 
By \cref{factlocal} the accessible $\mV$-enriched localization 
$\mC$ carries a $\mV$-enriched factorization system whose right class precisely consists of the morphisms whose image under the restricted $\mV$-enriched Yoneda-embedding is objectwise in the right class.
\end{proof}

\begin{proposition}\label{factorization1} Let $\mV$ be a presentably monoidal category and $p:\mC\to\mD$ a $\mV$-enriched cartesian fibration whose fibers carry a $\mV$-enriched factorization system and whose fiber transports preserve the left and right class.
There is a $\mV$-enriched factorization system on $\mC$ whose left class precisely consists of the morphisms that belong to the left class of some fiber, and whose right class precisely consists of the morphisms that factor as a morphism in the right class of some fiber, followed by a $p$-cartesian morphism.

\end{proposition}

\begin{proof}
For every morphism $f:X\to Y$ in $\mC$ we factor $f$ as $t\circ f'$, where $t:p^*Y\to Y$ is $p$-cartesian and $f':X\to p^*X$ lies in the fiber over $p(X)$. 
We factor $f'=q\circ i$, where $i$ lies in the left class of the fiber $\mC_{p(X)}$ and $q$ lies in the right class of the fiber $\mC_{q(X)}$.
We obtain a factorization $f=(t\circ q)\circ i$.

We verify that we obtain a factorization system. Let $i:A\to B$ lie in the left class and $q:X\to Y$ lie in the right class.
It suffices to check that the morphism
\[
\theta:\Mor_\mC(B,X)\to\Mor_\mC(A,X)\times_{\Mor_\mC(A,Y)}\Mor_\mC(B,Y)
\]
is an equivalence.
This morphism covers the morphism
\[
\Mor_\mD(p(B),p(X))\to\Mor_\mD(p(A),p(X))\times_{\Mor_\mD(p(A),p(Y))}\Mor(p(B),p(Y)),
\]
and since $p(A)\to p(B)$ is an equivalence, the above morphism is equivalent to $\id_{\Mor_\mD(p(A),p(X))}$.
So it suffices to prove that $\theta$ induces an equivalence on the fiber over any morphism $h:p(A)\to p(X)$.
Since $p$ is a cartesian fibration, this induced morphism on the fibers identifies with the morphism
\[
\Mor_{\mC_{p(A)}}(B,h^*X)\to\Mor_{\mC_{p(A)}}(A,h^*X)\times_{\Mor_{\mC_{p(A)}}(A,h^*Y)}\Mor_{\mC_{p(A)}}(B,h^*Y)
\]
But this morphhism is an equivalence because $A\to B$ lies in the left class of the fiber $\mC_{p(A)}$ and $h^*X\to h^*Y$ lies in the right class of the fiber $\mC_{p(A)}$, since $X\to Y$ factors as $X\to p(i)^*Y\to Y$ where $X\to p(i)^*Y$ lies in the right class of the fiber $\mC_{p(A)}$ and $h^*p(i)^*Y\simeq h^*Y$.
\end{proof}

\subsection{Factorizations systems on the category of enriched categories}

\begin{prop}\label{Vcatfac}\label{Vcatunivfac}
Let $\mV$ be a presentably monoidal category and $(\mL,\mR)$
a factorization system on $\mV$ whose left and right class are preserved by the tensor product.

\begin{enumerate}[\normalfont(1)]\setlength{\itemsep}{-2pt}
\item There is a factorization system on $\Cat_\mV$, where the left class precisely consists of the $\mV$-enriched functors inducing on morphism objects a morphism in $\mL$ and an equivalence on spaces of objects, and the right class precisely consists of the $\mV$-enriched functors inducing on morphism objects a morphism in $\mR$.

\item There is a factorization system on $\Cat_\mV^\univ$, where the left class precisely consists of the $\mV$-enriched functors inducing on morphism objects a morphism in $\mL$ and a surjection on spaces of objects, and the right class precisely consists of the $\mV$-enriched functors inducing on morphism objects a map in $\mR$.

\end{enumerate}
\end{prop}

\begin{proof}
(1): We consider the cartesian fibration ${_\mV}\Cat\to\mS$ which sends a $\mV$-enriched category to its space of objects.
We write ${_\mV}\Cat_S$ for the fiber over a space $S$, which we can identify with the category $$\Alg(\Fun^\L(\mP_\mV(S),\mP_\mV(S))$$ of associative algebras in monoidal category of colimit-preserving endofunctors with respect to composition.

\cref{corsecfac} gives a factorization system on $$\Fun^\L(\mP_\mV(S),\mP_\mV(S))\simeq\Fun(S^{\op}\times S,\mV).$$
\cref{algfaccormon} provides a factorization system on $\Alg(\Fun^\L(\mP_\mV(S),\mP_\mV(S))$.

We apply \cref{secfac} to the cartesian fibration $_\mV\Cat\to\mS$ which sends a $\mV$-enriched category to its space of objects.
The cartesian morphisms are precisely the $\mV$-enriched functors that induce equivalences on morphism objects.
We obtain a factorization system on $_\mV\Cat$, where the left class precisely consists of the $\mV$-enriched functors which induce on morphism objects a morphism in $\mL$ and an equivalence on spaces of objects, and the right class precisely consists of the $\mV$-enriched functors which induce on morphism objects a morphism in $\mR$.

(2): We apply (1) and \cref{factlocal}. We use that the full subcategory $\Cat_\mV^\univ \subset \Cat_\mV$ is an accessible localization whose local equivalences are precisely the $\mV$-enriched functors that induce equivalences on morphism objects and a surjective map on spaces of objects.
\end{proof}



    

\begin{example}\label{fullsurj0}\label{fullsurj}
Let $\mV$ be a presentably monoidal category. We apply \cref{Vcatfac}
to the factorization system $(\mathrm{All},\mathrm{Eq})$ whose left class consists of all morphisms and whose right class consists of the equivalences.

\begin{enumerate}[\normalfont(1)]\setlength{\itemsep}{-2pt}
\item Then $\Cat_{\mV}$ inherits a factorization system, where the left class precisely consists of the $\mV$-enriched functors inducing an equivalence on spaces of objects, and the right class precisely consists of the $\mV$-enriched functors inducing on morphism objects an equivalence.

\item Then $\Cat_{\mV}^\univ$ inherits a factorization system, where the left class precisely consists of the $\mV$-enriched functors inducing a surjection on spaces of objects, and the right class precisely consists of the $\mV$-enriched functors inducing on morphism objects an equivalence.

\end{enumerate}
\end{example}


\section{\mbox{Truncation and connectivity in  higher category theory}}

\subsection{$\infty$-categories}

In this subsection we inductively define $n$-categories for any $0 \leq n \leq \infty$ via the theory of homotopy coherent enrichment.
We follow \cite{gepner2025oriented}.

\begin{definition}
For every $\n \geq 0$ we inductively define the presentable cartesian closed category $\n\Cat$ of small (not necessarily univalent) $\n$-categories by setting
$$(\n+1)\Cat:= {_{\n\Cat}\Cat}, \ 0\Cat :=\mS.$$
\end{definition}

\begin{notation}
For every $\n \geq 0$ we inductively define colocalizations
$$\n\Cat \rightleftarrows (\n+1)\Cat: \iota_\n,$$
where both adjoints  preserve finite products and filtered colimits.
Let
$$0\Cat= \mS \rightleftarrows 1\Cat = {_\mS \Cat} : \iota_0 $$ be the canonical colocalization whose right adjoint assigns the space of objects. Let $$(\n+1)\Cat= \n\Cat \mathrm{-}\Cat \rightleftarrows (\n+2)\Cat= (\n+1)\Cat\mathrm{-}\Cat:\iota_{\n+1}:= (\iota_\n)_! $$
be the induced adjunction.	
	
\end{notation}

\begin{definition}The presentable category $\infty\Cat$ of small (non-univalent) $\infty$-categories is the limit
$$\infty\Cat:= \lim(\cdots\xrightarrow{\iota_{\n}} \n \Cat \xrightarrow{\iota_{\n-1}}\cdots \xrightarrow{\iota_0} 0 \Cat) $$
of presentable categories and right adjoint functors.

\end{definition}

\begin{notation}

For every $n \geq 0$ we write
$\iota_n : \infty\Cat \to n\Cat$ for the projection.
The functor $\iota_n : \infty\Cat \to n\Cat$ is right adjoint to the embedding $ n\Cat \to \infty\Cat$ induced by the embeddings
$ n\Cat \to \m\Cat$ for $m \geq n.$
  
\end{notation}

We have the following filtration of every $\infty$-category:

The following is \cite[Lemma 3.1.17.]{gepner2025oriented}:

\begin{lemma}\label{decom}
Let $\mC$ be an $\infty$-category.
The sequential diagram $$\iota_0(\mC) \to \cdots \to \iota_\n(\mC) \to \iota_{\n+1}(\mC) \to \cdots \to \mC$$ exhibits $\mC$ as the colimit in $\infty\Cat$ of the diagram $\iota_0(\mC) \to \cdots \to \iota_\n(\mC) \to \iota_{\n+1}(\mC) \to \cdots$
\end{lemma}

The next proposition, which is \cite[Proposition 3.1.4.]{gepner2025oriented}, follows from the fact that the functor ${_{(-)}\Cat}$
preserves small weakly contractible limits:

\begin{proposition}\label{fix}
	
There is a canonical equivalence
$\infty\Cat \simeq {_{\infty\Cat}\!\Cat} .$
	
\end{proposition}

\begin{notation}
Let $\partial\bD^1$ denote the set $\{0,1\}$ with two elements.
\end{notation}


\begin{notation}
Let $\Mor: \infty\Cat_{\partial\bD^1/} \to \infty\Cat$
be the canonical functor $$\infty\Cat_{\partial\bD^1/} \simeq \mS_{\partial\bD^1/}\times_\mS {_{\infty\Cat}\Cat} \to \infty\Cat $$
sending $(\mC,\X,\Y)$ to $\Mor_\mC(\X,\Y).$
\end{notation}

The next follows from \cite[Corollary 2.3.2.]{gepner2025oriented}:

\begin{proposition}\label{suspi}

The functor $\Mor: \infty\Cat_{\partial\bD^1/} \to \infty\Cat $
admits a left adjoint $S$, the categorical suspension, such that for every $\infty$-category $\mC$ the $\infty$-category 
$S(\mC) $ has two objects $0,1 $ and morphism $\infty$-categories $$\Mor_{S(\mC)}(1,0)\simeq \emptyset, \ \Mor_{S(\mC)}(0,1)\simeq \mC, \ \Mor_{S(\mC)}(0,0)\simeq \Mor_{S(\mC)}(1,1) \simeq *.$$ 

\end{proposition}

The next is \cite[Lemma 3.1.18.]{gepner2025oriented}:

\begin{lemma}\label{carclo} 
The presentable category $\infty\Cat$ is cartesian closed.

\end{lemma}

\begin{notation}

For every $\infty$-category $\mC$ let 
$$\Fun(\mC,-): \infty\Cat \to \infty\Cat$$ be the right adjoint of the functor $(-) \times \mC:\infty\Cat \to \infty\Cat.$
    
\end{notation}

\begin{corollary}
The category $\infty\Cat$ refines to an $\infty$-category $\infty\scat$ such that for every $\mC, \mD \in \infty\scat$:
\[
\Mor_{\infty\scat}(\mC,\mD)=\Fun(\mC,\mD).
\]

\end{corollary}


\begin{definition}Let $ n \geq 1.$ By induction on $n$ we define $n$-univalent $\infty$-categories.

\begin{enumerate}[\normalfont(1)]\setlength{\itemsep}{-2pt}

\item An $\infty$-category is $1$-univalent if it is local with respect to the unique functor $\{0 \simeq 1 \} \to *$, where $\{0\simeq 1\}$ is the groupoid obtained from $\bD^1$ by inverting the unique map from $0$ to $1$.

\item An $\infty$-category is $n$-univalent if it is 1-univalent and all morphism $\infty$-categories are $n-1$-univalent.

\item An $\infty$-category is univalent if it is $n$-univalent for every $n \geq 1.$
\end{enumerate}

\end{definition}

\begin{definition}
Let $ \infty\Cat^{\univ}\subset \infty\Cat$
be the full subcategory of univalent $\infty$-categories.
\end{definition}

The next is \cite[Proposition 3.1.24.]{gepner2025oriented}:

\begin{proposition}\label{completion}
The embedding $ \infty\Cat^{\mathrm{univ}}\subset \infty\Cat$ admits a left adjoint $(-)^\wedge$, the univalent completion. 
\end{proposition}


\begin{definition}Let $ n \geq 0.$ By induction on $n$ we define $n$-strict $\infty$-categories.

\begin{enumerate}[\normalfont(1)]\setlength{\itemsep}{-2pt}
\item An $\infty$-category is $0$-strict if and only if it its underlying space is a set.

\item An $\infty$-category is $n$-strict if it is
$0$-strict and all morphism $\infty$-categories are $n-1$-strict.

\item An $\infty$-category is strict if it is $n$-strict for every $n \geq 0.$

\end{enumerate}

\end{definition}

\begin{definition}
Let $ \infty\Cat^{\strict}\subset \infty\Cat$
be the full subcategory of strict $\infty$-categories.
\end{definition}


\begin{definition}Let $1 \leq \n \leq \infty $.
An $\infty$-category is $n$-gaunt if it is $n$-strict and $n$-univalent.
\end{definition}

\begin{remark}
For every $\n \geq 0$ the colocalization
$\n\Cat \rightleftarrows \infty\Cat: \iota_\n$
restricts to the respective full subcategories of univalent $\infty$-categories and strict $\infty$-categories.

\end{remark}

\begin{remark}
By definition there are canonical equivalences
\begin{align*}
0\Cat^{\univ}&\simeq\mS,\\
n\Cat^{\mathrm{univ}}&\simeq {_{(n-1)\Cat^{\mathrm{univ}}}}\Cat^{\mathrm{univ}},\\
\infty\Cat^{\mathrm{univ}}&\simeq  \lim(\cdots\xrightarrow{\iota_{\n}} \n \Cat^{\mathrm{univ}}\xrightarrow{\iota_{\n-1}}\cdots \xrightarrow{\iota_0} 0 \Cat^{\mathrm{univ}}).
\end{align*}

\end{remark}







\begin{definition}Let $\n \geq 0$.
The $\n$-disk is the $n$-fold suspension $\bD^\n:= S^{\n}(*)$ of the final $\infty$-category $\ast$.
\end{definition}


\begin{definition}Let $\n \geq 0$.
The boundary of the $\n$-disk is the $n$-fold suspension $\partial\bD^\n:= S^{\n}(\emptyset)$ of the initial $\infty$-category $\emptyset$.
Moreover we adopt the convention that $\partial\bD^{-1}=\bD^{-1}=\partial\bD^0=\emptyset$.
\end{definition}

\begin{remark}
The functor $\emptyset \subset *$ induces inclusions $\partial\bD^\n \subset \bD^\n$ for every $\n \geq 0.$
	
\end{remark}

\begin{example}
Then $\partial\bD^0=\emptyset, \partial\bD^1=S(\emptyset)=*\coprod*$ is the set with two elements.
\end{example}


	
	







\begin{notation}
For every $\n \geq m \geq 0$ let 
$ (\n,\m)\Cat \subset \infty\Cat $ be the full subcategory
of $(n,m)$-categories.
\end{notation}

\begin{lemma}Let $n \geq 0.$

\begin{enumerate}[\normalfont(1)]\setlength{\itemsep}{-2pt}

\item Let $S_n=\{S^n\bD^{m}\to S^n\bD^0\}_{\{m\geq 0\}}$.
There is an equivalence $$n\Cat\simeq S_n^{-1}\infty\Cat.$$

\item Let $T_n=\{S^n\bD^{m}\to S^n\bD^0\}_{\{m\geq 0\}}\cup\{S^n\mathrm{B}\bZ\to S^n\bD^0\}$.
There is an equivalence $$(n,n)\Cat\simeq T_n^{-1}\infty\Cat.$$

\end{enumerate}

\end{lemma}

\begin{proof}
(1): An $\infty$-category is an $n$-category if and only if it is local with respect to the maps 
$S^n\bD^{m}\to S^n\bD^0$ for every $m\geq 0.$
(2): An $n$-category is an $(n,n)$-category if and only if it is local with respect to the map $\Sigma^n\B\bZ\to\Sigma^n\bD^0$.
\end{proof}

\begin{notation}\label{sqGray}

Let $\cube \subset \infty\Cat $ be the full subcategory of oriented cubes.
\end{notation}

By \cite[Definition 3.3.23.]{gepner2025oriented} there is a monoidal structure on $\cube,$
the Gray monoidal structure $\boxtimes$, whose tensor unit is $\bD^0$
and such that $\cube^\n \boxtimes \cube^m = \cube^{n+m}$ for every $n,m \geq 0.$

\vspace{1mm}
There is the following result of Campion \cite{campion2022cubesdenseinftyinftycategories},
which was reproven by \cite[Corollary 3.8.2.]{gepner2025oriented}:

\begin{theorem}\label{locmon2}
The convolution monoidal structure descends along 
the localizations $$\mP(\cube) \rightleftarrows \infty\Cat, \ \mP(\cube) \rightleftarrows \infty\Cat^\univ.$$

\end{theorem}





	
	

\begin{notation}
	
Since the Gray tensor product defines a presentably monoidal structure on $\infty\Cat$, it is closed: for every $\infty$-category $\mC$ the functor $$ \mC \boxtimes (-): \infty\Cat \to \infty\Cat$$ admits a right adjoint $\Fun^\lax(\mC,-)$, and the functor $$ (-) \boxtimes \mC : \infty\Cat \to \infty\Cat$$ admits a right adjoint $\Fun^\oplax(\mC,-)$.
\end{notation}

\begin{definition}Let $\F,\G:\mC \to \mD $ be functors of $\infty$-categories.
\begin{enumerate}[\normalfont(1)]\setlength{\itemsep}{-2pt}
\item A lax natural transformation $\F \to \G$ is a morphism in $\Fun^\lax(\mC,\mD)$.

\item An oplax natural transformation $\F \to \G$ is a morphism in $\Fun^\oplax(\mC,\mD)$.
\end{enumerate}
\end{definition}

\begin{remark}
	
Let $\mC,\mD$ be $\infty$-categories. The canonical functors $\mC \to *, \mD \to *$
give rise to functors $\mC \boxtimes \mD \to \mC \boxtimes * \simeq \mC, \mC \boxtimes \mD \to * \boxtimes \mD \simeq \mD$ and so to a functor $\mC \boxtimes \mD \to \mC \times \mD.$
By adjointess the latter functor induces functors
$ \Fun(\mC,\mD)\to \Fun^\lax(\mC,\mD)$ and $ \Fun(\mC,\mD)\to \Fun^\oplax(\mC,\mD).$
\end{remark}



The following is \cite[Proposition 3.8.8.]{gepner2025oriented}:

\begin{proposition}\label{dua}
	
There are canonical monoidal involutions
$$(-)^\op, (-)^\co: (\infty\Cat, \boxtimes) \simeq (\infty\Cat, \boxtimes)^\rev $$
that reverse the odd dimensional morphisms, even dimensional morphisms, respectively.

 
\end{proposition}

The next is \cite[Corollary 3.8.9.]{gepner2025oriented}:

\begin{corollary}\label{grayhoms}
Let $\mC,\mD \in \infty\Cat$.
There are canonical equivalences $$ \Fun^\oplax(\mC,\mD)^\op \simeq \Fun^\lax(\mC^\op,\mD^\op), $$$$\Fun^\oplax(\mC,\mD)^\co \simeq \Fun^\lax(\mC^\co,\mD^\co).$$	
\end{corollary}

\subsection{Strong and weak $n$-equivalences}

\begin{convention}
In this subsection, we do not require our $\infty$-categories to be univalent.
\end{convention}
\begin{definition}Let $ \n \geq 1.$
\begin{enumerate}[\normalfont(1)]\setlength{\itemsep}{-2pt}
\item A functor of $\infty$-categories is a strong $0$-equivalence 
if induces an equivalence on underlying spaces.
	
\item A functor of $\infty$-categories is a strong $n$-equivalence 
if it is a strong $0$-equivalence and induces strong $n-1$-equivalences on morphism $\infty$-categories.
		
\end{enumerate}
	
\end{definition}

\begin{example}
Let $n \geq 0 $ and $X$ an $\infty$-category.
The functor $X \to *$ is a strong $n$-equivalence if and only if $X$ is the $n$-fold delooping of an $n$-monoidal $\infty$-category.
\end{example}

\begin{example}
It is important in the example above that we are not working univalently.
If we work instead in $\infty\Cat^\univ$, then a functor $X \to *$ is a strong $n$-equivalence if and only if 
$X$ is the $n$-fold delooping of an $n$-monoidal $\infty$-category whose full subspace of tensor invertible objects is contractible.
\end{example}

\begin{proposition}Let $\n \geq 0.$
There exists a factorization system on $\infty\Cat$, where the left class consists of the strong $n$-equivalences and the right class consists of the $n$-faithful functors.
\end{proposition}

\begin{proof}

This follows by induction on $\n \geq 0.$
The case $\n=0$ is \cref{fullsurj0}.
We assume the statement holds for $n.$
We apply \cref{Vcatfac} for enrichment in $ \infty\Cat$ to the resulting factorization system on $\infty\Cat$ to obtain the statement for $\n+1.$ 
\end{proof}

\begin{lemma}
Let $ \n \geq 0.$
A functor $\phi: X \to Y$ of $\infty$-categories is a strong $n$-equivalence 
if and only if the induced functor $\iota_n(X) \to \iota_n(Y)$ is an equivalence.
\end{lemma}

\begin{proof}

We proceed by induction over $n \geq 0.$
For $n=0$ there is nothing to show.
We assume the statement holds for $n-1$.

The functor $\phi: X \to Y$ and the functor $\iota_n(\phi): \iota_n(X) \to \iota_n(Y)$ induce equivalent functors on $\iota_0$. Moreover the induced functor $\iota_n(\phi): \iota_n(X) \to \iota_n(Y)$ induces on morphisms $\infty$-categories
the image under $\iota_{n-1}$ of what $\phi: X \to Y$ induces on morphism $\infty$-categories.
Thus by induction hypothesis the functor $\phi: X \to Y$ is a strong $n$-equivalence if and only if the induced functor $\iota_n(\phi): \iota_n(X) \to \iota_n(Y)$ induces an equivalence on underlying spaces and on morphism
$\infty$-categories, which holds precisely if $\iota_n(\phi)$ is an equivalence.
\end{proof}

\begin{corollary}\label{streqli}

Let $ \n \geq 0.$
A functor $\phi: X \to Y$ of $\infty$-categories is a strong $n$-equivalence if and only if it has the unique right lifting property with 
respect to all functors $\emptyset \to Z$ for every $n$-category $Z.$
\end{corollary}
    
\begin{corollary}
An $\infty$-category $Z$ is an $n$-category if and only if 
the functor $ \emptyset \to Z$ has the unique left lifting property with 
respect to all strong $n$-equivalences.
\end{corollary}

\begin{lemma}\label{strequ} Let $\n \geq 0.$
A functor $\phi :X\to Y$ of $\infty$-categories is a strong $n$-equivalence if for every $m\leq n$ the space of fillers of any commutative diagram of the form
\[
\xymatrix{
\partial\bD^m\ar[r]\ar[d] & X\ar[d]\\
\bD^m\ar[r] & Y,
}
\]
is contractible.
\end{lemma}

\begin{proof}

We proceed by induction on $n \geq 0.$
For $n=0$ there is nothing to show.
We assume the statement for $n-1$.

The statement for $n$ follows from the fact that the square of the statement admits a unique filler if and only if it admits a unique filler
for $m=0$ and for every $n \geq m > 0$ and every $A,B \in X$ the commutative square 
\[
\xymatrix{
\partial\bD^{m-1}\ar[r]\ar[d] & \Mor_X(A,B) \ar[d]\\
\bD^{m-1}\ar[r] & \Mor_Y(\phi(A), \phi(B))}
\]
admits a unique filler.
\end{proof}


\cref{strequ} and \cref{factlocal} give the following:

\begin{proposition}\label{factsteq} Let $\n \geq 0.$

\begin{enumerate}[\normalfont(1)]\setlength{\itemsep}{-2pt}
\item There is a factorization system on $\infty\Cat$, where the right class consists of the strong $n$-equivalences.

\item There is a factorization system on $\infty\Cat^\univ$, where the right class consists of the strong $n$-equivalences.

\end{enumerate}

\end{proposition}

\begin{example}

The factorization system provided by \cref{factsteq}
induces a colocalization on the category
$ \infty\Cat_{\emptyset /} \simeq \infty\Cat$.
An $\infty$-category $Z$ is colocal if and only if the functor $\emptyset \to Z$ is in the left class,
which is equivalent to say that it has the unique left lifting property with respect to all strong $n$-equivalences.
So by \cref{streqli} an $\infty$-category is colocal if and only if it is an $n$-category.
Hence the colocalization on the category
$ \infty\Cat$ identifies with the colocalization $\iota_\n.$
    
\end{example}





\cref{strequ} and \cref{factlocal} give the following:



\begin{definition}Let $ \n \geq 1.$
\begin{enumerate}[\normalfont(1)]\setlength{\itemsep}{-2pt}
\item A functor of $\infty$-categories is a $0$-equivalence 
if it induces a bijection on equivalence classes of objects.
	
\item A functor of $\infty$-categories is an $n$-equivalence 
if it is a $0$-equivalence and induces $n-1$-equivalences on morphism $\infty$-categories.
		
\end{enumerate}
	
\end{definition}

\begin{remark}

By definition and induction 
every $n$-equivalence is an $m$-equivalence for every $m < n.$
    
\end{remark}

\begin{notation}

Let $n \geq 0.$ By induction we define the homotopy $n$-category
$$ \Ho_n: \n\Cat^\univ \to (n,n)\Cat^\univ$$ by setting:

\begin{enumerate}[\normalfont(1)]\setlength{\itemsep}{-2pt}
\item For $n\geq 1$ let $$ \Ho_{n}: \n\Cat^\univ \simeq {_{n-1\Cat^\univ}\Cat^\univ} \xrightarrow{(\Ho_{n-1})_!} {_{(n-1,n-1)\Cat^\univ}\Cat^\univ} \simeq (n,n)\Cat^\univ. $$

\item For $n= 0$ let $$ \Ho_{0}:=\pi_0 : 0\Cat \to (0,0)\Cat. $$
\end{enumerate}
    
\end{notation}


    

\begin{remark}
Let $n \geq 0.$
By construction 
$ \Ho_n: \n\Cat \to (n,n)\Cat $
is left adjoint to the embedding $$ (n,n)\Cat \subset n\Cat.$$
    
\end{remark}

\begin{example}\label{homot}

Let $X$ be a space. Then $\Ho_0(X)$ is the set of equivalence classes of the space. Then $\Ho_1(X)$ is the fundamental groupoid of the space.
In particular, the group of endomorphisms in $\Ho_1(X)$ of any object is the fundamental group of $X$ in that point.
For every $n \geq 2$ the group of $n-1$-fold endomorphisms of the group of endomorphisms of any object in $\Ho_n(X)$ is the $n$-th homotopy group of the space in that point.
    
\end{example}

\begin{lemma}\label{equin} Let $n \geq 0.$
A functor $X \to Y$ of univalent $\infty$-categories is an $n$-equivalence
if and only if the induced functor $$\Ho_n(\iota_n(X)) \to \Ho_n(\iota_n(Y))$$ is an equivalence.
    
\end{lemma}

\begin{proof}

We proceed by induction on $n \geq 0$.
For $n=0$ there is nothing to show.
We assume the statement holds for $n-1$.
The statement for $n$ follows from the induction hypothesis and that
a functor is an equivalence if it induces equivalences on morphism $\infty$-categories and on the underlying space.   
\end{proof}

\begin{lemma}\label{lemcore}

Let $n, m \geq 0$ and $\phi: X \to Y$ an $n$-equivalence.
Then the functor $\iota_m(\phi)$ is an $n$-equivalence.
    
\end{lemma}

\begin{proof}

If $m \geq n$, the functor $\Ho_n(\iota_n(\iota_m(\phi))) = \Ho_n(\iota_n(\phi))$ and so is an equivalence since $\phi$ is an $n$-equivalence.

If $m < n$, the functor $\Ho_n(\iota_n(\iota_m(\phi))) = \Ho_n(\iota_m(\phi))$ is between $(n,m)$-categories and therefore the induced functor on classifying $(n,m)$-categories of the functor $\Ho_n(\iota_n(\phi))$ of $(n,n)$-categories, which is an equivalence.   
\end{proof}

\begin{proposition}\label{streqchar}

Let $\phi: X \to Y$ be a functor.
The following are equivalent:

\begin{enumerate}[\normalfont(1)]\setlength{\itemsep}{-2pt}
\item The functor $\phi: X \to Y$ is an equivalence.

\item For every $n \geq 0$ the functor $\phi: X \to Y$ is a strong $\n$-equivalence.

\item For every $n \geq 0$ the functor $\phi: X \to Y$ is an $\n$-equivalence.
    
\end{enumerate}

\end{proposition}

\begin{proof}

(1) trivially implies (2). (2) trivially implies (3).

We prove that (3) implies (1).
The functor $\phi: \mC \to \mD$ is the sequential colimit of
the sequence of functors $\iota_m(\phi)$ for $m \geq 0.$
Thus $\phi$ is an equivalence if and only if for every $m \geq 0$
the functor $\iota_m(\phi)$ is an equivalence.
If (3) holds, by \cref{lemcore} for every $m \geq 0$ the functor $\iota_m(\phi)$ is an $n$-equivalence of every $n \geq 0.$
So we can reduce to show that for every $m \geq 0$
a functor $\phi$ between $m$-categories is an equivalence if it is an
$n$-equivalence for every $n \geq 0.$
We prove this by induction on $m \geq 0.$

We start with the case $m=0.$
A map of spaces $\phi$ that is an $n$-equivalence for every $n \geq 0$,
induces isomorphisms on homotopy groups by \cref{homot} and so is an equivalence.

Let $m \geq 1$ and we assume the statement holds for $m-1$.
Let $\phi$ be a functor between $m$-categories that is an
$n$-equivalence for every $n \geq 0.$
Then $\phi$ induces on morphism $\infty$-categories a functor
between $m-1$-categories that is an
$n-1$-equivalence for every $n \geq 1.$
So by induction hypothesis the functor $\phi$ induces equivalences on morphism $\infty$-categories. 
Thus $\phi$ is an equivalence if $ \iota_0(\phi) $ is an equivalence.
Since $\phi$ is an $n$-equivalence for every $n \geq 0,$
by \cref{lemcore} the functor $ \iota_0(\phi) $ is an $n$-equivalence for every $n \geq 0$ and so an equivalence by the induction start.
\end{proof}

Next we define weak $n$-equivalences.


\begin{notation}
For every $\n \geq 0 $ the left adjoint embedding $\n\Cat \leftrightarrows \infty\Cat: \iota_\n$ preserves small limits
and so admits a left adjoint $$\tau_\n: \infty\Cat \to \n\Cat$$ by presentability.
Similarly, the left adjoint embedding $\n\Cat^{\mathrm{univ}}\leftrightarrows \infty\Cat^{\mathrm{univ}}: \iota_\n$
preserves small limits and so admits a left adjoint, which we also denote by $\tau_\n.$

\end{notation}

\begin{definition}Let $n \geq 0.$

\begin{enumerate}[\normalfont(1)]\setlength{\itemsep}{-2pt}


\item A functor $X \to Y$ is a coinductive $n$-equivalence
if the induced functor $\tau_{n}(X) \to \tau_{n}(Y)$ is an equivalence.

\item A functor $X \to Y$ is a weak $n$-equivalence
if the induced functor $\Ho_n(\tau_{n}(X)) \to \Ho_n(\tau_{n}(Y))$ is an equivalence.

\item A functor $X \to Y$ is a Postnikov $n$-equivalence
if the induced functor $\tau_{\leq n}(X) \to \tau_{\leq n}(Y)$ is an equivalence.

\end{enumerate}

\end{definition}

\begin{definition}
\begin{enumerate}[\normalfont(1)]\setlength{\itemsep}{-2pt}
\item
The {\em dimension tower} is the tower 
$$ \{ \tau_{n}: \infty\Cat \to n\Cat \}_{\n \geq 0}.$$
\item
The {\em truncation tower} is the tower 
$$ \{ \Ho_n \circ \tau_{n}: \infty\Cat \to (n,n)\Cat \}_{\n \geq 0}.$$
\end{enumerate}
\end{definition}


\begin{proposition}
Let $\bk \geq -1.$
There is a canonical equivalence
$$ \lim_{n \geq 0} n \Cat \simeq \lim_{n \geq 0} (n+k,n) \Cat. $$ 
\end{proposition}

\begin{proof}
Let $\n \geq 0.$
There is a canonical equivalence
$$ n \Cat \simeq \lim_{m \geq 0} (m,n) \simeq \lim_{m \geq n-1} (m,n) \Cat \simeq \lim_{m \geq n -k-1} (m+k,n) \Cat,$$
which gives rise to an equivalence
$$ \lim_{n \geq 0} n \Cat \simeq \lim_{n \geq k} n \Cat \simeq  \lim_{n \geq k} \lim_{m \geq n -k-1} (m+k,n) \Cat \simeq \lim_{n \geq k} (n+k,n) \Cat \simeq \lim_{n \geq 0} (n+k,n) \Cat. $$ 
\end{proof}

\begin{corollary}\label{towercomp}

There are canonical equivalences
$$ \lim_{n \geq 0} n \Cat \simeq \lim_{n \geq 0} (n,n) \Cat \simeq \lim_{n \geq 0} (n-1,n) \Cat.$$ 

\end{corollary}

\begin{corollary}\label{charas}

Let $\phi: X \to Y$ be a functor. The following are equivalent:

\begin{enumerate}[\normalfont(1)]\setlength{\itemsep}{-2pt}
\item For every $n \geq 0$ the functor $\phi: X \to Y$ is a coinductive $\n$-equivalence.

\item For every $n \geq 0$ the functor $\phi: X \to Y$ is a Postnikov $\n$-equivalence.

\item For every $n \geq 0$ the functor $\phi: X \to Y$ is a weak $\n$-equivalence.
    
\end{enumerate}
    
\end{corollary}

\begin{proof}

A functor $\phi$ is a coinductive $\n$-equivalence for every $n \geq 0$
if and only if its image under the functor
$$ \infty\Cat \to \lim_{n \geq 0} n \Cat$$ is an equivalence, where the limit is over the left adjoints of the natural embeddings.

A functor $\phi$ is a weak $\n$-equivalence for every $n \geq 0$
if and only if its image under the functor
$$ \infty\Cat \to \lim_{n \geq 0} (n,n) \Cat$$ is an equivalence.

A functor $\phi$ is a Postnikov $\n$-equivalence for every $n \geq 0$
if and only if its image under the functor
$$ \infty\Cat \to \lim_{n \geq 0} (n,n+1) \Cat$$ is an equivalence.
We apply \cref{towercomp}.
\end{proof}

\begin{definition}

A functor is Postnikov equivalence if it is a Postnikov $\n$-equivalence for every $n \geq 0.$
    
\end{definition}

\begin{remark}

\cref{charas} says that a functor is a Postnikov equivalence if and only if it is a weak $\n$-equivalence for every $n \geq 0$
or equivalently, if and only if it is inverted by the universal functor $\infty\Cat \to n\Cat$ for every $n \geq 0$.

\end{remark}

\subsection{Truncated and connected functors}

\begin{convention}
In this subsection and all subsequent subsections, we will always assume that our $\infty$-categories are univalent.
\end{convention}

\begin{definition}Let $\n \geq 1.$
\begin{enumerate}[\normalfont(1)]\setlength{\itemsep}{-2pt}
\item Every functor of $\infty$-categories is $-1$-full.
	
\item A functor of $\infty$-categories is 0-full if it is essentially surjective.
	
\item A functor of $\infty$-categories is $\n$-full if it induces $n-1$-full functors on morphism $\infty$-categories.
\end{enumerate}
\end{definition}

\begin{lemma}Let $\n \geq 0.$
A functor $\phi:X\to Y$ is $n$-full if every commutative square of the form
\[
\xymatrix{
\partial\bD^n\ar[r]\ar[d] & X\ar[d]\\
\bD^n\ar[r] & Y}
\]
admits a lift, which we do not require to be unique.
\end{lemma}

\begin{proof}

We proceed by induction on $n \geq 0.$
For $n=0$ there is nothing to show.
We assume the statement for $n-1$.

The statement for $n$ follows from the fact that the square of the statement admits a non-unique filler if and only if for for every $A,B \in X$
the commutative square 
\[
\xymatrix{
\partial\bD^{n-1}\ar[r]\ar[d] & \Mor_X(A,B) \ar[d]\\
\bD^{n-1}\ar[r] & \Mor_Y(\phi(A), \phi(B))}
\]
admits a non-unique filler. 
\end{proof}


\begin{definition} Let $\n \geq -1.$

\begin{enumerate}[\normalfont(1)]\setlength{\itemsep}{-2pt}

\item A functor of $\infty$-categories is $n$-connective if it is $m$-full for all $0 \leq m\leq n$.

\item Let $\n \geq -2.$ A functor of $\infty$-categories is $n$-connected if it is $n+1$-connective.

\item A functor of $\infty$-categories is $\infty$-connective, or equivalently $\infty$-connected, if it is $n$-connective for every $n\geq -1$.
\end{enumerate}
\end{definition}

\begin{definition} Let $\n \geq -2.$
An $\infty$-category $X$ is $n$-connected if the functor $X \to *$ is $n$-connected.

\end{definition}

\begin{remark}\label{induco} Let $\n \geq -1.$
By definition a functor is $n$-connected if and only if it is essentially surjective and induces on morphism $\infty$-categories
$n-1$-connected functors.
    
\end{remark}

\begin{example}

\begin{enumerate}[\normalfont(1)]\setlength{\itemsep}{-2pt}

\item Every functor of $\infty$-categories is $(-1)$-connective ($(-2)$-connected).
\item A functor of $\infty$-categories is $0$-connective ($(-1)$-connected) if and only if it is essentially surjective.
\end{enumerate}
\end{example}







\begin{proposition}\label{counter}

There is an $\infty$-connected functor which is not an equivalence.

\end{proposition}

\begin{proof}

For every $n \geq 0$ let $\Bord_n$ the bordism $n$-category of \cite[Definition 1.4.6.]{Cob} whose $i$-morphisms for $i \leq n$ are $i$-dimensional smooth manifolds, which are cobordisms between $i-1$-dimensional manifolds, whose $n+1$-morphisms are diffeomorphisms and whose higher morphisms are isotopies.

There is a canonical functor $\Bord_n \to \Bord_{n+1}$.
Let $\Bord_\infty $ be the sequential colimit $\Bord_1 \to \Bord_2 \to ...$
The $n$-category $\Bord_n$ admits adjoints of dimension smaller $n$,
which we define by induction.
This means that every morphism admits a left and right adjoint and
every morphism $n-1$-category admits adjoints of dimension smaller $n-1$.
Therefore the colimit $\Bord_\infty$ admits adjoints of dimension smaller $n$ for every $n \geq 1.$

By induction on $n \geq 0$ we prove that for every $\infty$-category $\mC$ that admits adjoints of dimension smaller or equal $n+1$ the $(n,n+1)$-category $\tau_{\leq n}(\mC)$ is a space.
For every $n \geq 0 $ the functor $\mC \to \tau_{\leq n}(\mC)$ is $n$-connected. Hence every $i$-morphism of $\tau_{\leq n}(\mC)$ for $i \leq n+1$ lifts to an $i$-morphism of $\mC$. So also $\tau_{\leq n}(\mC)$
admits adjoints for morphisms of dimension smaller or equal $n+1.$
Hence it suffices to see that every $n+1$-category that admits adjoints for morphisms of dimension smaller or equal $n+1$ is a space.
By induction on $n \geq 0$ we can reduce to show that in 
every $n+1$-category that admits adjoints for morphisms of dimension smaller or equal $n+1$ every morphism is an equivalence.
We prove this by induction on $ n \geq 0.$
Let $f : A \to B$ be a morphism in $\mC$ and $g : B \to A$ the right adjoint. The unit $\id \to gf$ and counit $fg \to \id$ are morphisms
in the $n$-category $\Mor_\mC(A,B) $, which admits adjoints for morphisms of dimension smaller or equal $n$. Thus by induction hypothesis the unit and counit are equivalences so that $f$ is an equivalence.
The induction start holds because in 
every $1$-category every morphism that admits a right adjoint is an equivalence since unit and counit are equivalences.

So we have seen that for every $n \geq 0 $ the $(n,n+1)$-category $\tau_{\leq n}(\Bord_\infty)$ is a space.
Thus for every $n \geq 0$ the functor $\tau_{\leq n+1}(\Bord_\infty) \to \tau_{\leq n}(\Bord_\infty)$ is an equivalence so that 
the Postnikov completion $\Bord_\infty \to \lim_{n \geq 0}\tau_{\leq n}(\Bord_\infty) $ is the functor $\Bord_\infty \to \tau_{\leq 0}(\Bord_\infty) $. The functor $\Bord_\infty \to \tau_{\leq 0}(\Bord_\infty) $ is not an equivalence since $\Bord_\infty$ is not a space.
On the other hand, for every $n \geq 0$ the functor $\Bord_\infty \to \tau_{\leq 0}(\Bord_\infty) $ is equivalent to the functor $\Bord_\infty \to \tau_{\leq n}(\Bord_\infty) $, which is $n$-connected. Hence the functor $\Bord_\infty \to \tau_{\leq 0}(\Bord_\infty) $ is $\infty$-connected.
\end{proof}

\begin{definition} Let $ \n \geq 1.$
\begin{enumerate}[\normalfont(1)]\setlength{\itemsep}{-2pt}
\item A functor of $\infty$-categories is $-1$-faithful if it is an equivalence.
	
\item A functor of $\infty$-categories is $0$-faithful if it is fully faithful.
		
\item A functor of $\infty$-categories is $\n$-faithful if it induces $n-1$-faithful functors on morphism $\infty$-categories.
		
\end{enumerate}
	
\end{definition}



\begin{lemma}
Let $\n \geq -1.$

\begin{enumerate}[\normalfont(1)]\setlength{\itemsep}{-2pt}

\item Every $\n$-faithful functor is $\n+1$-faithful and $\n+1$-full.

\item Every $\n+1$-faithful and $\n+1$-full functor between univalent $\infty$-categories is $\n$-faithful.
\end{enumerate}
    
\end{lemma}

\begin{proof}
Statement
(1) follows by induction, where the induction start says that every equivalence is essentially surjective and fully faithful.
Statement
(2) also follows by induction, where the induction start says that every essentially surjective and fully faithful functor between univalent
$\infty$-categories is an equivalence.
\end{proof}

\begin{corollary}\label{eqful}

Let $\n \geq -1$ and $m \geq 1.$
A functor between univalent $\infty$-categories is $\n$-faithful
if and only if it is $n+m$-faithful and $n+k$-full for every $1 \leq k \leq m.$
    
\end{corollary}

\begin{remark}Let $\n \geq -1.$
A functor from a $n$-category to an univalent $n$-category is the univalization if and only if it is $n-1$-full and $n-1$-faithful.

\end{remark}

\begin{definition}Let $\n \geq -2.$ A functor of $\infty$-categories is $n$-truncated if it is $n+1$-faithful.
    
\end{definition}

\begin{definition} Let $\n \geq -2.$
An $\infty$-category $X$ is $n$-truncated if the functor $X \to *$ is $n$-truncated.

\end{definition}

\begin{example}\label{poseto}

Let $n \geq -1$ and $X$ an $\infty$-category.
A functor $X \to *$ is $n$-truncated if and only if
$X$ is an $(n,n+1)$-category.
This follows by induction from \cref{induco}, where the induction start says that a functor $X \to *$ is $-2$-truncated (an equivalence) if and only if $X$ is a $(-2,-1)$-category (the final $\infty$-category).

For instance, a functor $X \to *$ is $-1$-truncated (fully faithful) if and only if $X$ is a $(-1,-0)$-category (the initial or final $\infty$-category).

\end{example}

\begin{lemma}
Let $\n \geq -1.$ A functor of $\infty$-categories is $n$-faithful if and only if it has the (unique) right lifting property with respect to the inclusions $\partial\bD^m\to\bD^m$ for every $m> n$.
    
\end{lemma}

\begin{proof}
We prove the statement by induction on $n \geq -1.$
For $n=-1$ the statement follows from \cref{strequ} and \cref{streqchar}.

We assume the statement for $n-1$ and prove the statement for $n.$
A functor of $\infty$-categories has the unique right lifting property with respect to the inclusions $\partial\bD^m\to\bD^m$ for every $m> n$ if and only if the functor induces on morphism $\infty$-categories functors, which have the unique right lifting property with respect to the inclusions $\partial\bD^{m-1}\to\bD^{m-1}$ for every $m> n$.
So the result follows from the induction hypothesis. 
\end{proof}


%
%

\begin{lemma}
If $n>m$, every $n$-connective functor of univalent $(m,m)$-categories is an equivalence.
\end{lemma}

\begin{proof}
We proceed by induction on $m \geq -1$. Let $m \geq -1$ and suppose the statement is true for $m$. Let $n>m+1$ and let $f:X\to Y$ be an $n$-connective functor of univalent $(m+1,m+1)$-categories.
We wish to show that $f$ is an equivalence.
But $f$ is essentially surjective, so it suffices to check that $f$ is fully faithful.
Thus we may pick a pair of objects $s,t\in X$ and consider the induced functor on morphism $(m,m)$-categories.
This functor is $(n-1)$-connective, but $n-1>m$, so this functor is an equivalence by inductive hypothesis.

To begin the induction, we consider $m=-1$ and $f:X\to Y$ is $n$-connected for $n>-1$.
In particular, $f$ is $0$-connective, so it is essentially surjective.
Hence it is an equivalence since both $X$ and $Y$ are either empty or contractible.
\end{proof}

%
%
%

\begin{corollary}\label{trunca}Let $n \geq -2$.
Every $n$-connected functor between univalent $n$-truncated $\infty$-categories is an equivalence.
    
\end{corollary}

\begin{proof}

We prove the statement by induction on $n \geq -2.$
Let $n \geq -1.$ We assume the statement for $n-1$ and prove the statement for $n.$
Let $\phi: X \to Y$ be an $n$-connected functor between univalent $n$-truncated $\infty$-categories. Then $\phi$ induces on morphism $\infty$-categories an $n-1$-connected functor between univalent $n-1$-truncated $\infty$-categories, which is an equivalence by induction hypothesis.
The functor $\phi: X \to Y$ is essentially surjective and so an equivalence since $\phi$ is $n$-connected and so in particular $-1$-connected. 
So it suffices to show the induction start.
We have to see that every $-1$-connected functor, i.e. essentially surjective functor, between univalent $-1$-truncated $\infty$-categories is an equivalence.
This is clear since an $\infty$-category is $-1$-truncated if and only if it is empty or contractible.
    
\end{proof}

\begin{theorem}\label{mainfact} Let $n \geq -2$.
There is a factorization system on $\infty\Cat^\univ$, where the left class consists of the $n$-connected ($n+1$-connective) functors and the right class consists of the $n$-truncated ($n+1$-faithful) functors.
\end{theorem}

\begin{proof}
We prove the statement by induction on $n$.
If $n=-1$, the factorization system is trivial since the left class consists of all functors between univalent $\infty$-categories and the right class consists of the equivalences.
Assuming the result for $n$, we show the result for $n+1$.
The existence of this factorization system on $_{\infty\Cat^\univ}\Cat^\univ$ follows from \cref{factorization}
and the equivalence $\infty\Cat^\univ \simeq {_{\infty\Cat^\univ}\Cat^\univ}$.
\end{proof}

\begin{corollary}\label{corasi}
Let $n \geq -2.$ There is a localization on $\infty\Cat$ whose local objects are the $n$-truncated $\infty$-categories, which are precisely the $(n,n+1)$-categories. For every $\infty$-category $X$ the unit
$X \to L(X)$ is an $n$-connected functor.
    
\end{corollary}

\begin{corollary}\label{weakequ} Let $n \geq -2$.
Every $n$-connected functor between univalent $\infty$-categories is a Postnikov $n$-equivalence.

\end{corollary}

\begin{proof}

By \cref{corasi} the embedding $ (n,n+1)\Cat \subset \infty\Cat$ admits a left adjoint $L$ and for every $\infty$-category $X$ the functor
$X \to *$ factors as an $n$-connected functor $X \to L(X)$ followed by an $n$-truncated functor $L(X) \to *.$
For every functor $X \to Y$ the functor $X \to Y \to L(Y)$ factors as $X \to L(X) \to L(Y)$.
Thus if the functor $X \to Y$ is $n$-connected, also the induced functor $L(X) \to L(Y)$ is $n$-connected and so an equivalence by \cref{trunca}.

\begin{corollary}\label{posteq}
Every $\infty$-connected functor between univalent $\infty$-categories is a Postnikov equivalence.

\end{corollary}

\cref{eqful} specializes to the following:

\begin{corollary}\label{equitrco}
Let $ n \geq -2.$
A functor between univalent $\infty$-categories is an equivalence
if and only if it is $n$-truncated and $n$-connected.

    
\end{corollary}

\cref{equin} and \cref{equitrco} give the following:

\begin{corollary}\label{equitrco2}
Let $ n \geq 0.$
A functor between univalent $\infty$-categories is an $n$-equivalence
if and only if the induced functor $$\Ho_n(\iota_n(X)) \to \Ho_n(\iota_n(Y))$$ is $n$-faithful and $n$-connective.
 
\end{corollary}




\end{proof}

\begin{remark}Let $n \geq -2.$
By \cref{corfact} and \cref{poseto} the factorization system on $\infty\Cat^\univ$ provided by \cref{mainfact} induces a localization $$\tau_{\leq n}: \infty\Cat^\univ \rightleftarrows (n,n+1)\Cat^\univ .$$ 

For $n=-1$ the local objects are precisely the initial or final $\infty$-category and the localization $$ \tau_{\leq -1}: \infty\Cat^\univ \rightleftarrows \{\emptyset, *\} $$ sends the initial $\infty$-category to itself and every other $\infty$-category to the final $\infty$-category.

The localization $$\tau_{\leq n}: \infty\Cat^\univ \rightleftarrows (n,n+1)\Cat^\univ $$
factors as
$$\tau_{\leq n}: \infty\Cat^\univ \simeq {_{\infty\Cat^\univ} \Cat^\univ} \xrightarrow{(\tau_{\leq n-1})_!} {_{(n-1,n)\Cat^\univ} \Cat^\univ} \simeq (n,n+1)\Cat^\univ .$$ 
    
\end{remark}

\begin{remark}

For $n=0$ the localization $$ \tau_{\leq 0}: \infty\Cat^\univ \rightleftarrows (0,1)\Cat^\univ $$ to the $\infty$-category of partially ordered sets, 
sends an $\infty$-category $\mC$ to the following partially ordered sets:
the set of equivalence classes of objects of $\mC$ modulo the equivalence relation that two equivalence classes represented by objects $A, B $ of $\mC$ are equivalent if and only if there are morphisms $A \to B $ and $B \to A.$
Then $\pi_0(\mC)$ is a partially ordered set via the relation
$[A] \leq [B]$ if and only if there is a morphism $A \to B$ in $\mC.$

For $n=1$ the localization $$ \tau_{\leq 1}: \infty\Cat^\univ \rightleftarrows (1,2)\Cat^\univ $$ 
sends an $\infty$-category $\mC$ to the following category enriched in partially ordered sets: the objects are those of $\mC$ and 
the morphism partially ordered sets between $A,B \in \mC$
are equivalence classes of morphisms $A \to B$ in $\mC$
modulo the equivalence relation that $F \sim G$ if and only if 
$F=G$ in $\tau_{\leq 0}(\Mor_\mC(A,B)),$ i.e. there are 2-morphisms $F \to G$
and $G \to F$ in $\mC.$
In particular, two objects $A, B \in \mC$ are equivalent in $\tau_{\leq 1}(\mC)$
if and only if there are morphisms $F: A \to B$ and $G:  B \to A$
and 2-morphisms $\alpha: \id \to G F, \beta: GF \to \id, \gamma: \id \to FG, \delta: FG \to \id$ in $\mC.$

Similarly, two objects $A, B \in \mC$ are equivalent in $\tau_{\leq 2}(\mC)$
if and only if there are morphisms $F: A \to B$ and $G:  B \to A$
and 2-morphisms $\alpha: \id \to G F, \beta: GF \to \id, \gamma: \id \to FG, \delta: FG \to \id$ in $\mC$ and 3-morphisms
$\id \to \beta \alpha, \id \to \alpha \beta, \id \to \delta \gamma, \id \to \gamma \delta. $

\end{remark}

\begin{example}\label{extop}

Let $X$ be a space.
It follows by induction on $n \geq 0$ that also $ \tau_{\leq n}(X)$ is a space, where we use that the functor $X \to \tau_{\leq n}(X)$ is $n$-connected so that every morphism of $\tau_{\leq n}(X)$ lifts to $X.$
Indeed, the poset $ \tau_{\leq 0}(X)$ is a set since every morphism lifts to the space $X$.
By induction hypothesis, $ \tau_{\leq n}(X)$ is a 1-category, whose morphisms
all lift to $X$ and so are equivalences.

Consequently, the Postnikov tower of a space is a tower of spaces, which agrees with the classical Postnikov tower since $n$-truncated ($n$-connected) functors between spaces agree with the topological counterparts.
    
\end{example}

\begin{example}\label{loopfree}

Let $X$ be a directed $\infty$-category, which means that any two objects are equivalent if there is a morphism $A \to B$ and $B \to A$,
and the same for iterated morphism $\infty$-categories.
Examples of directed $\infty$-categories are Steiner $\infty$-categories and more generally loopfree gaunt $\infty$-categories.

Then tautologically, the canonical surjection $
\pi_0(\iota_0(X)) \to \pi_0(\iota_0(\tau_{\leq 0}(X))) = \iota_0(\tau_{\leq 0}(X)) $ is a bijection so that $\tau_{\leq 0}(X)$ is the set of equivalence classes of objects of $X$. More generally, by induction on $n \geq 0$ 
the canonical essentially surjective functor $$\xi_X^n: \Ho_n(\iota_n(X)) \to \Ho_n(\iota_n(\tau_{\leq n}(X))) = \iota_n(\tau_{\leq n}(X)) $$ is an equivalence since it induces on morphism $\infty$-categories between $A,B \in X$
the functor 
$$\xi_{\Mor_X(A,B)}^{n-1}: \Ho_{n-1}(\iota_{n-1}(\Mor_X(A,B))) \to \Ho_{n-1}(\iota_{n-1}(\tau_{\leq n-1}(\Mor_X(A,B)))), $$ 
which is an equivalence by induction hypothesis.

So the underlying $(n,n)$-category of the $(n,n+1)$-category
$\tau_{\leq n}(X)$, which arises by forgetting the partial order on $n$-times iterated morphism $\infty$-categories, is $ \Ho_n(\iota_n(X))$,
the homotopy $(n,n)$-category of $X.$

In other words, the homotopy $(n,n)$-category of $X$ refines to an
$(n,n+1)$-category, i.e. the $n$-times iterated morphism $\infty$-categories carry a partial order, and the Postnikov tower looks like
$$ ... \to \Ho_2(\iota_2(X)) \to \Ho_1(\iota_1(X)) \to \Ho_0(\iota_0(X)).$$

\end{example}

\begin{proposition}\label{pullcon}
Let $n \geq -2.$
Let $X \to S, \ Y \to S$ be functors and
$\phi: X \to Y$ a functor over $S.$
The following are equivalent:

\begin{enumerate}[\normalfont(1)]\setlength{\itemsep}{-2pt}
\item The functor $\phi: X \to Y$ is $n$-connected ($n$-truncated).
    
\item For every functor $T \to S$ the induced functor $T \times_S \phi: T \times_S X \to T \times_S Y$ is $n$-connected ($n$-truncated).
    
\end{enumerate}
    
\end{proposition}

\begin{proof}

(2) trivially implies (1) taking $T \to S$ the identity of $S.$
We prove that (1) implies (2) by induction over $n \geq -2.$
The statement for $n=-2$ is the statement that equivalences are stable under pullback.
The induction step follows from the fact that morphism $\infty$-categories in pullbacks are the pullback of the respective morphism $\infty$-categories and that essentially surjective functors are stable under pullback.
\end{proof}

Next we study when fibrations of $\infty$-categories are connected and 
truncated.

We first recall the notion of fibration.

\begin{definition}

Let $\phi: X \to Y$ be a functor.
A morphism $A \to B$ of $X$ is $\phi$-cocartesian if for every
$C \in X$ the induced commutative square
\[
\xymatrix{
\Mor_X(B,C) \ar[r]\ar[d] & \Mor_X(A,C) \ar[d]\\
\Mor_Y(\phi(B),\phi(C)) \ar[r] & \Mor_Y(\phi(A),\phi(C))
}
\]
is a pullback square.
    
\end{definition}

\begin{definition}

A functor $\phi: X \to Y$ is a 1-cocartesian fibration
if for every $A \in X$ and morphism $h: \phi(A) \to B$ in $Y$ there is a $\phi$-cocartesian lift $A \to B'$ of $h.$
    
\end{definition}

\begin{definition}

A commutative square
\[
\xymatrix{
X \ar[r]\ar[d]^{\phi} & X' \ar[d]^{\phi'} \\
Y \ar[r] & Y'
}
\]
is a map of 1-cocartesian fibrations
if both vertical functors are 1-cocartesian fibrations and the top functor preserves cocartesian morphisms.
    
\end{definition}

\begin{definition}
Let $n \geq 1.$
A functor $\phi: X \to Y$ is an $n$-anticocartesian fibration
if it is a 1-cocartesian fibration, for every $A,B \in X$
the induced functor $$\Mor_X(A,B) \to \Mor_Y(\phi(A), \phi(B))$$ is an
$n-1$-anticocartesian fibration and for any morphisms
$A' \to A, B \to B'$ the induced commutative square
\[
\xymatrix{
\Mor_X(A,B) \ar[r]\ar[d] & \Mor_X(A',B') \ar[d]\\
\Mor_Y(\phi(A),\phi(B)) \ar[r] & \Mor_{Y}(\phi(A'),\phi(B'))
}
\]
is a morphism of $n-1$-anticocartesian fibrations.

A commutative square 
\[
\xymatrix{
X \ar[r]\ar[d]^{\phi} & X' \ar[d]^{\phi'} \\
Y \ar[r] & Y'
}
\]
is a map of $n$-anticocartesian fibrations if it is a map of
$1$-cocartesian fibrations and induces on morphism $\infty$-categories a 
map of $n-1$-anticocartesian fibrations.

\end{definition}

\begin{definition}

A functor $\phi: X \to Y$ is an anticocartesian fibration if it is an $n$-anticocartesian fibration for every $n \geq 1.$
    
\end{definition}

\begin{definition}Let $n \geq 1.$

\begin{enumerate}[\normalfont(1)]\setlength{\itemsep}{-2pt}
\item A functor $\phi: X \to Y$ is a cocartesian fibration 
if $\phi^\co: X^\co \to Y^\co$ is an anticocartesian fibration.

\item A functor $\phi: X \to Y$ is a cartesian fibration 
if $\phi^\op: X^\op \to Y^\op$ is an anticocartesian fibration.

\item A functor $\phi: X \to Y$ is an anticartesian fibration 
if $\phi^{\coop}: X^{\coop}\to Y^{\coop}$ is an anticocartesian fibration.

\end{enumerate}

\end{definition}

\begin{lemma}\label{fibequ}

Let $0 \leq n \leq \infty.$
Let $X \to S, \ Y \to S$ be $n$-anticocartesian fibrations and
$\phi: X \to Y$ a map of $n$-anticocartesian fibrations over $S.$
The following are equivalent:
\begin{enumerate}[\normalfont(1)]\setlength{\itemsep}{-2pt}

\item The functor $\phi: X \to Y$ is an $n$-equivalence.

\item For every object $s \in S$ the induced functor
$X_s \to Y_s$ on the fiber over $s$ is an $n$-equivalence.

\end{enumerate}

\end{lemma}

\begin{proof}

The case of $n=\infty$ follows from the finite case and \cref{streqchar}.

(1) clearly implies (2) since $n$-equivalences are stable under pullback as one sees by induction.

We prove that (2) implies (1) by induction on $n \geq 0$.
For $n=0$ we have to see that $\phi$ induces a bijection on equivalence classes of objects if $\phi$ induces fiberwise a bijection on equivalence classes of objects.
If $\phi$ is fiberwise essentially surjective, then $\phi$ is essentially surjective.
We prove that $\phi$ is essentially injective if $\phi$ is fiberwise essentially injective: let $A,B \in \X$ such that $\tau: \phi(A) \simeq \phi(B).$
The latter equivalence lies over an equivalence $\sigma: s \simeq t$ in $\rS.$
Then $(A,\id_s), (B, \sigma) \in X_s$ and $(\tau,\sigma)$ determines an equivalence
$(\phi(A),\id_s), (\phi(B), \sigma) \in Y_s$ between the images in $Y_s$.
Thus $(A,\id_s) \simeq (B, \sigma) $ in $X_s$ so that $A \simeq B$ in $X.$

This proves the induction start. 
Next we prove the induction step.
We assume the statement holds for $n-1$ and let $\phi: X \to Y$ be a functor
such that for every object $s \in S$ the induced functor $X_s \to Y_s$ on the fiber over $s$ is an $n$-equivalence.

We have to see that $\phi: X \to Y$ is an $n$-equivalence.
Since every $n$-equivalence is a 0-equivalence, the induction start implies that 
$\phi: X \to Y$ is a $0$-equivalence. Thus $\phi: X \to Y$ is an $n$-equivalence if and only if for every $A,B \in X$ the induced functor
$\rho: \Mor_X(A,B) \to \Mor_Y(\phi(A), \phi(B))$ is an $n-1$-equivalence. Let $s,t \in S$ be the images of $A,B$.
The functor $\rho$ is a map of $n-1$-anticocartesian fibrations over $\Mor_S(s,t)$
and therefore by induction hypothesis an $n-1$-equivalence if it induces an $n-1$-equivalence on the fiber over every morphism $\alpha : s \to t$ in $\rS.$
The functor $\rho$ induces on the fiber over $\alpha : s \to t$ the functor
$\rho: \Mor_{X_s}(\alpha_!(A),B) \to \Mor_{Y_s}(\phi(\alpha_!(A)), \phi(B))$,
which is an $n-1$-equivalence by assumption.
\end{proof}

\begin{theorem}\label{fiberwiseeqq}
Let $n \geq -2.$
Let $X \to S, \ Y \to S$ be anticocartesian fibrations and
$\phi: X \to Y$ a map of anticocartesian fibrations over $S.$
The following are equivalent:
\begin{enumerate}[\normalfont(1)]\setlength{\itemsep}{-2pt}

\item The functor $\phi: X \to Y$ is $n$-connected ($n$-truncated).

\item For every object $s \in S$ the induced functor
$X_s \to Y_s$ on the fiber over $s$ is $n$-connected ($n$-truncated).

\end{enumerate}

\end{theorem}

\begin{proof}

(1) implies (2) by \cref{pullcon}.
We prove by induction on $n \geq -2$ that (2) implies (1).

We start with $n=-2$. Since every functor is $-2$-connected, there is nothing to show for the connected version.
On the other hand, the $-2$-truncated functors are precisely the equivalences. So for the truncated version we have to see that $\phi$
is an equivalence if $\phi$ induces fiberwise equivalences.
This is \cref{fibequ}.

So we have proven the induction start.
Next we prove the induction step.

Let $n \geq -1$.
We assume the statement holds for $n-1$ and let $\phi: X \to Y$ be a functor
such that for every object $s \in S$ the induced functor $X_s \to Y_s$ on the fiber over $s$ is $n$-connected ($n$-truncated).

We have to see that $\phi: X \to Y$ is $n$-connected ($n$-truncated).
The functor $\phi: X \to Y$ is essentially surjective if and only if for every object $s \in S$ the induced functor $X_s \to Y_s$ on the fiber over $s$ is essentially surjective.
Hence by \cref{induco} it suffices to show that for every $A,B \in X$ the induced functor
$$\rho: \Mor_X(A,B) \to \Mor_Y(\phi(A), \phi(B))$$ is $n-1$-connected ($n-1$-truncated).
Let $s,t \in S$ be the images of $A,B$.
The functor $\rho$ is a map of $n-1$-anticocartesian fibrations over $\Mor_S(s,t)$
and therefore by induction hypothesis an $n-1$-connected ($n-1$-truncated) functor
if it induces an $n-1$-connected ($n-1$-truncated) functor on the fiber over every morphism $\alpha : s \to t$ in $\rS.$
The functor $\rho$ induces on the fiber over $\alpha : s \to t$ the functor
$$\rho: \Mor_{X_s}(\alpha_!(A),B) \to \Mor_{Y_s}(\phi(\alpha_!(A)), \phi(B)),$$
which is a $n-1$-connected ($n-1$-truncated) functor by assumption.
\end{proof}

%
%





\subsection{Connections}\label{connection}

\begin{definition}Let $X$ be an $\infty$-category.
Let $n \geq -1.$
By induction on $n$ we define an $n$-connection in $X.$

\begin{enumerate}[\normalfont(1)]\setlength{\itemsep}{-2pt}
\item A $-1$-connection in $X$ is an object of $X$.

\item A $0$-connection in $X$ is a collection $(A,B,f:A \to B,g:B \to A)$ of objects $A,B \in X$ and $1$-morphisms $f,g$ in $X$.
We say that $(A,B, f:A \to B,g:B \to A)$ is a $0$-connection between $A$ and $B$. If a $0$-connection exists between $A$ and $B$ in $X$, we say that $A$ and $B$ are $0$-connected.

\item An $n+1$-connection in $X$ is a triple $(A,B,f,g,\sigma, \tau),$
where $(A,B,f:A \to B,g:B \to A)$ is a $0$-connection and $\sigma$ is an $n$-connection between $\id_A$ and $g f$ in $\Mor_X(A,B)$ and $\tau$ is an $n$-connection between $\id_B$ and $ f g$ in $\Mor_X(B,A).$
We say that $(A,B,f,g,\sigma, \tau)$ is an $n+1$-connection in $X$ between $A, B.$

\end{enumerate}

\end{definition}


    



\begin{notation}

Let $D^1:= \bD^1$ and let $D^2$ be the 2-category 
classifying a pair of objects $A,B$, morphisms $f: A \to B, g: B \to A $ and $2$-morphims $\id \to g f, g f \to \id, \id \to f g, f g \to \id. $

\end{notation}

	
\begin{remark}\label{map0}.

There is a canonical functor $\bD^1 \to D^2$ selecting the generator of $f, $ which is $-1$-connected.

\end{remark}

The following definition is used in \cite[Construction 5.35.]{ozornova2026core}:

\begin{notation}Let $n \geq 2.$
Let $D^{n+1}$ be the pushout
\[
\xymatrix{
\coprod_{4^{n-1}} S^{n-1}(\bD^1) \ar[r]\ar[d] & D^{n} \ar[d] \\
\coprod_{4^{n-1}} S^{n-1}(D^2) \ar[r] & D^{n+1},
}
\]
where the left hand functor is induced by the functor of \cref{map0}.

Let $D^\infty$ be the colimit of the sequence 
$$ D^0 \to D^1 \to ...$$

\end{notation}







\begin{remark}

Since the functor $D^1 \to D^2$ is $-1$-connected, the functor $D^n \to D^{n+1} $ is $n-2$-connected.
Since $D^1 $ is $-1$-connected and $D^2$ is $0$-connected,
by induction on $n $ the $\infty$-category $D^n $ is $n-2$-connected.
Hence $D^\infty$ is $\infty$-connected.
    
\end{remark}



    












    


    

\begin{remark}Let $ n \geq 0$ and $X$ an $\infty$-category.
An $n$-connection in $X$ is precisely a functor $D^{n+1} \to X.$

\end{remark}

\begin{definition}Let $X$ be an $\infty$-category.
An $\infty$-connection in $X$ is a functor $D^\infty \to X.$
\end{definition}

\begin{definition}Let $ 0 \leq n \leq \infty.$
Two objects of an $\infty$-category $X$ are $n$-connected if there is an 
$n$-connection between these objects.
	
\end{definition}

\begin{lemma}\label{liftu}

Let $\phi: X \to Y$ be an $\infty$-connected functor.
The induced functor
$$\Fun^\oplax(D^\infty,X) \to \Fun^\oplax(D^\infty,Y)$$ is $-1$-connected.


\end{lemma}

\begin{proof}

The induced functor
$\Fun^\oplax(D^\infty,X) \to \Fun^\oplax(D^\infty,Y)$ identifies with the functor
$$\lim_{n \geq 1}\Fun^\oplax(D^n,X) \to \lim_{n \geq 1}\Fun^\oplax(D^n,Y).$$
The functor $\Fun^\oplax(D^1,X) \to \Fun^\oplax(D^1,Y)$
is essentially surjective since $D^1= \bD^1$ and $\phi: X \to Y$ is $0$-connected and so essentially surjective and full. 
Consequently, it suffices to prove that the functor
$$ \Fun^\oplax(D^{n+1},X) \to \Fun^\oplax(D^{n+1},Y) \times_{\Fun^\oplax(D^n,Y)} \Fun^\oplax(D^n,X)   $$ is essentially surjective. 
This functor is over $\Fun^\oplax(D^n,X)$ and it suffices to see that this functor induces an essentially surjective functor on the fiber over every object of $\Fun^\oplax(D^n,X).$
It induces on the fiber over every object of $\Fun^\oplax(D^n,X)$
the functor $$\{ \alpha\} \times_{\Fun^\oplax(D^n,X)} \Fun^\oplax(D^{n+1},X) \to \{ \phi \alpha\} \times_{\Fun^\oplax(D^n,Y)} \Fun^\oplax(D^{n+1},Y).$$
This functor identifies with the functor 
$$\{ \alpha'\} \times_{\prod_{4^{n-1}} \Fun^\oplax(S^{n-1}(D^1),X)} \prod_{4^{n-1}}\Fun^\oplax(S^{n-1}(D^2),X) \to $$$$ \{ \phi \alpha'\} \times_{\prod_{4^{n-1}} \Fun^\oplax(S^{n-1}(D^1),Y)} \prod_{4^{n-1}}\Fun^\oplax(S^{n-1}(D^2),Y).$$

This functor further identifies with the functor 
$$\{ \alpha'\} \times_{\prod_{4^{n-1}} \Fun^\oplax_{\partial\bD^n/}(S^{n-1}(D^1),X)} \prod_{4^{n-1}}\Fun^\oplax_{\partial\bD^n/}(S^{n-1}(D^2),X) \to$$$$ \{ \phi \alpha'\} \times_{\prod_{4^{n-1}} \Fun^\oplax_{\partial\bD^n/}(S^{n-1}(D^1),Y)} \prod_{4^{n-1}}\Fun^\oplax_{\partial\bD^n/}(S^{n-1}(D^2),Y).$$
By adjointness the latter functor identifies with the functor 
$$\{ \alpha''\} \times_{\prod_{4^{n-1}} \Fun^\oplax_{\partial\bD^1/}(D^1,\Mor^{n-1}(X))} \prod_{4^{n-1}}\Fun^\oplax_{\partial\bD^1/}(D^2,\Mor^{n-1}(X)) \to$$$$ \{ \phi \alpha''\} \times_{\prod_{4^{n-1}} \Fun^\oplax_{\partial\bD^1/}(D^1,\Mor^{n-1}(Y))} \prod_{4^{n-1}}\Fun^\oplax_{\partial\bD^1/}(D^2,\Mor^{n-1}(Y)).$$
Consequently, it suffices to see that for every $\infty$-connected functor $\phi: X \to Y$ and morphism $ f: A \to B $ in $X$ the induced functor 
$$\{ f\} \times_{\Fun^\oplax_{\partial\bD^1/}(D^1,X)} \Fun^\oplax_{\partial\bD^1/}(D^2,X) \to \{ \phi(f)\} \times_{\Fun^\oplax_{\partial\bD^1/}(D^1,Y)} \Fun^\oplax_{\partial\bD^1/}(D^2,Y)$$ is essentially surjective.

This is clear since for every morphism $ g: \phi(B) \to \phi(A) $ and $2$-morphisms $\alpha: \id \to g \phi(f), \beta: g \phi(f) \to \id, \gamma: \id \to \phi(f) g, \delta: \phi(f) g \to \id$
there is a morphism $g': B \to A$ lying over $g: \phi(B) \to \phi(A)$
and $2$-morphisms $$\alpha': \id \to g' f, \beta': g' f \to \id, \gamma': \id \to f g', \delta': f g' \to \id$$ lying over $$\alpha: \id \to g \phi(f), \beta: g \phi(f) \to \id, \gamma: \id \to \phi(f) g, \delta: \phi(f) g \to \id.$$
\end{proof}
	
\begin{proposition}Let $ 0 \leq n \leq \infty.$
An $\infty$-category is $n$-connected if and only if it is non-empty and any two parallel $m$-morphisms for $0 \leq m \leq n$ are $n-m$-connected.
		
\end{proposition}

\begin{proof}

We show first that an $\infty$-category is $n$-connected if it is non-empty and any two parallel $m$-morphisms for $0 \leq m \leq n$ are $n-m$-connected.
The case of $n=\infty$ follows immediately from the finite case since
an $\infty$-category is $\infty$-connected if and only if it is $n$-connected for every $n \geq 0.$
We prove the statement for $n$ by induction on $n \geq 0.$
The case $n=0$ follows immediately from the definitions. 
We assume the statement holds for $n$. We prove the statement for $n+1.$
Let $X$ be a non-empty $\infty$-category such that any two parallel $m$-morphisms in $X$ for $0 \leq m \leq n+1 $ are $n+1-m$-connected.
Then for every $A,B \in X$ any two parallel $m-1$-morphisms in $\Mor_X(A,B)$ for $1 \leq m \leq n+1 $ are $n+1-m$-connected.
This is equivalent to say that for every $A,B \in X$ any two parallel $m$-morphisms in $\Mor_X(A,B)$ for $0 \leq m \leq n $ are $n-m$-connected.
So by induction hypothesis for every $A,B \in X$ the $\infty$-category $\Mor_X(A,B)$ is $n$-connected. Thus $X$ is $n$+1-connected since $X$ is non-empty. 

Next we prove the converse. We show next that in every $n$-connected $\infty$-category any two parallel $m$-morphisms for $0 \leq m \leq n$ are $n-m$-connected.
We first assume that $n < \infty.$
We prove the statement by induction on $n \geq 0.$
The case $n=0$ follows immediately from the definitions. 
We assume the statement holds for $n$. We prove the statement for $n+1.$
Let $X$ be an $n+1$-connected $\infty$-category.
Then every morphism $\infty$-category of $X$ is $n$-conneted.
Thus by induction hypothesis any two parallel $m$-morphisms for $0 \leq m \leq n$ in any morphism $\infty$-category of $X$ are $n-m$-connected.
This is equivalent to say that any two parallel $m+1$-morphisms for $0 \leq m \leq n$ in $X$ are $n-m= n+1-(m+1)$-connected.
Replacing $m+1$ by $m$ we find that any two parallel $m$-morphisms for $1 \leq m \leq n+1$ in $X$ are $n+1-m$-connected.
It remains to see that any two objects of any $n+1$-connected $\infty$-category $X$ are $n+1$-connected.
Let $A, B \in X $. Since $X$ is $n+1$-connected, it is in particular an 
$0$-connected $\infty$-category. Hence there is a $0$-connection $(A,B, f:A \to B, g: \B \to A)$ between $A,B.$ 
Since $X$ is $n+1$-connected, the morphism $\infty$-categories
$\Mor_X(A,A), \Mor_X(B,B)$ are $n$-connected. So by induction hypothesis,
the objects $\id, g f$ in $ \Mor_X(A,A)$ are $n$-connected, and the morphisms $\id, f g$ in $ \Mor_X(B,B) $ are $n$-connected.
We obtain an $n+1$-connection between $A,B.$ So $A,B$ are $n+1$-connected.

It remains to see the case $n=\infty.$
We have to see that in every $\infty$-connected $\infty$-category any two parallel $m$-morphisms for $0 \leq m$ are $\infty $-connected.
We prove this statement by induction on $m \geq 0.$
We assume the statement for $m.$
Let $X$ be an $\infty$-connected $\infty$-category. We want to see that any two parallel $m+1$-morphisms are $\infty $-connected.
This holds if for every $A,B \in X$ any two parallel $m$-morphisms in $\Mor_X(A,B) $ are $\infty $-connected. The latter holds by induction hypothesis since $\Mor_X(A,B) $ is $\infty$-connected.
So it remains to see the induction start $n=0.$
We have to see that in every $\infty$-connected $\infty$-category $X$ any two objects are $\infty $-connected.
This follows from Lemma \cref{liftu} applied to the $\infty$-connected functor $X \to *.$ 
\end{proof}

\begin{definition}
An $\infty$-category $X$ is hypercomplete (or $\infty$-truncated) if every $\infty$-connection in $X$ and all iterated morphism $\infty$-categories of $X$ is an equivalence.

\end{definition}

\begin{remark}
An $\infty$-category is hypercomplete if and only if it is local with respect to all categorical suspensions of the functor
$D^\infty \to *.$
Indeed, if an $\infty$-category is local with respect to all categorical suspensions of the functor
$D^\infty \to *,$ it is evidently hypercomplete.
On the other hand, if an $\infty$-category $X$ is hypercomplete,
then for every $n \geq 0$ every functor $\Sigma^n(D^\infty) \to X$ lands in $\iota_n(X).$ So the induced map
$\Map_{\infty\Cat}(*,X) \to \Map_{\infty\Cat}(\Sigma^n(D^\infty),X)$ indentifies with the functor
$\Map_{\infty\Cat}(*,\iota_n(X)) \to \Map_{\infty\Cat}(\Sigma^n(D^\infty),\iota_n(X)).$ This is an equivalence 
because $\iota_n(X)$ is an $n$-category.
    
\end{remark}

\begin{lemma}\label{hypereq}
An $\infty$-connected functor starting at a hypercomplete $\infty$-category is an equivalence.

\end{lemma}

\begin{proof}

We prove by induction on $n \geq 0$ that every $\infty$-connected functor $X \to Y$  starting at a hypercomplete $\infty$-category
is an $n$-equivalence. Then \cref{streqchar} gives the result.
Let $n \geq 0$. We assume the statement for $n.$
An $\infty$-connected functor  starting at a hypercomplete $\infty$-category induces on morphism $\infty$-categories an
$\infty$-connected functor  starting at a hypercomplete $\infty$-category, which is an $n$-equivalence by induction hypothesis.
Thus it suffices to see that every $\infty$-connected functor  starting at a hypercomplete $\infty$-category is a $0$-equivalence.
Every $\infty$-connected functor is essentially surjective.
So it suffices to see that every $\infty$-connected functor $\phi: X \to Y$  starting at a hypercomplete $\infty$-category
induces an injection on equivalence classes of objects. 
Let $A, B \in X$ such that $\phi(A) \simeq \phi(B).$
Since $\phi$ is $\infty$-connected, by \cref{liftu} there is an $\infty$-connection in $X$ between $A $ and $ B$ lying over the equivalence $\phi(A) \simeq \phi(B).$ 
Hence $A \simeq B$ by hypercompleteness of $X.$
\end{proof}

\begin{corollary}
An $\infty$-category is hypercomplete if and only if it is local with respect to every $\infty$-connected functor.
    
\end{corollary}

\begin{proof}

Since the canonical functor $D^\infty \to *$ and all suspensions are $\infty$-connected, every $\infty$-category, which is local with respect to every $\infty$-connected functor, is also hypercomplete.

Conversely, let $X$ be hypercomplete. 
The functor $X \to *$ factors as an $\infty$-connected functor
$X \to X'$ followed by a functor $X' \to *$ that has the right lifting property with respect to any $\infty$-connected functor so that $X'$ is local with respect to every $\infty$-connected functor. By \cref{hypereq} the $\infty$-connected functor
$X \to X'$ is an equivalence.
\end{proof}



\begin{corollary}
The $\infty$-category $\infty\Cat^{\land}$ of hypercomplete $\infty$-categories is the accessible localization of $\infty\Cat$ obtained by inverting the $\infty$-connective functors.
\end{corollary}

\subsection{Oriented categories}







 



    

In the following we consider categories enriched in the Gray tensor product.


\begin{definition}Let $\boxtimes$ denote the Gray tensor product monoidal structure on $\infty\Cat^\univ$.
\begin{enumerate}[\normalfont(1)]\setlength{\itemsep}{-2pt}
\item An oriented category is a univalent category right enriched in $(\infty\Cat^\univ,\boxtimes)$, and an oriented functor is a right $(\infty\Cat^\univ,\boxtimes)$-enriched functor.

\item An antioriented category is a univalent category left enriched in $(\infty\Cat^\univ,\boxtimes)$, and an antioriented functor is a left $(\infty\Cat^\univ,\boxtimes)$-enriched functor

\item A bioriented category is a univalent category bienriched in 
$((\infty\Cat^\univ,\boxtimes), (\infty\Cat^\univ,\boxtimes))$, and a bioriented functor is a $((\infty\Cat^\univ,\boxtimes), (\infty\Cat^\univ,\boxtimes))$-enriched functor.
\end{enumerate}
\end{definition}

\begin{remark}

By definition a bioriented category is equivalently a univalent category right enriched in
$(\infty\Cat^\univ, \boxtimes)^\rev \ot (\infty\Cat^\univ, \boxtimes) $
and a univalent category left enriched in 
$(\infty\Cat^\univ, \boxtimes) \ot (\infty\Cat^\univ, \boxtimes)^\rev. $

\end{remark}




	

\begin{example}
The Gray monoidal structure on $\infty\Cat^\univ$ is closed and so endows $\infty\Cat^\univ$ as bienriched in $(\infty\Cat^\univ,\boxtimes).$ This way we see
$\infty\Cat^\univ$ as a large bioriented category, which we denote by $ \infty\fcat^\univ.$
Every full subcategory of $\infty\Cat^\univ$ inherits the structure of a bioriented category.

\end{example}	



\begin{notation}We adopt the following notation for the categories of oriented, antioriented, and bioriented categories.
\begin{enumerate}[\normalfont(1)]\setlength{\itemsep}{-2pt}
\item Let $\Cat\boxtimes$
be the 2-category of oriented categories and oriented functors.	
\item Let $\boxtimes\Cat$
be the 2-category of antioriented categories and antioriented functors.
\item Let $\boxtimes\Cat\boxtimes$
be the 2-category of bioriented categories and bioriented functors.
\end{enumerate}
\end{notation}

\begin{notation}
We fix the following notation:

\begin{enumerate}[\normalfont(1)]\setlength{\itemsep}{-2pt}
\item Let $\mC,\mD \in \boxtimes\Cat$. Let $${\boxtimes\Fun}(\mC,\mD)$$ be the category of antioriented functors $\mC \to \mD.$	
		
\item Let $\mC,\mD \in \Cat\boxtimes$. Let $${\Fun\boxtimes}(\mC,\mD)$$ be the category of oriented functors $\mC \to \mD.$	
		
\item Let $\mC,\mD \in \boxtimes\Cat\boxtimes $. Let $${\boxtimes\Fun\boxtimes}(\mC,\mD)$$ be the category of bioriented functors $\mC \to \mD.$		

\end{enumerate}
\end{notation}


    

\begin{remark}If $\mC$ is an oriented category, then $\mC$ has an underlying category.
However, $\mC$ does not have an underlying $\infty$-category, because $\mC$ does note satisfy the interchange law; instead, $\mC$ satisfies an oriented version of the interchange law in the following sense:
For $1$-cells $f,f':X\to Y$ and $g,g':Y\to Z$ and $2$-cells $\alpha:f\Rightarrow f'$ and $\beta:g\Rightarrow g'$, viewed as $1$-cells $\bD^1\to\RMor_{\mC}(X,Y)$ and $\bD^1\to\RMor_{\mC}(Y,Z)$, we obtain map $\bD^1\boxtimes\bD^1\to\RMor_{\mC}(X,Y)\boxtimes\RMor_{\mC}(Y,Z)\to\RMor_{\mC}(X,Z)$, which we can picture as an oriented square
\[
\xymatrix{
g\circ f\ar[r]^{\id_g\circ\alpha}\ar[d]_{\beta\circ\id_f} & g\circ f'\ar[d]^{\beta\circ \id_{f'}}\\
g'\circ f\ar[r]_{\id_{g'}\circ\alpha}\ar@{=>}[ru] & g'\circ f'.
}
\]
The failure of this square (and its higher dimensional analogues, coming from mapping higher dimensional cells into the morphism $\infty$-categories to commute is precisely the obstruction to $\mC$ being an $\infty$-category (viewed as an oriented category).
\end{remark}

\begin{notation}
We refer to adjunctions of oriented, antioriented, and bioriented categories as oriented, antioriented, or bioriented adjunctions.
\end{notation}

\begin{remark}\label{adj2}

\begin{enumerate}[\normalfont(1)]\setlength{\itemsep}{-2pt}
\item An antioriented (oriented) functor $\mC \to \mD$ admits a right adjoint if and only if it preserves left (right) tensors and the underlying unenriched functor admits a right adjoint.

\item A bioriented functor $\mC \to \mD$ admits a right adjoint if and only if it preserves left and right tensors and the underlying unenriched functor admits a right adjoint.

\item An antioriented (oriented) functor $\mC \to \mD$ admits a left adjoint if and only if it preserves left (right) cotensors and the underlying unenriched functor admits a left adjoint.

\item A bioriented functor $\mC \to \mD$ admits a left adjoint if and only if it preserves left and right cotensors and the underlying (non-enriched) functor admits a left adjoint.

\end{enumerate}

\end{remark}

	

\begin{definition}
An oriented, antioriented, bioriented category is presentable if the respective left, right, bienriched category is presentable (\cref{present}).
\end{definition}


Next we define the appropriate notions of opposite oriented category.

\begin{definition}
    
Let
\begin{align*}
    &(-)^\circ: {\Cat\boxtimes} \simeq {\boxtimes\Cat}\\
    &(-)^\circ: {\boxtimes\Cat} \simeq {\Cat\boxtimes}\\
    &(-)^\circ:  {\boxtimes\Cat\boxtimes} \simeq {\boxtimes\Cat\boxtimes}
\end{align*}
be the opposite enriched category involutions.	
	
\end{definition}

\begin{definition}
We define the following involutions, which reverse the even dimensional cells.
Note that we are also forced to reverse the orientation.
The equivalences of \cref{dua} give rise to the equivalences
\begin{align*}
&(-)^\co:= (-)^\op_!: {\boxtimes\Cat} \simeq {\Cat\boxtimes}\\
&(-)^\co:= (-)^\op_!: {\Cat\boxtimes} \simeq {\boxtimes\Cat}\\
&(-)^\co:= ((-)^\op, (-)^\op)_!: {\boxtimes\Cat\boxtimes} \simeq {\boxtimes\Cat\boxtimes}.
\end{align*}
\end{definition}

\begin{definition}
We define the following involutions, which reverse the odd dimensional cells.
Note that we are also forced to keep the orientation.
\begin{align*}
&(-)^\op:= (-)^\circ\circ (-)^\co_!: {\boxtimes\Cat} \simeq {\boxtimes\Cat}\\
&(-)^\op:= (-)^\circ\circ (-)^\co_!: {\Cat\boxtimes} \simeq {\Cat\boxtimes}\\
&(-)^\op :=(-)^\circ \circ ((-)^\co, (-)^\co)_!: {\boxtimes\Cat\boxtimes} \simeq {\boxtimes\Cat\boxtimes}
\end{align*}
\end{definition}

\begin{definition}
Combining the latter two types of involutions gives rise to the following sort of involution, which reverses all cells:
\begin{align*}
&{(-)^{\co\op}}:= {(-)^{\co}} \circ {(-)^{\op}} \simeq {(-)^{\op}} \circ{(-)^{\co}}:{\boxtimes\Cat} \simeq {\Cat\boxtimes}\\
&{(-)^{\co\op}}:= {(-)^{\co}} \circ {(-)^{\op}} \simeq {(-)^{\op}} \circ {(-)^{\co}}:{\Cat\boxtimes} \simeq {\boxtimes\Cat}\\
&{(-)^{\co\op}}:= {(-)^{\co}} \circ {(-)^{\op}} \simeq {(-)^{\op}} \circ {(-)^{\co}}: {\boxtimes\Cat\boxtimes} \simeq {\boxtimes\Cat\boxtimes},
\end{align*}
where all equivalence are involutions.
\end{definition}



The equivalences of \cref{dua} gives rise to the following equivalences of bioriented categories:
\begin{corollary}\label{duae}There are canonical equivalences of
bioriented categories:
\begin{align*}  
&(-)^\op: (\infty\Cat^\univ)^\co \simeq \infty\Cat^\univ\\
&(-)^\co: (\infty\Cat^\univ)^\op \simeq (\infty\Cat^\univ)^\circ\\
&(-)^{\co\op} : {^\cop(\infty\Cat^\univ)^{\cop}} \simeq \infty\Cat^\univ
\end{align*}
\end{corollary}	




Next we define oriented pullbacks.

\begin{definition}
Let $\mC$ be an oriented category.
\begin{enumerate}[\normalfont(1)]\setlength{\itemsep}{-2pt}
\item 
The oriented pullback of the diagram $X \to Z\leftarrow Y$ in $\mC$ is a diagram
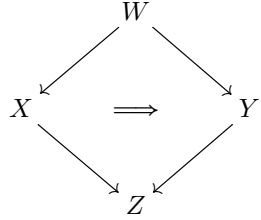
\[
\xymatrix{& W\ar[rd]\ar[ld] &\\
X\ar[rd] & \Longrightarrow & Y\ar[ld]\\
& Z &}
\]
in $\mC$ such that for all objects $T$ of $\mC$ the induced functor
\[
\R\Mor_{\mC}(T,W)\to \R\Mor_{\mC}(T,X) \underset{\R\Mor_{\mC}(T,Z)}{\times}
\Fun^\oplax(\bD^1,\R\Mor_{\mC}(T,Z)) \underset{\R\Mor_{\mC}(T,Z)}{\times} \R\Mor_{\mC}(T,Y)
\]
is an equivalence.
In this case, we will write $X\underset{Z}{\vec{\times}} Y$ or $ Y\underset{Z}{{\cev\times}} X$ for $W.$

\item 
The oriented pushout of the oriented diagram $X \leftarrow Z\to Y$ in $\mC$ is a diagram
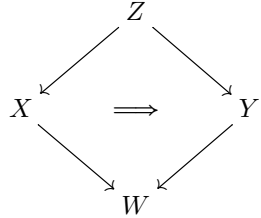
\[
\xymatrix{& Z\ar[rd]\ar[ld] &\\
X\ar[rd] & \Longrightarrow & Y\ar[ld]\\
& W &}
\]
in $\mC$ such that for all objects $T$ of $\mC$ the induced functor
\[
\R\Mor_{\mC}(W,T)\to \R\Mor_{\mC}(X,T) \underset{\R\Mor_{\mC}(Z,T)}{\times}
\Fun^\lax(\bD^1,\R\Mor_{\mC}(Z,T)) \underset{\R\Mor_{\mC}(Z,T)}{\times} \R\Mor_{\mC}(Y,T)
\]
is an equivalence.
In this case, we will write $ X\underset{Z}{{\vec{+}}} Y $ or $Y\underset{Z}{\cev{+}} X$ for $W.$

\end{enumerate}

\end{definition}

\begin{remark}
An oriented pullback square in an oriented category $\mC$ is a diagram $\cube^2\to\mC$ which satisfies the universal property above.
\end{remark}
\begin{definition}
Let $\mC$ be an antioriented category.

\begin{enumerate}[\normalfont(1)]\setlength{\itemsep}{-2pt}

\item The antioriented pushout of the diagram $A\leftarrow C \to B$ in $\mC$ is the oriented pushout in the oriented category $(\mC^\op)^\circ$ of the corresponding diagram.

\item The antioriented pullback of the diagram $A\to C \leftarrow B$ in $\mC$ is the oriented pullback in the oriented category $(\mC^\op)^\circ$ of the corresponding diagram.



\end{enumerate}

\end{definition}














\begin{remark}
Note that $\mC\mapsto\mC^{\op}$ does not change the orientation, whereas the standard involutions $\mC\mapsto\mC^{\co}$, $\mC\mapsto\mC^{\co\op}$, and $\mC\mapsto\mC^\circ$ do change the orientation.
\end{remark}

The next result is \cite[Lemma 4.8.9.]{gepner2025oriented}:

\begin{lemma}\label{adesc}
Let $\mC$ be an oriented category.
\begin{enumerate}[\normalfont(1)]\setlength{\itemsep}{-2pt}
\item Let $A \leftarrow C \to \B$ be morphisms in $\mC$. If $\mC$ admits pushouts and right tensors with $\bD^1$, 
there is a canonical equivalence $$\A\,\underset{C}{\vec{+}}\, \B \simeq \A\!\!\underset{C \otimes \{0\}}{+} \!\!(\C \ot \bD^1) \!\!\underset{C \otimes \{1\}}{+}\!\!\B.$$	

\item Let $A \to C\leftarrow B$ be morphisms in $\mC$. If $\mC$ admits pullbacks and right cotensors with $\bD^1$, 
there is a canonical equivalence $${\A \,\underset{\C}{\vec{\times}}}\, \B \simeq \A \!\!\underset{\C^{\{0\}}}{\times} \!\!{\C^{\bD^1}}\!\!\underset{{\C^{\{1\}}}}{\times}  \!\!\B.$$	
\end{enumerate}

\end{lemma}

Dually, we obtain the following, which is \cite[Corollary 4.8.10.]{gepner2025oriented}:

\begin{corollary}\label{bdesc}
Let $\mC$ be an antioriented category.

\begin{enumerate}[\normalfont(1)]\setlength{\itemsep}{-2pt}

\item 
Let $A \leftarrow C \to \B$ be morphisms in $\mC$. If $\mC$ admits pushouts and left tensors with $\bD^1$, 
there is a canonical equivalence $$\A\,\underset{C}{\bar{\vec{+}}}\, \B \simeq \A \underset{\{0\}\otimes\C}{+} ( \bD^1\ot \C) \underset{\{1\}\otimes\C}{+} \B.$$	

\item Let $A \to C\leftarrow B$ be morphisms in $\mC$. If $\mC$ admits pullbacks and left cotensors with $\bD^1$, 
there is a canonical equivalence $$\A\underset{C}{\bar{\vec{\times}}} \B \simeq \A \underset{^{\{0\}}\C}{\times}{^{\bD^1}\C} \underset{^{\{1\}}\C}{\times} \B.$$	
\end{enumerate}

\end{corollary}

The next result is \cite[Proposition 4.6.4.]{gepner2025oriented}:

\begin{proposition}\label{mormor} Let $\mC$ be an $\infty$-category and $\X,\Y \in \mC$.
There is a canonical equivalence
$$\{X \} \underset{C}{\vec{\times}} \{Y \} \simeq \Mor_\mC(X,Y).$$

\end{proposition}



Finally, we have the following pasting law, which is \cite[Lemma 4.8.14.]{gepner2025oriented}:

\begin{lemma}\label{pasting}

Consider the following diagram in any oriented category $\mC$, where the left hand square is a commutative square:
\[
\begin{tikzcd}
\Q \ar{d} \ar{r} & \P \ar{r}{} \ar{d}[swap]{} & \B  \ar{d}{} \\
\E \ar{r} & \A \ar[double]{ur}{}  \ar{r}[swap]{} & \C
\end{tikzcd}
\]
If the right hand square is an oriented pullback square, the left hand square is a pullback square if and only if the outer square is an oriented pullback square.

\end{lemma}

The next result is \cite[Proposition 4.8.15., Corollary 4.8.16.]{gepner2025oriented}:

\begin{proposition}\label{homs}\label{homso} Let $\F: \mA \to \mC,\G: \mB \to \mC$ be functors and $\A,\A'\in \mA, \B,\B' \in \mB$ and $\sigma: \F(\A)\to \G(\B), \sigma': \F(\A')\to \G(\B')$ morphisms. 

\begin{enumerate}[\normalfont(1)]\setlength{\itemsep}{-2pt}
\item There is a canonical equivalence $$\Mor_{\mA\underset{\mC}{\bar{\vec{\times}}} \mB}((\A,\B, \sigma),(\A',\B', \sigma')) \simeq \Mor_\mA(\A,\A')\underset{\Mor_\mC(\F(\A),\G(\B'))}{\bar{\vec{\times}}} \Mor_\mB(\B,\B').$$

\item There is a canonical equivalence $$\Mor_{{\mA \,\underset{\mC}{\vec{\times}}}\, \mB}((\A,\B, \sigma),(\A',\B', \sigma')) \simeq {\Mor_\mB(\B,\B') \,\underset{\Mor_\mC(\F(\A),\G(\B'))}{\vec{\times}}}\, \Mor_{\mA}(\A,\A').$$

\end{enumerate}

\end{proposition}

We obtain the following corollary, which is \cite[Corollary 4.8.17.]{gepner2025oriented}:

\begin{corollary}\label{homs2}\label{homso2} Let $\mB$ be an $\infty$-category and $\sigma: \A \to \B, \sigma': \A' \to \B'$ morphisms in $\mB$. There are canonical equivalences $$\Mor_{\Fun^\lax(\bD^1,\mB)}(\sigma, \sigma') \simeq \Mor_\mA(\A,\A')\underset{\Mor_\mC(\A,\B')}{\bar{\vec{\times}}} \Mor_\mB(\B,\B'), $$
$$\Mor_{\Fun^\oplax(\bD^1,\mB)}(\sigma, \sigma') \simeq {\Mor_\mB(\B,\B') \,\underset{\Mor_\mC(\A,\B')}{\vec{\times}}}\, \Mor_{\mA}(\A,\A'). 
$$\end{corollary}





We have the following characterization of connected and truncated functors in terms of oriented pullbacks.

\begin{proposition}\label{orientedchar}

Let $ n \geq -1.$

\begin{enumerate}[\normalfont(1)]\setlength{\itemsep}{-2pt}
\item A functor $ X \to Y $ is $n$-truncated 
if and only if the canonical functor
$$ X \overset{\to}{\times}_X X \to X \overset{\to}{\times}_Y X $$
is $n-1$-truncated. A functor is $-2$-truncated 
if and only if it is an equivalence.

\item A functor $ X \to Y $ is $n$-connected
if and only if it is essentially surjective and the canonical functor
$$ X \overset{\to}{\times}_X X \to X \overset{\to}{\times}_Y X $$
is $n-1$-connected.
Every functor is $-2$-connected.





\end{enumerate}

\end{proposition}

\begin{proof}

By definition every functor is $-2$-connected. Moreover 
by definition a functor is $0$-truncated if and only if it is an equivalence.
By definition a functor $\phi: X \to Y $ is $n$-truncated 
if and only if it induces $n-1$-truncated functors on morphism
$\infty$-categories.
By definition a functor $\phi: X \to Y $ is $n$-connected
if and only if it is essentially surjective and induces $n-1$-connected functors on morphism $\infty$-categories.

The canonical functor
$$ X \overset{\to}{\times}_X X \to X \overset{\to}{\times}_Y X $$
is a functor over $X \times X$ that is a map of cartesian fibrations over $X$ projecting to the first factor and a map of cocartesian fibrations over $X$ projecting to the second factor.
Moreover by \cref{mormor} this functor induces on the fiber over every $A,B \in X \times X$ the canonical functor 
$\Mor_X(A,B) \to \Mor_Y(\phi(A), \phi(B)).$
Hence the result follows from \cref{fiberwiseeqq}.  
\end{proof}

\subsection{Truncation and connective covers preserve orientations}

Next we study the interaction of $n$-connected and $n$-truncated functors with oriented pullbacks (\cref{trunfiber}), (\cref{conpull}) and prove that the factorization system of $n$-connected and $n$-trucated functors is compatible with the Gray tensor product (\cref{Grayfacto}).

\begin{proposition}\label{faithpull}

Let $n \geq -2$ and 
\[
\xymatrix{
\mA \ar[r]^f\ar[d]^{\alpha} & \mC \ar[d]^{\gamma} \\
\mA' \ar[r]^{f'} & \mC',
}\qquad
\xymatrix{
\mB \ar[r]^g \ar[d]^{\beta} & \mC \ar[d]^{\gamma} \\
\mB' \ar[r]^{g'} & \mC'
}
\]
commutative squares of $\infty$-categories.
If $\alpha, \beta, \gamma$ are $n$-truncated, the induced functor
$$\mA \overset{\to}{\times}_\mC \mB \to \mA' \overset{\to}{\times}_{\mC'} \mB'$$
is $n$-truncated.
    
\end{proposition}

\begin{proof}

We prove the statement by induction on $n \geq -2$.
Since $-2$-truncated functors are equivalences, the statement for $n=-2$ holds.
Let $n \geq -1$. We assume the statement for $n-1$ and prove the statement for $n.$

The functor $$\mA \overset{\to}{\times}_\mC \mB \to \mA' \overset{\to}{\times}_{\mC'} \mB' $$
is $n$-truncated if it induces on morphism $\infty$-categories $n-1$-truncated functors.
This follows immediately from \cref{homs}, the induction hypothesis and the assumption that the functors $\alpha, \beta, \gamma$ induce $n-1$-truncated functors on morphism
$\infty$-categories.
\end{proof}

\begin{corollary}\label{faithcoro} Let $n \geq -2$ and $\phi: \mA \to \mB$ an $n$-truncated functor and $\mC$ an $\infty$-category. The induced functors $$ \Fun^\oplax(\mC,\mA) \to \Fun^\oplax(\mC,\mB), $$
$$ \Fun^\lax(\mC,\mA) \to \Fun^\lax(\mC,\mB), $$
$$ \Fun(\mC,\mA) \to \Fun(\mC,\mB) $$
are $n$-truncated.
    
\end{corollary}

\begin{proof}

In view of \cref{grayhoms} it suffices to prove the first and third statement. We start with the first statement.
By definition of truncatedness the full subcategory of $\Fun(\bD^1, \infty\Cat)$ spanned by the $n$-truncated functors is closed under limits.
Thus the full subcategory of $\infty\Cat$ spanned by the $\infty$-categories $\mC$
such that the induced functor $ \Fun^\oplax(\mC,\mA) \to \Fun^\oplax(\mC,\mB) $
is $n$-truncated, is closed under small colimits.
Thus by cubical generation we can reduce to the case that $\mC$ is an oriented $n$-cube for some $n \geq 0.$ We prove the statement by induction on $n \geq 0$.
For $n=0$ there is nothing to show. The statement for $n=1$ is a special case of \cref{faithpull}.

The induction step follows from the canonical equivalence
$$ \Fun^\oplax(\cube^1, \Fun^\oplax(\cube^n,\mA)) \simeq  \Fun^\oplax(\cube^{n+1},\mA) $$
and the statement for $n=1.$
This proves the first statement. Next we prove the third statement.
The third statement trivially holds for $n = -2$ since $-2$-truncated functors are precisely equivalences. 

By definition of truncatedness for every $n \geq -2$ the full subcategory of $\Fun(\bD^1,\infty\Cat)$ spanned by the $n$-truncated functors is closed under limits.
Therefore the full subcategory of $\infty\Cat $ spanned by those $\infty$-categories $\mC$, for which the third statement holds,
is closed under small colimits. Therefore we can reduce to the case that $\mC$ is an $\m$-category for $m \geq 0.$
We prove by induction on $m \geq 0$ that the third statement holds for $\mC$ any $m$-category and every $n \geq -1$.

The third statement trivially holds for $\mC$ the final $\infty$-category and therefore also for $\mC$ any space since the full subcategory of spaces is generated under small colimits by the final $\infty$-category.
We prove the induction step.
Let $m \geq 1.$
We assume that the third statement holds for $\mC$ any $m-1$-category and every $n \geq -2$.

Let $\mC$ be an $m$-category and $n \geq -1$ and $\phi: \mA \to \mB$ an $n$-truncated functor.
By \cite[Corollary 4.49.]{heine2024bienriched} for every $F,G \in \Fun(\mC,\mA)$ the induced functor
$$ \Mor_{\Fun(\mC,\mA)}(F,G) \to  \Mor_{\Fun(\mC,\mB)}(\phi F, \phi G) $$ identifies with the functor
$$ \lim_{[n]\in \Delta^\op} \lim_{Z_0,..., Z_n \in \mC}
\Fun(\Mor_\mC(Z_{n-1},Z_n) \times ... \times \Mor_\mC(Z_0,Z_1), \Mor_\mA(F(Z_0), G(Z_n))) \to $$$$ \lim_{[n]\in \Delta^\op} \lim_{Z_0,..., Z_n \in \mC}
\Fun(\Mor_\mC(Z_{n-1},Z_n) \times ... \times \Mor_\mC(Z_0,Z_1), \Mor_\mB(\phi(F(Z_0)), \phi(G(Z_n)))).$$
This functor is $n-1$-truncated by induction hypothesis since the morphism $\infty$-categories of $\mC$ are $m-1$-categories and 
$\phi: \mA \to \mB$ induces $n-1$-truncated functors on morphism $\infty$-categories
and the full subcategory of $\Fun(\bD^1,\infty\Cat)$ spanned by the $n-1$-truncated functors is closed under limits.
\end{proof}

\begin{corollary}Let $n \geq -2.$
The Gray tensor product of two $n$-connected functors is again $n$-connected.

\end{corollary}

\cref{faithcoro} and \cref{enrfactor} imply the following:

\begin{theorem}\label{Grayfacto}
Let $n \geq -2$. The factorization system on $\infty\Cat^\univ$ of \cref{mainfact}, where the left class consists of the $n$-connected ($n+1$-connective) functors and the right class consists of the $n$-truncated ($n+1$-faithful) functors,
is a monoidal factorization system for the Gray monoidal structure.

\end{theorem}

\begin{remark}
It is clear that the factorization system on $\infty\Cat^\univ$ of \cref{mainfact}, where the left class consists of the $n$-connected ($n+1$-connective) functors and the right class consists of the $n$-truncated ($n+1$-faithful) functors,
is a monoidal factorization system for the cartesian symmetric monoidal structure.
This holds because $n$-connected and $n$-truncated functors are evidently closed under products.

\end{remark}

\begin{remark}

Let $n \geq -2$. \cref{Grayfacto} says that the Gray tensor product of two $n$-connected functors is again $n$-connected.
This is equivalent to say that for every $\infty$-category
$X$ the functors $X \boxtimes (-), (-) \boxtimes X: \infty\Cat^\univ \to \infty\Cat^\univ $ preserve $n$-connected functors.

A functor $X \to Y$ is $n$-connected if and only if the induced functor $X^\op \to Y^\op$ is $n$-connected.
By \cref{dua} there is a canonical equivalence
$(X \boxtimes (-))^\op \simeq (-)^\op \boxtimes X^\op$ of functors $ \infty\Cat^\univ \to \infty\Cat^\univ $.

Hence \cref{Grayfacto} is equivalent to the statement that 
for every $\infty$-category
$X$ the functor $ (-) \boxtimes X: \infty\Cat^\univ \to \infty\Cat^\univ $ preserves $n$-connected functors, or equivalently that for every $X \in \infty\Cat^\univ$ the functor $ \Fun^\oplax(X,-): \infty\Cat^\univ \to \infty\Cat^\univ $ preserves $n$-truncated functors.
This precisely means that the factorization system on $\infty\Cat^\univ$ of \cref{mainfact}, where the left class consists of the $n$-connected functors and the right class consists of the $n$-truncated functors, is an oriented factorization system.
    
\end{remark}






\begin{lemma}\label{fibequ2}

Let $0 \leq n \leq \infty.$
Let $X \to S, \ Y \to S$ be functors and
$\phi: X \to Y$ a functor over $S.$
The following are equivalent:
\begin{enumerate}[\normalfont(1)]\setlength{\itemsep}{-2pt}

\item The functor $\phi: X \to Y$ is an $n$-equivalence.

\item For every object $s \in S$ the induced functor
$ \{ s\} \overset{\to}{\times}_S \phi: \{ s\}\overset{\to}{\times}_S X \to \{ s\} \overset{\to}{\times}_S Y$
is an $n$-equivalence.

\end{enumerate}

\end{lemma}

\begin{proof}

The case of $n=\infty$ follows from the finite case and \cref{streqchar}.

(1) clearly implies (2) since $n$-equivalences are stable under pullback as one sees by induction, and the functor of (2) identifies with the functor
$$ \{ s\} \times_S \Fun^\oplax(\bD^1,S) \times_S X \to \{ s\} \times_S \Fun^\oplax(\bD^1,S) \times_S Y$$ 
by \cref {adesc} and so is the pullback of the functor $\phi$. 

We prove that (2) implies (1) by induction on $n \geq 0$.
For $n=0$ we have to see that $\phi$ induces a bijection on equivalence classes of objects if $\phi$ induces on left oriented fibers a bijection on equivalence classes of objects.
If $\phi$ induces on left oriented fibers an essentially surjective functor, then $\phi$ is essentially surjective.
We prove that $\phi$ is essentially injective if $\phi$ induces on left oriented fibers an essentially injective functor.
Let $A,B \in \X$ such that $\tau: \phi(A) \simeq \phi(B).$
The latter equivalence lies over an equivalence $\sigma: s \simeq t$ in $\rS.$
Then $(A,\id_s), (B, \sigma) \in X_s$ and $(\tau,\sigma)$ determines an equivalence
$(\phi(A),\id_s), (\phi(B), \sigma) \in Y_s$ between the images in $Y_s$.
Thus $(A,\id_s) \simeq (B, \sigma) $ in $X_s$ so that $A \simeq B$ in $X.$

This proves the induction start. 
Next we prove the induction step.
We assume the statement holds for $n-1$ and let $\phi: X \to Y$ be a functor
such that for every object $s \in S$ the induced functor 
$$ \{ s\} \overset{\to}{\times}_S \phi: \{ s\}\overset{\to}{\times}_S X \to \{ s\} \overset{\to}{\times}_S Y$$ is an $n$-equivalence.

We have to see that $\phi: X \to Y$ is an $n$-equivalence.
Since every $n$-equivalence is a 0-equivalence, the induction start implies that 
$\phi: X \to Y$ is a $0$-equivalence. Thus $\phi: X \to Y$ is an $n$-equivalence if and only if for every $A,B \in X$ the induced functor
$\rho: \Mor_X(A,B) \to \Mor_Y(\phi(A), \phi(B))$ is an $n-1$-equivalence. Let $s,t \in S$ be the images of $A,B$.
The functor $\rho$ is a functor over $\Mor_S(s,t)$
and therefore by induction hypothesis an $n-1$-equivalence if it induces an $n-1$-equivalence on the left oriented fiber over every morphism $\alpha : s \to t$ in $\rS.$
The functor $\rho$ induces on the left oriented fiber over $\alpha : s \to t$ the functor
$\rho: \Mor_{\{ s\}\overset{\to}{\times}_S X}(A,B) \to \Mor_{\{ s\}\overset{\to}{\times}_S Y}(\phi(A), \phi(B))$,
which is an $n-1$-equivalence by assumption.
\end{proof}

\begin{theorem}\label{trunfiber}
Let $n \geq -2.$
Let $X \to S, \ Y \to S$ be functors and
$\phi: X \to Y$ a functor over $S.$
The following are equivalent:

\begin{enumerate}[\normalfont(1)]\setlength{\itemsep}{-2pt}
\item The functor $\phi: X \to Y$ is $n$-connected ($n$-truncated).
    
\item For every functor $T \to S$ the induced functor $T \overset{\to}{\times}_S \phi: T \overset{\to}{\times}_S X \to T \overset{\to}{\times}_S Y$ is $n$-connected ($n$-truncated).

\item For every object $s \in S$ the induced functor $\{ s\} \overset{\to}{\times}_S \phi: \{ s\}\overset{\to}{\times}_S X \to \{ s\} \overset{\to}{\times}_S Y$ is $n$-connected ($n$-truncated).
    
\end{enumerate}
    
\end{theorem}

\begin{proof}

(2) trivially implies (3). Condition (1) implies (2) by \cref{pullcon} because the functor of (2) 
identifies with the functor 
$$ T \times_S \Fun^\oplax(\bD^1,S) \times_S X \to T \times_S \Fun^\oplax(\bD^1,S) \times_S Y$$ 
by \cref {adesc} and so is the pullback of the functor $\phi$. 

We prove by induction on $n \geq -2 $ that (3) implies (1).
We start with $n=-2$. Since every functor is $-2$-connected, there is nothing to show for the connected version.
On the other hand, the $-2$-truncated functors are precisely the equivalences. So the truncated version is \cref{fibequ2}.

We have proven the induction start.
Next we prove the induction step.

Let $n \geq -1$.
We assume the statement holds for $n-1$ and let $\phi: X \to Y$ be a functor satisfying (3).

We have to see that $\phi: X \to Y$ is $n$-connected ($n$-truncated).
The functor $\phi: X \to Y$ is essentially surjective if and only if for every object $s \in S$ the induced functor $\{ s\}\overset{\to}{\times}_S X \to \{ s\} \overset{\to}{\times}_S Y$ is essentially surjective.

Hence by \cref{induco} it suffices to show that for every $A,B \in X$ the induced functor
$$\rho: \Mor_X(A,B) \to \Mor_Y(\phi(A), \phi(B))$$ is $n-1$-connected ($n-1$-truncated).
Let $s,t \in S$ be the images of $A,B$.
The functor $\rho$ is a functor over $\Mor_S(s,t)$
and therefore by induction hypothesis an $n-1$-connected ($n-1$-truncated) functor
if it induces an $n-1$-connected ($n-1$-truncated) functor on the oriented left fiber over every morphism $\alpha : s \to t$ in $\rS.$
The functor $\rho$ induces on the oriented left fiber over $\alpha : s \to t$ the functor
$$\rho: \Mor_{\{ s\}\overset{\to}{\times}_S X}(A,B) \to \Mor_{\{ s\}\overset{\to}{\times}_S Y}(\phi(A), \phi(B)),$$
which is a $n-1$-connected ($n-1$-truncated) functor by assumption.
\end{proof}


\begin{theorem}\label{conpull} Let $\n \geq -2$ and 
\[
\xymatrix{
\mA \ar[r]^f\ar[d]^{\alpha} & \mC \ar[d]^{\gamma} \\
\mA' \ar[r]^{f'} & \mC',
}\qquad
\xymatrix{
\mB \ar[r]^g \ar[d]^{\beta} & \mC \ar[d]^{\gamma} \\
\mB' \ar[r]^{g'} & \mC'
}
\]
commutative squares of $\infty$-categories.
If $\alpha, \beta$ are $n$-connected and $\gamma$ is $n+1$-connected, the induced functor
$$ \mA \overset{\to}{\times}_\mC \mB \to \mA' \overset{\to}{\times}_{\mC'} \mB' $$
is $\n$-connected.
    
\end{theorem}

\begin{proof}

We prove the statement by induction on $\n \geq -2.$
For $\n=-2$ there is nothing to show.
Let $\n \geq -1.$
The canonical functor
$$\mA \overset{\to}{\times}_\mC \mB \to \mA' \overset{\to}{\times}_{\mC'} \mB'$$
induces on morphism $\infty$-categories between
$$(A, B, f(A) \to \g(B)), (A', B', f(A') \to \g(B')) \in \mA \overset{\to}{\times}_\mC \mB$$ the canonical functor
$$\Mor_\mB(B,B') \overset{\to}{\times}_{\Mor_\mC(f(A),\g(B'))} \Mor_\mA(A,A') \to $$$$ \Mor_{\mB'}(\beta(B),\beta(B')) \overset{\to}{\times}_{\Mor_{\mC'}(f'(\alpha(A)),\g'(\beta(B')))} \Mor_{\mA'}(\alpha(A),\alpha(A')) $$
which is $\n-1$-connected by induction hypothesis.

So suffices to prove that the canonical functor
$$\mA \overset{\to}{\times}_\mC \mB \to \mA' \overset{\to}{\times}_{\mC'} \mB'$$
is essentially surjective.
Since $\n \geq -1,$ the functors $\alpha, \beta$ are essentially surjective and $\gamma$ is essentially surjective and full.

An object of $\mA' \overset{\to}{\times}_{\mC'} \mB'$ is a triple $(A',B', h'): f'(A') \to \g'(B'),$
where $A' \in \mA', B' \in \mB', h': f'(A') \to \g'(B')$.
Since $\alpha, \beta$ are essentially surjective, there are objects
$\A \in \mA, \B \in \mB$ such that $\alpha(\A) \simeq \A', \beta(\B) \simeq \B'.$
Since $\gamma$ is full, the morphism
$$\gamma(f'(A)) \simeq f'(\alpha(A)) \simeq f'(A') \xrightarrow{h} \g'(B') \simeq  \g'(\beta(B)) \simeq \gamma(g'(B)) $$
is the image under $\gamma$ of a morphism $h: f'(A) \to g'(B)$.
The triple $(A,B, h)$ is an object of $\mA \overset{\to}{\times}_\mC \mB$ whose image in $\mA' \overset{\to}{\times}_{\mC'} \mB'$
is $ (A',B', h').$
\end{proof}

\begin{corollary}

Let $\alpha: \mA \to \mC, \beta: \mB \to \mC$ be functors and $\n \geq -1.$
The canonical functor
$$\mA \overset{\to}{\times}_\mC \mB \to \tau_{\leq n}(\mA) \overset{\to}{\times}_{\tau_{\leq n}(\mC)} \tau_{\leq n}(\mB) $$
is $\n-1$-connected and so the canonical functor
$$\tau_{\leq n}(\mA \overset{\to}{\times}_\mC \mB) \to \tau_{\leq n}(\mA) \overset{\to}{\times}_{\tau_{\leq n}(\mC)} \tau_{\leq n}(\mB) $$
is $\n-1$-connected.

\end{corollary}

\begin{corollary}
Let $\alpha: \mA \to \mA', \beta: \mB \to \mB', \gamma: \mC \to \mC'$ be $\infty$-connected functors.
The canonical functor
$$\mA \overset{\to}{\times}_\mC \mB \to \mA' \overset{\to}{\times}_{\mC'} \mB' $$
is $\infty$-connected.
    
\end{corollary}

\begin{corollary}Let $n \geq -1.$
Let $ \gamma: \mC \to \mC'$ be an $n$-connected functor.
The canonical functors
$$\Fun^\oplax(\bD^1,\mC) \to \Fun^\oplax(\bD^1,\mC'), $$
$$\Fun^\lax(\bD^1,\mC) \to \Fun^\lax(\bD^1,\mC') $$
are $n-1$-connected.
Thus for every $m \geq 0$ and $\n \geq m -2$ the canonical functors
$$\Fun^\oplax(\cube^\m,\mC) \to \Fun^\oplax(\cube^\m,\mC'), $$
$$\Fun^\lax(\cube^\m,\mC) \to \Fun^\lax(\cube^\m,\mC') $$
are $n-m$-connected.
    
\end{corollary}


\subsection{Postnikov completion}
















\begin{definition}

Let $\mC$ be an oriented category and $ n \geq -2.$

\begin{enumerate}[\normalfont(1)]\setlength{\itemsep}{-2pt}
\item A morphism $ X \to Y $ in $\mC$ is $n$-truncated 
if for every $Z \in \mC$ the induced functor
$$ \RMor_\mC(Z,X) \to \RMor_\mC(Z,Y) $$ is $n$-truncated.

\item A morphism $ X \to Y $ in $\mC$ is $n$-connected
if it has the left lifting property with respect to all $n$-truncated morphisms.

\item A morphism in $\mC$ is $\infty$-connected
if it is $n$-connected for every $n \geq -2.$

\end{enumerate}

\end{definition}

\begin{example}
Let $\mC$ be an oriented category and $ n \geq -2.$
An object $X$ of $\mC$ is $n$-truncated if and only if for every $Z$ in $\mC$ the morphism $\infty$-category $\RMor_\mC(Z,X)$ is $n$-truncated.
    
\end{example}

\begin{example}
Let $n \geq -2.$ A morphism of $\infty\fcat^\univ$ is $n$-truncated if and only if it is $n$-truncated as a functor by \cref{faithcoro}.

\end{example}

\cref{orientedchar} gives the following corollary:

\begin{corollary}\label{orientedchar2}

Let $\mC$ be an oriented category that admits oriented pullbacks
and let $ n \geq -1.$

A morphism $ X \to Y $ in $\mC$ is $n$-truncated 
if and only if the canonical morphism
$$ X \overset{\to}{\times}_X X \to X \overset{\to}{\times}_Y X $$
is $n-1$-truncated. A morphism $ X \to Y $ in $\mC$ is $0$-truncated 
if and only if it is an equivalence.

\end{corollary}






    

\begin{remark}Let $n \geq -2.$
Every oriented functor $G: \mD \to \mC$ that preserves oriented pullbacks, also preserves $n$-truncated morphisms.
Let $F: \mC \to \mD$ be an oriented functor that admits an oriented right adjoint $G: \mD \to \mC$.
Then $G$ preserves oriented pullbacks and so $n$-truncated morphisms.
Therefore the oriented left adjoint preserves $n$-connected morphisms.


\end{remark}

\begin{notation}Let $\mC$ be an oriented category that admits oriented pullbacks and let $ n \geq - 1.$
Let $\tau_{\leq n}\mC\subset\mC$ be the full subcategory of $n$-truncated objects in $\mC$.

\end{notation}

\cref{enrprefact}, \cref{Grayfacto} and \cref{orientedchar2} imply the following result: 

\begin{theorem}\label{Grayfactor}
Let $\mC$ be a presentable oriented category and $n \geq -2.$
There is an oriented factorization system on $\mC$ whose left class precisely consists of the $n$-connected morphisms and whose right class precisely consists of the $n$-truncated morphisms.
For every oriented adjuntion $F:\mC\rightleftarrows\mD:G$ the oriented right adjoint preserves $n$-truncated morphisms
and the oriented left adjoint preserves $n$-connected morphisms. If the oriented left adjoint preserves oriented pullbacks, it also preserves $n$-truncated morphisms.
\end{theorem}

\begin{corollary}
Let $\mC$ be a presentable oriented category and $n \geq -2.$ The embedding $\tau_{\leq n}\mC\to\mC$ admits an oriented left adjoint
$\tau_{\leq n}: \mC \to \tau_{\leq n}(\mC).$

\end{corollary}



Let $\mC$ be a presentable oriented category and $n \geq -1.$
Since $n$-truncated objects are $n+1$-truncated, there are oriented embeddings $\tau_{\leq n}(\mC) \subset \tau_{\leq n+1}(\mC).$
We obtain the following tower of compatible accessible oriented localizations:

\begin{definition}
Let $\mC$ be a presentable oriented category. The {\em Postnikov tower} is the tower of oriented categories
$$ \{ \tau_{\leq n}: \mC \to \tau_{\leq \n}\mC \}_{\n \geq 0}. $$

\end{definition}

\begin{proposition}\label{oriright}

Let $\mC$ be a presentable oriented category. The oriented functor
$$ \mC \to \lim_{n \geq 0}\tau_{\leq \n}\mC, X \mapsto \{ \tau_{\leq n}(X)\} $$
admits an oriented right adjoint that sends $Y \in \lim_{n \geq 0}\tau_{\leq \n}\mC$ to the limit of the tower $... \to Y_{n+1} \to Y_n \to ... \to Y_0 $,
where the transition morphism $Y_{n+1} \to Y_n$ is the $n$-truncation.

The unit at every $X \in \mC$ is the Postnikov completion
$$ X \to \lim_{n \geq 0} \tau_{\leq n} X. $$

\end{proposition}

\begin{proof}
Since $\mC$ is a presentable oriented category and the oriented functors $\tau_{\leq n}: \mC \to \tau_{\leq \n}\mC$ admits an oriented right adjoint, also the oriented functor
$$ \mC \to \lim_{n \geq 0}\tau_{\leq \n}\mC, X \mapsto \{ \tau_{\leq n}(X)\} $$
admits an oriented right adjoint.
We identify the right adjoint.

For every $Y \in \lim_{n \geq 0} \tau_{\leq n}\mC$ let $i_\infty: \lim_{n \geq 0} \tau_{\leq n}\mC \to \mC $ be the functor sending $Y$ to the limit of the tower $... \to Y_{n+1} \to Y_n \to ... \to Y_0 $,
where the transition morphism $Y_{n+1} \to Y_n$ is the $n$-truncation.

We prove that $i_\infty$ is a right adjoint of $\tau_{\leq \infty}.$
The canonical functor $$ X \to i_\infty(\tau_\infty(X)) \simeq \lim_{n \geq 0} \tau_{\geq n}(X) $$ induces for every $X \in \mC$ and $ Y \in \lim_{n \geq 0} \tau_{\leq n}\mC $
a functor 
$$ \Mor_{\lim_{n \geq 0} \tau_{\leq n}\mC}(\tau_{\geq \infty}(X), Y) \to \Mor_{\mC}(X, i_\infty(Y)).$$
This functor identifies with the canonical equivalence
$$ \Mor_{\lim_{n \geq 0} \tau_{\leq n}\mC}(\tau_{\geq \infty}(X), Y) \simeq $$
$$ \lim_{n\geq 0} \Mor_{\tau_{\leq n}\mC}(\tau_{\leq n}(X), Y_n) \simeq $$
$$ \lim_{n\geq 0} \Mor_{\mC}(X, i_n(Y_n) ) \simeq $$
$$ \Mor_{\mC}(X, \lim_{n\geq 0} i_n(Y_n) ) . $$
\end{proof}

\begin{remark}Let $\mC$ be a presentable oriented category.
For every $X \in \mC$ the object
\[
\tau_\infty(X):= \lim_{n \geq 0} \tau_{\leq n}(X) \in \mC
\]
is hypercomplete, i.e. local with respect to every $\infty$-connected morphism in $\mC$.
Indeed, let $f: A \to B$ be an $\infty$-connected morphism.
Then $f$ is $n$-connected for every $n \geq 0 $ so that the induced map $$\lim_{n \geq 0}  \Map_{\mC}(B, \tau_{\leq n}(X)) \simeq \Map_{\mC}(B, \lim_{n \geq 0} \tau_{\leq n}(X)) \to\Map_{\mC}(A, \lim_{n \geq 0} \tau_{\leq n}(X)) \simeq \lim_{n \geq 0} \Map_{\mC}(A, \tau_{\leq n}(X)) $$ is an equivalence since $\tau_{\leq n}(X)$ is $n$-truncated.
    
\end{remark}

\begin{lemma}\label{fullness}
Let $n \geq 0. $
Let $$ ... \to Y_2 \to Y_1 \to Y_0 $$
be a tower of $n$-connective functors.
For every $m \geq 0$ the functor 
$$ \lim_{k \geq 0} Y_k \to Y_m $$
is $n$-connective.

\end{lemma}

\begin{proof}

We proceed by induction on $n \geq 0.$
For $n=0$ we use that since all functors in the tower
$$ ... \to Y_{m+2} \to Y_{m+1} \to Y_m $$
are essentially surjective, we can lift every object of
$Y_m$ to the limit $ \lim_{k \geq m} Y_k. $ 

We prove the induction step.
We assume the statement holds for $n \geq 0.$
Let $$ ... \to Y_2 \to Y_1 \to Y_0 $$
be a tower of $n+1$-connective functors.
We wish to see that the functor $$ \lim_{k \geq 0} Y_k \to Y_m $$
is $n+1$-connective.
By what we have proven this functor is essentially surjective.
Hence it remains that this functor induces on morphism $\infty$-categories $n$-connective functors.
It induces on morphism $\infty$-categories the projection of the limit of a tower of $n$-connective functors, which is $n$-connective by induction hypothesis.   
\end{proof}

\begin{lemma}\label{tri}
Let $n \geq 0.$
Let $\alpha: X \to Y, \beta: Y \to Z$ be functors of $\infty$-categories.
If $\alpha$ is $n-1$-connective and $\beta \circ \alpha$ is $n$-connective, then $\beta$ is $n$-connective.

\end{lemma}

\begin{proof}

We proceed by induction on $n \geq 0.$
If $n=0$, we have to see that $\beta$ is essentially surjective if $\beta \circ \alpha$ is essentially surjective. This is clear.

We assume the statement for $n$ and prove the statement for $n+1$.
So we assume that $\alpha$ is $n$-connective and that $\beta \circ \alpha$ is $n+1$-connective. We wish to see that $\beta$ is $n+1$-connective.
The induction hypothesis implies that $\beta$ is $n$-connective
since $\alpha$ is $n-1$-connective and $\beta \circ \alpha$ is $n$-connective.
In particular, $\beta$ is $0$-connective, i.e. essentially surjective.
Therefore it suffices to show for every $y,y' \in \mY$
the induced functor \begin{equation}\label{equol}
\Mor_Y(y,y') \to \Mor_Z(\beta(y),\beta(y')) \end{equation} is $n$-connective.
By assumption the functor $\alpha$ is $n$-connective and so $0$-connective, i.e. essentially surjective.
Hence there are $x,x' \in X$ such that $y \simeq \alpha(x), y' \simeq \alpha(x'). $
Hence by induction hypothesis, the functor (\ref{equol}) is $n$-connective since the functor
$ \Mor_X(x,x') \to \Mor_Y(y,y') $ is $n-1$-connective and the composition 
$$ \Mor_X(x,x') \to \Mor_Y(y,y') \to \Mor_Z(\beta(y),\beta(y')) $$
is $n$-connective.
\end{proof}

\begin{corollary}\label{trii}
Let $n \geq 0.$
Let $\alpha: X \to Y, \beta: Y \to Z$ be functors of $\infty$-categories such that $\alpha$ is $n$-connective.
The functor $\beta$ is $n$-connective if and only if $\beta \circ \alpha$ is $n$-connective.

\end{corollary}

\begin{lemma}\label{connec}Let $\n \geq 0.$
For every $Y \in \lim_{n \geq 0} \tau_{\leq n}\infty\fcat^\univ$ and $m \geq 0$ the canonical functor 
$$ \tau_{\leq m}(\lim_{n \geq 0} i_n(Y_n)) \to \tau_{\leq m}(i_m(Y_m)) \simeq Y_m $$ 
is $m$-connected.
    
\end{lemma}

\begin{proof}

For every $\ell \geq m$ the canonical functor
$$ i_{\ell+1}(Y_{\ell+1}) \to i_\ell(Y_\ell) $$ is $\ell$-connected.
Hence for every $\ell > m$ the canonical functor
$$ i_{\ell}(Y_{\ell}) \to i_m(Y_m) $$ is $m$-connected.
We apply \cref{fullness} to deduce that the functor
$$ \lim_{n \geq 0} i_n(Y_n) \to i_m(Y_m) $$ 
is $m$-connected.
The latter functor factors as $$ \lim_{n \geq 0} i_n(Y_n) \to \tau_{\leq m}(\lim_{n \geq 0} i_n(Y_n)) \to \tau_{\leq m}(i_m(Y_m)) \simeq Y_m. $$ 

Since the functor $\lim_{n \geq 0} i_n(Y_n) \to \tau_{\leq m}(\lim_{n \geq 0} i_n(Y_n))$
is $m$-connected, by \cref{trii} also the functor
$$\tau_{\leq m}(\lim_{n \geq 0} i_n(Y_n)) \to \tau_{\leq m}(i_m(Y_m)) \simeq Y_m $$ 
is $m$-connected.
\end{proof}

\begin{lemma}\label{indstart}
Let $\n \geq 0.$ For every $Y \in \lim_{n \geq 0} \tau_{\leq n}\infty\fcat^\univ$ the canonical map of partially ordered sets 
$$ \tau_{\leq 0}(\lim_{n \geq 0} i_n(Y_n)) \to \tau_{\leq 0}(i_0(Y_0)) \simeq Y_0 $$ 
is an isomorphism of partially ordered sets.
    
\end{lemma}

\begin{proof}


By \cref{connec} the functor $$ \tau_{\leq 0}(\lim_{n \geq 0} i_n(Y_n)) \to \tau_{\leq 0}(i_0(Y_0)) \simeq Y_0 $$ is 0-connected.
This means that it is a surjective map of partially ordered sets, which detects the partial order.
It remains to see that the map $$ \tau_{\leq 0}(\lim_{n \geq 0} i_n(Y_n)) \to \tau_{\leq 0}(i_0(Y_0)) \simeq Y_0 $$ is injective.
Since the functor $$ \lim_{n \geq 0} i_n(Y_n) \to \tau_{\leq 0}(\lim_{n \geq 0} i_n(Y_n)) $$ is essentially surjective, it suffices to show that any 
$ A, B \in \lim_{n \geq 0} i_n(Y_n)$ are identified in $ \tau_{\leq 0}(\lim_{n \geq 0} i_n(Y_n)) $ if $A_0 = B_0$ in $Y_0. $
By definition of $\tau_{\leq 0}$ we have to see that there are morphisms
$A \to B, B \to A $ in $ \lim_{n \geq 0} i_n(Y_n).$
So by symmetry, it suffices to see that for every
$ A, B \in \lim_{n \geq 0} i_n(Y_n)$ there is a morphism 
$A \to B $ in $ \lim_{n \geq 0} i_n(Y_n)$ if there is a morphism
$A_0 \to B_0$ in $Y_0.$

For every $n \geq 1$ the functor $ i_n(Y_n) \to i_{n-1}(Y_{n-1}) $ is $n-1$-connected and so in particular $0$-connected, i.e. 1-connective, which means it is full and essentially surjective.

So any morphism $A_0 \to B_0$ in $Y_0$ lifts to a morphism $A_1 \to B_1$ in $Y_1,$ which lifts to a morphism $A_2 \to B_2$ in $Y_2$ and so on.
So we obtain a morphism $A \to B $ in $ \lim_{n \geq 0} i_n(Y_n)$ that maps to the morphism $A_0 \to B_0$ in $Y_0$.
This completes the proof.
\end{proof}

\begin{theorem}\label{Postnikov} Let $\n \geq 0.$
The oriented functor $$\tau_{\leq \infty}: \infty\fcat^\univ \to \lim_{n \geq 0} \tau_{\leq n}\infty\fcat^\univ $$ admits a fully faithful oriented right adjoint $R$.
An $\infty$-category $X$ lies in the essential image of the right adjoint if and only if the canonical functor $$ X \to R(\tau_{\leq \infty}(X)) \simeq \lim_{n \geq 0} \tau_{\leq n}(X) $$ is an equivalence.
The local equivalences are precisely the functors inverted by $\tau_{\leq \infty},$ i.e. inverted by $\tau_{\leq n}$ for every $n \geq 0.$


\end{theorem}

\begin{proof}

If we have proven that the right adjoint exists and is fully faithful, the objects $X$ in the essential image of the right adjoint are precisely the objects $X$, for which the unit $$ X \to R(\tau_{\leq \infty}(X)) \simeq \lim_{n \geq 0} \tau_{\leq n}(X) $$ is an equivalence.
Moreover in this case a functor is local if it is inverted by $\tau_{\leq \infty}$. Thus a functor is local if and only if it is inverted by $\tau_{\leq n}$ for every $n \geq 0$. +

The oriented functor $\tau_{\leq \infty}: \infty\fcat^\univ \to \lim_{n \geq 0} \tau_{\leq n}\infty\fcat^\univ $ admits an oriented right adjoint since for every $n \geq 0$ the oriented functor $\tau_{\leq \infty}: \infty\fcat^\univ \to \tau_{\leq n}\infty\fcat^\univ $ admits an oriented right adjoint and 
$\infty\fcat^\univ$ is a presentable oriented category.

By \cref{oriright} the oriented functor $\tau_{\leq \infty}$ admits an oriented right adjoint $i_\infty$.
It remains to see that this right oriented adjoint is fully faithful.  

We prove that for every $Y \in \lim_{n \geq 0} \tau_{\leq n}\infty\fcat^\univ$ the counit $\tau_{\leq \infty}(i_\infty(Y)) \to Y$ is an equivalence.
For every $m \geq 0$ the image of the counit in $ \tau_{\leq m}\infty\fcat^\univ$ 
is the canonical functor 
\begin{equation}\label{eqqu}
\tau_{\leq m}(\lim_{n \geq 0} Y_n) \to \tau_{\leq m}(Y_m) \simeq Y_m. \end{equation}



We prove by induction on $m \geq 0$ that the functor 
(\ref{eqqu}) is an equivalence for every $Y \in \lim_{n \geq 0} \tau_{\leq n}\infty\fcat^\univ.$

\cref{indstart} gives the induction start.
We prove the induction step.
By \cref{connec} the functor (\ref{eqqu}) is essentially surjective.
Hence it remains to see that the functor (\ref{eqqu}) induces an equivalence on morphism $\infty$-categories.
By construction of $\tau_{\leq m}$ the functor (\ref{eqqu}) induces on $\infty$-categories of morphisms between
$A,B \in \tau_{\leq m}(\lim_{n \geq 0} Y_n)$
the canonical functor
$$ \tau_{\leq m-1}(\lim_{n \geq 0} \Mor_{Y_n}(A,B)) \to \tau_{\leq m-1}(\Mor_{Y_m}(A,B)) \simeq \Mor_{Y_m}(A,B). $$ 

This functor is an equivalence by induction hypothesis since
$$ \{\Mor_{Y_{n+1}}(A,B)\}_{n \geq 0} \in \lim_{n \geq 0} \tau_{\leq n}\infty\fcat^\univ $$
by the inductive definition of $\tau_{\leq k}$ for $k \geq 0.$
\end{proof}






\begin{definition}

An $\infty$-category $X$ is Postnikov complete if the canonical functor 
$$ X \to \lim_{n \geq 0} \tau_{\leq n}(X) $$ is an equivalence.
    
\end{definition}

\begin{remark}

The proof of \cref{counter} shows that the infinite cobordism $\infty$-category is not Postnikov complete.
Thus not all $\infty$-categories are Postnikov complete.
\end{remark}

\begin{definition}\label{weakdir} Let $n \geq 0.$

\begin{enumerate}[\normalfont(1)]\setlength{\itemsep}{-2pt}
\item An $\infty$-category $\mC$ is weakly 0-directed if there is an $n \geq 0$
such that the surjection $$\pi_0(\iota_0(X)) \to \pi_0(\iota_0(\tau_{\leq n}(X)))$$ is bijective.

\item An $\infty$-category $\mC$ is weakly $n+1$-directed if it is weakly 0-directed and all morphism $\infty$-categories are weakly $n$-directed.

\item An $\infty$-category $\mC$ is weakly directed if it is weakly $n$-directed for every $n \geq 0.$

\end{enumerate}

\end{definition}

\begin{remark}
Condition (1) of \cref{weakdir} is equivalent to say that 
there is an $n \geq 0$
such that for every $m \geq n$ the surjection $$\pi_0(\iota_0(X)) \to \pi_0(\iota_0(\tau_{\leq m}(X)))$$ is bijective.
This holds for the following reason: 
for every $m \geq n$ the map 
$$\pi_0(\iota_0(X)) \to \pi_0(\iota_0(\tau_{\leq n}(X)))$$
factors as the map 
$$\pi_0(\iota_0(X)) \to \pi_0(\iota_0(\tau_{\leq m}(X)))$$
followed by the map
$$\pi_0(\iota_0(\tau_{\leq m}(X))) \to \pi_0(\iota_0(\tau_{\leq n}(X))).$$
Thus the map $$\pi_0(\iota_0(X)) \to \pi_0(\iota_0(\tau_{\leq m}(X)))$$ is injective for every $m \geq n$ if and only if the map 
$$\pi_0(\iota_0(X)) \to \pi_0(\iota_0(\tau_{\leq n}(X)))$$ is injective.
\end{remark}

\begin{remark}

Every weakly $n+1$-directed $\infty$-category is weakly $n$-directed.
The morphism $\infty$-categories of every weakly directed $\infty$-category are again weakly directed.
    
\end{remark}

\begin{example}

Let $n \geq 0.$ Every $n$-directed $\infty$-category is weakly $n$-directed.

\end{example}

\begin{lemma}\label{postlem}

Let $\n \geq 0$ and $X$ an $\n$-category.
For every $m \geq n$ the induced map
$$ \pi_0(\iota_0(X)) \to \pi_0(\iota_0(\tau_{\leq m}(X))) $$ is a bijection.
So every $n$-category is weakly 0-directed.

\end{lemma}

\begin{proof}

Since the functor $X \to \tau_{\leq m}(X)$ is $m$-connected, it is essentially surjective. So it suffices to see that for every $m \geq n$ and every $n$-category $X$ the induced map
$ \pi_0(\iota_0(X)) \to \pi_0(\iota_0(\tau_{\leq m}(X))) $ is injective.
To see this, it trivially suffices to assume that $m=n.$
We prove this statement by induction on $n \geq 0.$
For $n=0$ we have to see that every two objects $A,B \in X$ are equivalent if their images in the set $\iota_0(\tau_{\leq 0}(X))$ agree, which means that there are morphisms $A \to B$ and $B \to A$ in $X$.
This is clear since every morphisms in a space is invertible.

We prove the induction step. Let $n \geq 1$. We assume the statement holds for $n-1$. 
Let $A,B \in X$ such that the images of $A,B$ in $\tau_{\leq n}(X)$ are equivalent.
Then there are morphisms $f : A \to B$ and $g : B \to A$ in 
$X$ such that $ g f $ and $fg$ are equivalent to the respective identities in $\tau_{\leq n}(X).$
This is equivalent to say that $ g f $ is equivalent to $\id_A$
in $\Mor_{\tau_{\leq n}(X)}(A,B) \simeq \tau_{\leq n-1}(\Mor_{X}(A,B)) $
and $ f g $ is equivalent to $\id_B$
in $\Mor_{\tau_{\leq n}(X)}(B,A) \simeq \tau_{\leq n-1}(\Mor_{X}(B,A)).$
By induction hypothesis the induced functors
$$\pi_0(\Mor_{X}(A,B)) \to \pi_0(\tau_{\leq n-1}(\Mor_{X}(A,B))), \ \pi_0(\Mor_{X}(B,A)) \to \pi_0(\tau_{\leq n-1}(\Mor_{X}(B,A)))$$
are bijections.
Therefore also $ g f $ is equivalent to $\id_A$
in $\Mor_{X}(A,B) $
and $ f g $ is equivalent to $\id_B$
in $\Mor_{X}(B,A)$ so that $f$ and $g$ are equivalences in $X$.
Thus $A,B$ are equivalent in $X.$
This proves the induction step. Hence for every $n$-category $X$ the induced map
$ \pi_0(\iota_0(X)) \to \pi_0(\iota_0(\tau_{\leq n}(X))) $ is injective,
and therefore also for every $m \geq n$ and every $n$-category $X$ the induced map
$ \pi_0(\iota_0(X)) \to \pi_0(\iota_0(\tau_{\leq m}(X))) $ is injective and so bijective.
\end{proof}

\begin{corollary}

Let $\n \geq 0$.
Every $n$-category is weakly directed.
    
\end{corollary}

\begin{proof}
We proceed by induction on $n \geq 0.$
The induction start $n=0$ is \cref{postlem}.
We assume the statement hold for $n$ and let $X$ be an $n+1$-category.
By \cref{postlem} the $n+1$-category $X$ is  weakly 0-directed.
So it remains to see that all morphism $\infty$-categories of $X$ are weakly $n$-directed. This follows from the induction hypothesis since all morphism $\infty$-categories of $X$ are $n$-categories.
\end{proof}

\begin{theorem}\label{weaklydirpost}
Every weakly directed $\infty$-category is Postnikov complete.
    
\end{theorem}

\begin{proof}

We prove the statement by induction on $n \geq 0.$
For $n=0$ we have to see that every space is Postnikov complete.
This follows from \cref{extop}.
Let $ \geq 1.$ We assume the statement for $n-1$.
Let $X$ be an $n$-category. We like to see that the canonical functor
\begin{equation}\label{post}
X \to \lim_{m \geq 0} \tau_{\leq m}(X) \end{equation} is an equivalence.
The functor (\ref{post}) induces on morphism $\infty$-categories between any two objects $A,B \in X$ the canonical functor
$$ \Mor_X(A,B) \to \lim_{m \geq 0} \tau_{\leq m-1}(\Mor_X(A,B)), $$
which is an equivalence by induction hypothesis.
So it remains to see that the functor (\ref{post}) is essentially surjective.
For that it suffices to see that the induced map
$$ \pi_0(\iota_0(X)) \to \pi_0(\iota_0(\lim_{m \geq 0} \tau_{\leq m}(X)))$$
is essentially surjective.
We prove that the induced map
\begin{equation}\label{mapsu}
\pi_0(\iota_0(X)) \to \lim_{m \geq 0} \pi_0(\iota_0(\tau_{\leq m}(X)))\end{equation}
is essentially surjective, and the induced map 
\begin{equation}\label{mapsu2}\pi_0(\iota_0(\lim_{m \geq 0} \tau_{\leq m}(X))) \to \lim_{m \geq 0} \pi_0(\iota_0(\tau_{\leq m}(X))) \end{equation} is a bijection.

By assumption there is an $n \geq 0$ such that for every $m \geq n$ the induced map
$ \pi_0(\iota_0(X)) \to \pi_0(\iota_0(\tau_{\leq m}(X))) $ is bijective.
Thus for every $m \geq n$ the map $\pi_0(\iota_0(\tau_{\leq m+1}(X))) \to \pi_0(\iota_0(\tau_{\leq m}(X)))$ is bijective. In other words the tower
$$ ... \to \pi_0(\iota_0(\tau_{\leq 1}(X))) \to \pi_0(\iota_0(\tau_{\leq 0}(X))) $$
stabilizes at $n$. Thus the canonical map 
$$ \lim_{m \geq 0} \pi_0(\iota_0(\tau_{\leq m}(X))) \simeq \lim_{m \geq n}  \pi_0(\iota_0(\tau_{\leq m}(X))) \to \pi_0(\iota_0(\tau_{\leq n}(X))) $$
is a bijection. So the map (\ref{mapsu}) identifies with the map 
$\pi_0(\iota_0(X)) \to \pi_0(\iota_0(\tau_{\leq n}(X))) $,
which is surjective since the functor $X \to \tau_{\leq n}(X)$ is
$n+1$-full and so essentially surjective.

Next we prove that the map (\ref{mapsu}) is a bijection.
We first observe that for every tower of spaces $... \to Y_1 \to Y_0$
the induced map $$\pi_0(\lim_{m \geq 0} Y_m) \to \lim_{m \geq 0} \pi_0(Y_m) $$ is surjective.
Moreover the map $\pi_0(\lim_{m \geq 0} Y_m) \to \lim_{m \geq 0} \pi_0(Y_m) $
is bijective if there is some $n \geq 0$ such that for every $m \geq n$
the map $Y_{m+1} \to Y_m$ is full.
This holds because if $A,B \in \lim_{m \geq 0} Y_m$
have equivalent images in $Y_m$ for every $m \geq n,$
then $A,B \in \lim_{m \geq 0} Y_m$
have equivalent images in $Y_n$ and the equivalene between $A$ and $B$ in $Y_n$ can be lifted to an equivalence between $A $ and $B$ in 
$Y_{n+1}$ and further lifted to $Y_{n+2}$, etc. since for every $m \geq n$ the map $Y_{m+1} \to Y_m$ is full.
So we obtain an equivalence $A \to B$ in $\lim_{m \geq 0} Y_m.$ 

Therefore to see that the map (\ref{mapsu}) is a bijection
it suffices to prove that for every $m \geq n$
the map $\iota_0(\tau_{\leq m+1}(X)) \to \iota_0(\tau_{\leq m}(X)) $ is full.
We prove that for every $m \geq n$
the map
\[
\Ho(\iota_0(\tau_{\leq m+1}(X))) \to \Ho(\iota_0(\tau_{\leq m}(X)))
\]
is an equivalence.

By assumption for every $A,B \in X$ there is an $n \geq 0$ such that for every $m \geq n-1$ the induced map
$$\pi_0(\iota_0(\tau_{\leq m+1}(\Mor_X(A,B)))) \to \pi_0(\iota_0(\tau_{\leq m}(\Mor_X(A,B))))$$ is a bijection, which identifies with the map
$$\pi_0(\iota_0(\Mor_{\tau_{\leq m+2}(X)}(A,B))) \to \pi_0(\iota_0(\Mor_{\tau_{\leq m+1}(X)}(A,B))). $$

So for every $m \geq n$ and every $A,B \in X$ the map
$$\pi_0(\iota_0(\Mor_{\tau_{\leq m+1}(X)}(A,B))) \to \pi_0(\iota_0(\Mor_{\tau_{\leq m}(X)}(A,B))) $$ is a bijection.
In other words the induced functor $\Ho(\tau_{\leq m+1}(X)) \to \Ho(\tau_{\leq m}(X))$ 
is fully faithful. Thus the induced map
$\iota_0(\Ho(\tau_{\leq m+1}(X))) \to \iota_0(\Ho(\tau_{\leq m}(X)))$ is an embedding,
which identifies with the map
$$\Ho(\iota_0((\tau_{\leq m+1}(X))) \to \Ho(\iota_0((\tau_{\leq m}(X))). $$
\end{proof}

\begin{corollary}Let $\n \geq 0.$ 

\begin{enumerate}[\normalfont(1)]\setlength{\itemsep}{-2pt}
\item Every $\n$-category is Postnikov complete.

\item Every directed $\infty$-category is Postnikov complete.
    
\end{enumerate}

\end{corollary}

Next we consider the Postnikov tower of objects of $\Theta.$

\begin{lemma}\label{diskdim}
Let $m, n \geq 0$.
Then $\tau_{\leq m}(\bD^n) = \bD^{\min(m,n)+1}.$
    
\end{lemma}

\begin{proof}
We first observe that $\bD^n$ is an $(n-1,n)$-category as one sees by induction on $n \geq 0$ since $\bD^1$ is an $(0,1)$-category.
So if $m \geq n-1$, then $\tau_{\leq m}(\bD^n) = \bD^n.$
On the other hand if $m < n-1$, then $\tau_{\leq m}(\bD^n) = \bD^{m+1}.$ This follows by induction on $\n \geq 1$, where for $n=1$ there is nothing to prove.
\end{proof}

\begin{proposition}
Let $\theta \in \Theta_n := \Theta \cap \n\Cat.$
If $\theta= \colim \bD^i$, then $\tau_{\leq m}(\theta) =\colim \bD^{\min(i,m)+1}$.
In particular, $\tau_{\leq m}(\theta) =\theta$ if $m \geq n-1 $.



    
\end{proposition}

\begin{proof}

Let $\theta \in \Theta_n = \Theta \cap \n\Cat$ and $\theta= \colim \bD^i$. 
We first prove by induction on $n \geq 0$ that every $\theta \in \Theta_n$ is an $(n-1,n)$-category. 
We can write $\theta$ as $\theta_1 \wedge ... \wedge \theta_\ell$ for
$\ell \geq 1$ and $\theta_1, ..., \theta_\ell $ are suspensions of
of objects of $\Theta$ of dimension smaller $n.$
By induction hypothesis we find that $\theta_1, ..., \theta_\ell $
are suspensions of $(n-2,n-1)$-categories and so $(n-1,n)$-categories.
So by the formula for morphism $\infty$-categories in wedges
\cite[Corollary 2.3.4.]{gepner2025oriented} also $\theta $ is an $(n-1,n)$-category.

By \cref{diskdim} we find that $ \tau_{\leq m}(\theta) \simeq \colim(\tau_{\leq m}(\bD^i)) = \colim \bD^{\min(i,m)+1}$,
where the colimit is taken in the category $(m,m+1)\Cat.$
We show that this colimit is preserved by the embedding
$(m,m+1)\Cat \subset \infty\Cat.$
For that it suffices to see that this colimit
$\colim \bD^{\min(i,m)+1}$ in $\infty\Cat$ lies in $(m,m+1)\Cat.$
Let $k:= \max(\{\min(i,m)+1\}_i) \leq m+1.$ Then $\colim \bD^{\min(i,m)+1}$ lies in $\Theta_{k}$ and so by the first part of the proof is an $(m,m+1)$-category.
\end{proof}

\section{\mbox{Homotopy posets of an $\infty$-category}}

\subsection{Homotopy posets}

\begin{definition}
Let $\mC$ be an $\infty$-category and $\n \geq 0.$
\begin{enumerate}[\normalfont(1)]\setlength{\itemsep}{-2pt}
\item 
An oriented base point of $\mC$ of dimension $\n$ is a sequence of pairs
$$(X_\bi,Y_\bi)_{0 \leq \bi \leq \n} $$
such that $(X_\bi,Y_\bi)$ are $\bi$-morphisms $\X_{\bi-1} \to \Y_{\bi-1}$ in $\mC$ for every $0 \leq \bi \leq \n.$
\item
An oriented base point of $\mC$ is a sequence of pairs
$$(X_\bi,Y_\bi)_{\bi \geq 0} $$
such that $(X_\bi,Y_\bi)$ are $\bi$-morphisms $\X_{\bi-1} \to \Y_{\bi-1}$ in $\mC$ for every $\bi \geq 0.$
\end{enumerate}

\end{definition}

\begin{notation}

The inclusion $\emptyset \to \partial\bD^1$ induces an inclusion
$$ \partial\bD^\n= S^n(\emptyset) \to S^n(\partial\bD^1)= S^{\n+1}(\emptyset)=\partial\bD^{\n+1}.$$
Let $$ \partial\bD^\infty:= \colim(\partial\bD^0 \to ... \to \partial\bD^n \to ...) $$

\end{notation}

\begin{remark}Let $\mC$ be an $\infty$-category and $\n \geq 0.$
An oriented base point of $\mC$ of dimension $\n$ is precisely a functor
$ \partial\bD^{\n+1}\to \mC.$
An oriented base point of $\mC$ is precisely a functor
$ \partial\bD^\infty\to \mC.$
\end{remark}

\begin{notation}
Let $\mC$ be an $\infty$-category. 
\begin{enumerate}[\normalfont(1)]\setlength{\itemsep}{-2pt}
\item We set $$\Mor^0_\mC({Z}):=\mC.$$

\item Let $\n \geq 1$ and ${Z}:= (X_\bi,Y_\bi)_{0 \leq \bi \leq n-1} $ an oriented base point of $\mC$ of dimension $\n-1.$ 
We set $$\Mor^\n_\mC({Z}):= \Mor_{\Mor^{\n-1}_\mC({Z})}(X_{\n-1}, Y_{\n-1}).$$
\end{enumerate}
\end{notation}

\begin{lemma}\label{htpy}

Let $\mC$ be an $\infty$-category and $\n \geq 0.$ Let ${Z}:= (X_\bi,Y_\bi)_{0 \leq \bi \leq n-1} $ be an oriented base point of $\mC$ of dimension $\n-1$ corrresponding to a functor $\partial\bD^\n \to \mC.$
There is a canonical equivalence 
$$\Mor^\n_\mC({Z}) \simeq \Fun_{\partial\bD^\n/}^\oplax(\bD^\n,\mC).$$
    
\end{lemma}

\begin{proof}

We prove the statement by induction on $\n \geq 0.$
For $\n=0$ we have a canonical equivalence 
$$\mC \simeq \Fun^\oplax(\bD^0,\mC) \simeq \Fun_{\emptyset/}^\oplax(\bD^0,\mC).$$

There is a canonical equivalence 
$$\Mor^\n_\mC({Z}) \simeq \Mor^{\n-1}_{\Mor_{\mC}(X_0,Y_0)}({Z}) \simeq \Fun_{\partial\bD^{\n-1}/}^\oplax(\bD^{\n-1}, \Mor_{\mC}(X_0,Y_0)) \simeq $$$$ * \times_{\Fun^\oplax(\partial\bD^{\n-1},\Mor_{\mC}(X_0,Y_0))} \Fun^\oplax(\bD^{\n-1},\Mor_{\mC}(X_0,Y_0)) \simeq $$$$ * \times_{\Fun_{\partial\bD^1/}^\oplax(\partial\bD^\n,\mC)} \Fun_{\partial\bD^1/}^\oplax(\bD^\n,\mC) \simeq * \times_{\Fun^\oplax(\partial\bD^\n,\mC)} \Fun^\oplax(\bD^\n,\mC) \simeq \Fun_{\partial\bD^\n/}^\oplax(\bD^\n,\mC).$$
\end{proof}

\begin{lemma}Let $\mC$ be an $\infty$-category and $\n \geq 0.$ Let ${Z}:= (X_\bi,Y_\bi)_{0 \leq \bi \leq n} $ be an oriented base point of $\mC$ of dimension $\n.$
Then $$ \bar{{Z}}:=  (X_1,\id_{Y_0}, \id_{X_1}, Y_2, X_3,\id_{Y_2},\id_{X_3}, Y_4,...) $$ is an oriented base point of $\mC_{//Y_0}$ of dimension $\n-1.$

\end{lemma}

\begin{proof}

We prove the statement by induction on $\n \geq 0.$
For $\n =0$ there is nothing to show.

Let $\n \geq 1.$ We assume the statement holds for $m < n$.
Since ${Z}= (X_\bi,Y_\bi)_{0 \leq \bi \leq n} $ is an oriented base point of $\mC$ of dimension $\n,$ we find that
${Z}':= (X_\bi,Y_\bi)_{2 \leq \bi \leq n} $
is an oriented base point of $\Mor_{\Mor_{\mC}(X_0, Y_0)}(X_1,Y_1))$ of dimension $\n-2.$
By induction hypothesis the sequence $$ (X_3,\id_{Y_2},\id_{X_3}, Y_4,...) $$ is an oriented base point of $(\Mor_{\Mor_{\mC}(X_0, Y_0)}(X_1,Y_1))_{//Y_2}$ of dimension $n-3.$

By \cref{homso} there is a canonical equivalence
$$ \Mor_{\mC_{//Y_0}}(X_1, \id_{Y_0}) \simeq \Mor_{\mC}(X_0, Y_0)_{X_1//}.$$
So there is a canonical equivalence
$$ \Mor_{\Mor_{\mC_{//Y_0}}(X_1, \id_{Y_0})}(\id_{X_1}, Y_2) \simeq \Mor_{\mC}(X_0, Y_0)_{X_1//}(\id_{X_1}, Y_2) \simeq (\Mor_{\Mor_{\mC}(X_0, Y_0)}(X_1,Y_1))_{//Y_2}.$$
Hence the sequence $$ (X_3,\id_{Y_2},\id_{X_3}, Y_4,...) $$ is an oriented base point of $\Mor_{\Mor_{\mC_{//Y_0}}(X_1, \id_{Y_0})}(\id_{X_1}, Y_2)$ of dimension $n-3.$
So the sequence $$ \bar{{Z}}=  (X_1,\id_{Y_0}, \id_{X_1}, Y_2, X_3,\id_{Y_2},\id_{X_3}, Y_4,...) $$ is an oriented base point of $\mC_{//Y_0}$ of dimension $\n-1.$
\end{proof}






\begin{remark}Let $\mC$ be an $\infty$-category and $\n \geq 0.$ Let ${Z}:= (X_\bi,Y_\bi)_{0 \leq \bi \leq n} $ be an oriented base point of $\mC$ of dimension $\n.$
The oriented base point $$ (X_1,\id_{Y_0}, \id_{X_1}, Y_2, X_3,\id_{Y_2},\id_{X_3}, Y_4,...) $$ of $\mC_{//Y_0}$
is sent by the forgetful functor $ \mC_{//Y_0} \to \mC$ to the 
oriented base point $(X_\bi,Y_\bi)_{\bi \geq 0} $ of $\mC.$ 

\end{remark}

\begin{corollary}Let $\phi: \mC \to \mD$ be a functor and $\n \geq 0.$ Let ${Z}:= (X_\bi,Y_\bi)_{0 \leq \bi \leq n} $ be an oriented base point of $\mC$ of dimension $\n.$
Then $$ ({Z}, \overline{\phi{Z}}) $$ is an oriented base point of dimension $\n$ of the oriented right fiber $\mC \overset{\to}{\times}_\mD \{\phi(Y_0)\} \simeq \mC \times_\mD \mD_{//\phi(Y_0)}$. 

\end{corollary}

\begin{definition}Let $\mC$ be an $\infty$-category and $n \geq 0$.
Let $ {Z}:= (X_\bi,Y_\bi)_{0 \leq \bi \leq n}$ and $ {Z}':= (X'_\bi,Y'_\bi)_{0 \leq \bi \leq n} $ oriented base points of $\mC$ of dimension $n.$

By induction on $n \geq 0$ we define what a morphism of oriented base points $ {Z} \to {Z}'$ of $\mC$ of dimension $n$ is.

\begin{enumerate}[\normalfont(1)]\setlength{\itemsep}{-2pt}
\item For $n=0$ a morphism of oriented base points $ {Z} \to {Z}'$ of $\mC$ of dimension $0$ is a pair of morphisms
$$ X'_0 \to X_0, Y_0 \to Y'_0 .$$

\item Let $ {Z}:= (X_\bi,Y_\bi)_{0 \leq \bi \leq n}$ and $ {Z}':= (X'_\bi,Y'_\bi)_{0 \leq \bi \leq n} $ be oriented base points of $\mC$ of dimension $n$ corresponding to
objects $X_0, Y_0, X'_0, Y'_0 $ and oriented base points $\widetilde{{Z}} $ of $\Mor_\mC(X_0,Y_0) $ and $ \widetilde{{Z}'}$ of $\Mor_\mC(X'_0,Y'_0) $ of dimension $n-1$.

Assume we have already defined what a morphism of oriented base points of $\mC$ of dimension $n-1$ is.
A morphism of oriented base points $ {Z} \to {Z}'$ of $\mC$ of dimension $n$ 
is a pair of morphisms $ X'_0 \to X_0, Y_0 \to Y'_0 $
(giving rise to a functor $\theta: \Mor_\mC(X_0,Y_0) \to \Mor_\mC(X'_0,Y'_0) $) and a morphism of oriented base points $$ \theta(\widetilde{{Z}}) \to \widetilde{{Z}'}$$ of $\Mor_\mC(X'_0,Y'_0)$ of dimension $n-1$.

\end{enumerate}
    
\end{definition}

\begin{lemma}

Let $n \geq 0$ and $\phi: \mC \to \mD$ a functor and $ {Z}:= (X_\bi,Y_\bi)_{0 \leq \bi \leq n}$ an oriented base point of $\mC$ of dimension $n-1.$
Let $\nu$ be the canonical functor $\Mor_\mD(\phi(X_0),\phi(Y_0)) \to \mC \times_\mD \mD_{//\phi(Y_0)}.$
There is a canonical morphism
$$ (\id_{X_1},Y_1,X_2, \id_{Y_2}, \id_{X_3},Y_3, ...): \nu\phi{Z} \to ({Z}, \overline{\phi{Z}})$$
of oriented base points of dimension $n$ in $\mC \times_\mD \mD_{//\phi(Y_0)}.$

\end{lemma}

\begin{proof}


We prove the statement by induction on $n \geq 0.$
For $n=0$ there is nothing to show.
Let $n \geq 1$ and we assume the statement holds for $m <n.$
We want to show the statement for $n.$
We first observe that it suffices to show the statement for
$n$ if $\phi: \mC \to \mD$ is the identity.
Let $\omega$ be the canonical functor $\Mor_\mC(X_0,Y_0) \to \mC_{//Y_0}.$
If there is a canonical morphism
$$ (\id_{X_1},Y_1,X_2, \id_{Y_2}, \id_{X_3},Y_3, ...): \omega Z \to \overline{Z} $$
of oriented base points of dimension $n$ in $\mC_{//Y_0},$
then applying the functor 
$ \mC_{//Y_0} \to \mC \times_\mD \mD_{//\phi(Y_0)},$
we obtain the desired morphism $ \nu\phi{Z} \to ({Z}, \overline{\phi{Z}})$
of oriented base points of dimension $n$ in $\mC \times_\mD \mD_{//\phi(Y_0)}.$
Hence it suffices to show the statement for
$n$ if $\phi: \mC \to \mD$ is the identity
and can assume the statement for $m < n$ and any functor $\phi: \mC \to \mD$.
So we have to see that for every oriented base point
$ {Z}= (X_\bi,Y_\bi)_{0 \leq \bi \leq n}$ of $\mC$ of dimension $n$
there is a canonical morphism
$$ (\id_{X_1},Y_1,X_2, \id_{Y_2}, \id_{X_3},Y_3, ...): \omega Z \to \overline{Z}$$
of oriented base points of dimension $n$ in $\mC_{//Y_0}.$


We define a map
$$ (\id_{X_1},Y_1,X_2, \id_{Y_2}, \id_{X_3},Y_3, ...): \omega((X_\bi,Y_\bi)_{\bi \geq 1}) \to (X_1,\id_{Y_0}, \id_{X_1}, Y_2, X_3,\id_{Y_2},\id_{X_3}, Y_4,...)$$
of oriented base points of dimension $n-1$ in $\mC_{//Y_0} $ 
by giving morphisms
$\id : (X_1: X_0 \to Y_0) \to (X_1: X_0 \to Y_0)$ in $\mC_{//Y_0} $ and $Y_1: (Y_1: X_0 \to Y_0) \to (\id_{Y_0}: Y_0 \to Y_0)$ in $\mC_{//Y_0}$ 
giving rise to a functor
$$\omega': \Mor_{\Mor_\mC(X_0,Y_0)}(X_1,Y_1) \to  \Mor_{\mC_{//Y_0}}(X_1, Y_1)\simeq $$$$ \Mor_{\mC}(X_0, X_0)\times_{\Mor_{\mC}(X_0, Y_0)}\Mor_{\mC}(X_0, Y_0)_{X_1//} \to \Mor_{\mC_{//Y_0}}(X_1,\id_{Y_0}) \simeq \Mor_{\mC}(X_0, Y_0)_{X_1//} $$
and a morphism of oriented base points of dimension $n-2$ in $\Mor_{\mC}(X_0, Y_0)_{X_1//} $ from 
$$ \omega'((X_\bi,Y_\bi)_{\bi \geq 2}) \to (\id_{X_1}, Y_2, X_3,\id_{Y_2},\id_{X_3}, Y_4,...)$$

To specify the latter, we give morphisms 
$\id : (Y_2: X_1 \to Y_1) \to (Y_2: X_1 \to Y_1)$ in $\Mor_{\mC}(X_0, Y_0)_{X_1//}$ and $X_2: (\id_{X_1}: Y_1 \to Y_1) \to (X_2: X_1 \to Y_1) $ in $\Mor_{\mC}(X_0, Y_0)_{X_1//}$ 
giving rise to a functor
$$\omega'': \Mor_{\Mor_{\Mor_\mC(X_0,Y_0)}(X_1,Y_1)}(X_2,Y_2) \xrightarrow{\omega'}\Mor_{\Mor_{\mC}(X_0, Y_0)_{X_1//}}(X_2,Y_2) \to \Mor_{\Mor_{\mC}(X_0, Y_0)_{X_1//}}(\id_{X_1},Y_2) $$
together with a map
$$ (\id_{X_3},Y_3, ...): \omega''((X_\bi,Y_\bi)_{\bi \geq 3}) \to (X_3,\id_{Y_2},\id_{X_3}, Y_4,...)$$
of oriented base points in
$$ \Mor_{\Mor_{\mC}(X_0, Y_0)_{X_1//}}(\id_{X_1}, Y_2) \simeq (\Mor_{\Mor_{\mC}(X_0, Y_0)}(X_1,Y_1))_{//Y_2}$$
of dimension $n-3$.
The latter map exists by induction hypothesis.
\end{proof}

\subsection{The oriented exact sequence of homotopy posets}

In the following we consider an $\infty$-category with oriented base point of any dimension as an $\infty$-category with a distinguished object by remembering only the minimal or maximal distinguished $0$-cell.

\begin{definition}\label{oplaxexact} 
An oriented sequence of $\infty$-categories with distinguished object is a sequence
$$... \xrightarrow{\alpha_{\n-1}} A_\n \xrightarrow{\alpha_\n} A_{\n+1} \xrightarrow{\alpha_{\n+1}} A_{\n+2}  \xrightarrow{\alpha_{\n+2}}  ... $$
together with oplax natural transformations
$$ \alpha_{n+1} \circ \alpha_n \to 0$$
for even $n \in \bZ$ 
and oplax natural transformations
$$ 0 \to \alpha_{n+1} \circ \alpha_n$$
for odd $n \in \bZ.$ 
\end{definition}

\vspace{2mm}

\begin{proposition}\label{oriseq}

Let $\phi: \mC \to \mD$ be a functor of $\infty$-categories, ${Z}:=(X_\bi,Y_\bi)_{\bi \geq 0} $ an oriented base point of $\mC$.
Let $\mF$ be the oriented right fiber of $\phi$ over $Y_0$.
There is an oriented sequence of $\infty$-categories
$$ ...\to \Mor^2_\mD(\phi{Z}) \to \Mor^1_\mF(\bar{{Z}}) \to \Mor^1_\mC({Z}) \to \Mor^1_\mD(\phi{Z}) \to \Mor^0_\mF(\bar{{Z}}) \to \Mor^0_\mC({Z}) \to \Mor^0_\mD(\phi{Z}) $$

\end{proposition}

\begin{proof}
This follows immediately from \cref{homso}.   
\end{proof}

\begin{definition}Let $\mC$ be an $\infty$-category.
The poset of components of $\mC$ is the poset
$$\pi_0(\mC):= \tau_{\leq 0}(\mC).$$
\end{definition}

\begin{definition}Let $n \geq 1$ and $\mC$ an $\infty$-category and ${Z}:= (X_\bi,Y_\bi)_{n > \bi \geq 0} $ an oriented base point of $\mC$ of dimension $n-1.$
The $n$-th homotopy poset of $\mC$ is 
$$\pi_\n(\mC,{Z}):= \pi_0(\Mor^\n_\mC({Z})).$$
\end{definition}

\begin{remark}

Let $n \geq 1$ and $\mC$ an $\infty$-category and ${Z}:= (X_\bi,Y_\bi)_{n > \bi \geq 0} $ an oriented base point of $\mC$ of dimension $n-1.$
If $X_{n-1} = Y_{n-1}$, the $n$-th homotopy poset
$$\pi_\n(\mC,{Z}) = \pi_0(\Mor_{\Mor^{n-1}_\mC({Z}))}(X_{n-1},Y_{n-1})$$ 
is a monoidal poset since the functor $\pi_0=\tau_{\leq 0}: \infty\Cat \to \tau_{\leq 0}\infty\Cat$ preserves finite products.
If also $X_{n-2}=Y_{n-2}, $ the $n$-th homotopy poset
$$\pi_\n(\mC,{Z}) = \pi_0(\Mor_{\Mor_{\Mor^{n-2}_\mC({Z}))}(X_{n-2},Y_{n-2})}(X_{n-1},Y_{n-1}))$$ 
is a symmetric monoidal poset.
\end{remark}

\begin{remark}\label{remseq}

\cref{Grayfacto} implies that the functor $\pi_0: \infty\Cat \to \mathrm{Poset}$ is an oriented functor and therefore sends the oriented sequence of \cref{oriseq} to an oriented sequence
of partially ordered sets
$$ \cdots\to \pi_2(\mD, \phi{Z}) \to \pi_1(\mF, \bar{{Z}}) \to \pi_1(\mC, {Z}) \to \pi_1(\mD,\phi{Z}) \to \pi_0(\mF, \bar{{Z}}) \to \pi_0(\mC, {Z}) \to \pi_0(\mD,\phi{Z}). $$

\end{remark}

\begin{theorem}\label{longexact}

Let $\phi: \mC \to \mD$ be a functor of $\infty$-categories, ${Z}:= (X_\bi,Y_\bi)_{\bi \geq 0} $ an oriented base point of $\mC$.
Let $\mF$ be the oriented right fiber of $\phi$ over $Y_0$.
Consider the oriented exact sequence of partially ordered sets of \cref{remseq}:
$$ \cdots\to \pi_2(\mD, \phi{Z}) \to \pi_1(\mF, \bar{{Z}}) \to \pi_1(\mC, {Z}) \to \pi_1(\mD,\phi{Z}) \to \pi_0(\mF, \bar{{Z}}) \to \pi_0(\mC, {Z}) \to \pi_0(\mD,\phi{Z}).$$

Let $\n \geq 0.$

\begin{enumerate}[\normalfont(1)]\setlength{\itemsep}{-2pt} 

\item An object of $\pi_\n(\mC, {Z})$ belongs to the image of the map $\pi_{\n}(\mF, \bar{{Z}}) \to \pi_\n(\mC, {Z})$
if and only if it belongs to the oriented fiber of the map $\pi_\n(\mC, {Z}) \to \pi_\n(\mD, \phi{Z})$.

\vspace{1mm}
\item An object of $\pi_{\n}(\mF, \bar{{Z}})$ 
belongs to the image of the map $\pi_{\n+1}(\mD,\phi{Z}) \to \pi_{\n}(\mF, \bar{{Z}}) $
if and only if its image in $\pi_\n(\mC, {Z})$ is 
the image of $X_{\n+1}$ in $\tau_{\leq 1}(\Mor^n_\mC({Z}))$.

\vspace{1mm}
\item An object of $\pi_{\n+1}(\mD, \phi{Z})$ 
belongs to the image of the map $\pi_{\n+1}(\mC, {Z}) \to \pi_{\n+1}(\mD,\phi{Z}) $
if and only if its image in $\pi_{\n}(\mF, \bar{{Z}})$ is 
the image of $X_{\n+1}$ in $\tau_{\leq 1}(\Mor^n_\mF(\bar{{Z}}))$.

\end{enumerate}

\end{theorem}

\begin{proof}

(1): It suffices to assume that $\n=0.$ We prove the non-trivial direction. Let $T \in \mC$ represent an object of $\pi_0(\mC, {Z})$ that
belongs to the oriented right fiber of the map $\pi_0(\mC, {Z}) \to \pi_0(\mD, \phi{Z})$.
Then there is a morphism $\alpha: \phi(T) \to \phi(Y_0)$ in $\mD. $
The pair $(T, \alpha) $ defines an object of $\mF$ that lies over
$T \in \mC.$
Hence the image of $(T, \alpha) $ in $\pi_{0}(\mF, \bar{{Z}})$
is sent to the image of $T $  in $ \pi_0(\mC, {Z})$.

(2): It suffices to assume that $\n=0.$ We prove the non-trivial direction. Let $(T, \alpha: \phi(T) \to \phi(Y_0)) \in \mF$ representing an object of
$\pi_{0}(\mF, \bar{{Z}})$ whose image in $\pi_0(\mC, {Z})$ represents $X_0$ in $\pi_{1}(\mC, {Z})$.
Then there are morphisms $\kappa: T \to X_0, \theta: X_0 \to T $ in $\mC$ and 2-morphisms $$ \id \to \kappa \theta, \kappa \theta \to \id, \id \to \theta \kappa, \theta \kappa \to \id $$ in $\mC.$ 
The composition $$\zeta: \phi(X_0) \xrightarrow{\phi(\theta)} \phi(T) \xrightarrow{\alpha} \phi(Y_0) $$
represents an object of
$\pi_{1}(\mD,\phi{Z}).$
The image of $\zeta$ in $\pi_{0}(\mF, \bar{{Z}})$
is represented by $(X_0, \zeta)$.
We prove that $(X_0, \zeta)$ and $(T, \alpha)$ represent the same object in $ \pi_{0}(\mF, \bar{{Z}}). $
For that we have to construct morphisms
$$(X_0, \zeta) \to (T, \alpha), \ (T, \alpha) \to (X_0, \zeta)$$ in $\mF.$
The morphism $\theta: X_0 \to T $ defines a morphism
$(X_0, \zeta) \to (T, \alpha)$ in $\mF.$
The morphism $$(\kappa,\alpha \to \zeta \phi(\kappa) = \alpha \phi(\theta) \phi(\kappa)): T \to X_0 $$ defines a morphism
$ (T, \alpha) \to (X_0, \zeta)$ in $\mF.$

(3): It suffices to assume that $\n=0.$ We prove the non-trivial direction.
Let $ T \in \Mor_\mD(X_0,Y_0) $ represent an object of $ \pi_{1}(\mD, \phi{Z})$ whose image in $\pi_{0}(\mF, \bar{{Z}})$,
which is represented by $$(X_0,T: \phi(X_0) \to \phi(Y_0)) \in \mF,$$ 
is represented by $ (X_0, X_1: \phi(X_0) \to \phi(Y_0))$ in $\pi_{1}(\mF, \bar{{Z}}).$
Then there are morphisms $$(\kappa, \sigma) : (X_0, T) \to (X_0,X_1), (\theta, \tau) : (X_0,X_1) \to (X_0, T) $$ in $\mF$
and 2-morphisms $$ \id \to (\kappa, \sigma) \circ (\theta, \tau), (\kappa, \sigma) \circ (\theta, \tau) \to \id, \id \to (\theta, \tau) \circ (\kappa, \sigma), (\theta, \tau) \circ (\kappa, \sigma) \to \id $$ in $\mF.$
The morphisms $(\kappa, \sigma) : (X_0, T) \to (X_0,X_1), (\theta, \tau) : (X_0,X_1) \to (X_0, T) $ in $\mF$ are given by morphisms 
$\kappa, \theta : X_0 \to X_0$ in $\mC$
and 2-morphisms $$ \sigma: T \to \phi(X_1 \kappa), \tau: \X_1 \to T \phi(\theta)$$ in $\mD$.
The 2-morphisms $$ \id \to (\kappa, \sigma) \circ (\theta, \tau), (\kappa, \sigma) \circ (\theta, \tau) \to \id, \id \to (\theta, \tau) \circ (\kappa, \sigma), (\theta, \tau) \circ (\kappa, \sigma) \to \id $$ in $\mF$
lie over 2-morphisms $$ \id \to \kappa \circ \theta, \kappa \circ \theta \to \id, \id \to \theta \circ \kappa, \lambda: \theta \circ \kappa \to \id $$ in $\mD.$

The morphism $X_1 \kappa: X_0 \to Y_0$ in $\mC$ represents an object of
$\pi_1(\mC, {Z})$ whose image in 
$\pi_{1}(\mD,\phi{Z})$ is represented by $\phi(X_1 \kappa).$
We prove that $\phi(X_1 \kappa)$ and $T$ represent the same object of
$\pi_{1}(\mD,\phi{Z})$.
This holds because there are 2-morphisms
$$ \sigma: T \to \phi(X_1 \kappa),\qquad \X_1 \phi(\kappa) \xrightarrow{\tau \phi(\kappa)} T \phi(\theta) \phi(\kappa) \xrightarrow{T \phi(\lambda)} T $$ in $\mD.$ 
\end{proof}

\begin{remark}
In \cref{longexact} (1) the condition refers precisely to oriented right fibers if $n$ is even, and to oriented left fibers if $n$ is odd.
    
\end{remark}

\begin{definition}
Let $n \geq 0.$ 

\begin{enumerate}[\normalfont(1)]\setlength{\itemsep}{-2pt}
\item An $\infty$-category $X$ is $0$-directed if for every pair of objects $A, B \in X$
we have that $A \simeq B$ if there are morphisms $A \to B $ and $ B \to A $.

\item An $\infty$-category is $n+1$-directed if it is 0-directed and all morphism $\infty$-categories are $n$-directed.

\end{enumerate}

\end{definition}

\begin{remark}

Every $n+1$-directed $\infty$-category is $n$-directed.
The morphism $\infty$-categories of every directed $\infty$-category are again directed.
    
\end{remark}

\begin{remark}

An $\infty$-category $X$ is $0$-directed if and only if the surjection $ \pi_0(\iota_0(X)) \to \pi_0(\iota_0(\tau_{\leq 0}(X))) $ 
is an equivalence.

\end{remark}

\begin{definition}

An $\infty$-category is directed if it is $n$-directed for every $n \geq 0.$
    
\end{definition}

\begin{theorem}\label{Whitehead}
Let $0 \leq n \leq \infty.$
A functor $X \to Y$ of directed $\infty$-categories 
is an $n$-equivalence if and only if
for every $0 \leq m \leq n $ and $m-1$-dimensional oriented base point $ {Z}$ of $X$ the induced map of partially ordered sets $$ \pi_m(X, {Z}) \to \pi_m(Y, \phi {Z})$$ is an isomorphism.
 
\end{theorem}

\begin{proof}
We prove the statement by induction on $n \geq 0.$
For the induction start $n=0$ we have to see that a functor $X \to Y$ of directed $\infty$-categories is a $0$-equivalence if and only if the induced map of partially ordered sets $\pi_0(X) \to \pi_0(Y)$ is an isomorphism.
This holds because $\pi_0(X)$ is the set of equivalence classes since $X$ is directed, and similar for $Y.$ 

We prove the induction step. Let $n \geq 1.$ We assume the statement for $n-1$
and prove the statement for $n.$
We assume that for every $0 \leq m \leq n $ and every $m-1$-dimensional oriented base point $ {Z}$ the induced map of partially ordered sets $$ \pi_m(X, {Z}) \to \pi_m(Y, \phi {Z})$$ is an isomorphism.
Then $\pi_0(X) \to \pi_0(Y)$ is an isomorphism so that $\phi$ is a 0-equivalence by the induction start. Let $A,B \in X$ and $0 \leq m \leq n-1 $ and $ {Z}$
a $m-1$-dimensional oriented base point of $\Mor_X(A,B)$.
Let $\phi': \Mor_X(A,B)) \to \Mor_Y(\phi(A),\phi(B))$ be the induced functor.
The $m-1$-dimensional oriented base point $ {Z}$ of $\Mor_X(A,B)$
corresponds to an $m$-dimensional oriented base point $ {Z}'$ of $X.$

By construction the induced map of partially ordered sets $$\pi_m(\Mor_X(A,B), {Z}) \to \pi_m(\Mor_Y(\phi(A),\phi(B)), \phi' {Z})$$ identifies with the 
induced map of partially ordered sets $$\pi_{m+1}(X, {Z}') \to \pi_{m+1}(Y, \phi {Z}'),$$ which is an isomorphism by assumption.
Thus by induction hypothesis for every $A,B \in X$ the induced functor 
$\Mor_X(A,B)) \to \Mor_Y(\phi(A),\phi(B))$ is an $n-1$-equivalence.
Thus $\phi$ is an $n$-equivalence.  
\end{proof}

\begin{remark}
This can drastically fail for $\infty$-categories which are not loop free, such as $\infty\Cat$ itself.
\end{remark}

\begin{example}\label{thetaposet} Let $n,k \geq 1$ and $X_1,..., X_n$ be $\infty$-categories. Let ${Z}$ be an oriented base point of dimension $k-1$ of $S(X_1)\vee ... \vee  S(X_n),$ corresponding to a pair of objects $0 \leq i \leq j \leq n $ and an oriented base point $ {Z}'$ of dimension $k-1$ of $\Mor_{S(X_1)\vee ... \vee S(X_n)}(i,j) \simeq X_{i+1} \times ... \times X_{j},$
where the last equivalence is by \cite[Corollary 2.3.4.]{gepner2025oriented}.
For $\bi \leq \ell \leq \bj$ let ${Z}'_\ell$ be the image of the oriented base point ${Z}'$ in $X_\ell.$

Then $$ \pi_k(S(X_1)\vee ... \vee S(X_n),{Z}) = \pi_{k-1}(X_{i+1}, {Z}'_{i+1}) \times ... \times \pi_{k-1}(X_j, {Z}'_j). $$
In particular, for every $k \geq 1, 0 \leq m \leq k$ and $\infty$-category $X$ we have $$ \pi_k(S^m(X),{Z}) = \pi_{k-m}(X,{Z}'),$$
where ${Z}$ is an oriented base point of dimension $k-1$ of $S^m(X)$ corresponding to an oriented base point $ {Z}'$ of dimension $k-m$ of $X.$

\end{example}

\subsection{Some elementary computations}
In this section we compute the fundamental categories of the orientals and the cubes by calculating its fundamental posets.

\begin{proposition}\label{components}
Let $X$ be a Steiner $\infty$-category.
The poset $\pi_0 X$ is isomorphic to the set of objects of $X$ partially ordered by the relation $A\leq B$ if there is a morphism $A\to B$.
\end{proposition}

\begin{proof}
Since $X$ is gaunt (and so univalent), the space of objects of $X$ is discrete.
Since $X$ is loopfree, if there exists a nonidentity morphism $A\to B$ then there cannot exist a morphism from $B$ to $A$.
Thus the canonical map $\iota_0 X \to \iota_0\pi_0$ is a bijection.
\end{proof}

\begin{proposition}\label{fundamental}
Let $X$ be a Steiner $\infty$-category and $A,B \in X$.
The poset
$ \pi_1(X,(A,B)) $
is canonically isomorphic to the set of sequences of atomic morphisms $ A \to T_1 \to ... \to T_n \to B $, for $ n \geq 0 $ eqipped with the following partial order:
A sequence of atomic morphisms $ A \to S_1 \to ... \to S_m \to B $
is smaller or equal as a sequence of atomic morphisms $ A \to T_1 \to ... \to T_n \to B $ for $ m, n \geq 0$ if and only if 
there is an atomic 2-morphism from a subsequence
$ S_{i_1} \to ... \to S_{i_k}  $ to a subsequence
$ T_{j_1} \to ... \to T_{j_\ell} $.
\end{proposition}

\begin{proof}
Let $X$ be a Steiner $\infty$-category and $X'$ the associated Steiner complex. Let $A,B \in X$.
Objects of $X$ correspond to functors $\bD^0 \to X$ and so correspond to maps of augmented directed chain complexes $C(\Delta^0) \to X',$
which correspond to generators of $X'$.

Morphisms of $X$ from $A$ to $B$ correspond to functors $\bD^1 \to X$ and so correspond to maps of augmented directed chain complexes $C(\Delta^1) \to X'.$ The differential of $C(\Delta^1) $ sends the generator of $\bZ$ in degree 1 to $ (1,-1) \in \bZ \oplus \bZ. $
Hence maps of augmented directed chain complexes $C(\Delta^1) \to X'$
correspond to positive linear combinations of generating 1-chains of $X'$ whose differential is $B-A.$ We note that for every 1-chain $Y$
the differential $\partial_1(Y) = t(Y) - s(Y)$ for unique positive linear combinations $t(Y), s(Y).$
Such positive linear combinations are precisely of the form
$\alpha_1 + ... + \alpha_m$, where $m \geq 0$ and $\alpha_1, ..., \alpha_m$ are generating 1-chains of $X'$ and $\alpha_1=A, \alpha_m=B$
and $s(\alpha_{k+1}) = t(\alpha_k)$ for every $1 \leq k < m.$
\end{proof}

\begin{corollary}
There is a canonical equivalence of posets
\[
\pi_0(\bDelta^n)\cong\Delta^n.
\]
For each nondegenerate $1$-cell $(i<j):\bD^1\to\pi_0(\bDelta^n)$, there is a pullback square of the form
\[
\xymatrix{
S((\bD^{1})^{\times(j-i-1)})\ar[r]\ar[d] & \tau_{\leq 1}(\bDelta^n)\ar[d]\\
\bD^1\ar[r]^{(i<j)} & \tau_{\leq 0}(\bDelta^n).
}
\]
In particular, there are canonical equivalences
$$ \pi_1(\bDelta^n, (i, j))= \bD^{1\times(j-i-1)}$$
for all $0$-dimensional oriented basepoints $(i,j):\partial\bD^1\to\cube^n$ such that $i<j$. If $i=j$ then $\pi_1(\bDelta^n, (i, j))= \bD^0$ and if $i>j$ then $\pi_1(\bDelta^n,(i,j))$ is empty. 
\end{corollary}
\begin{proof}
The $\pi_0$ computation follows from \cref{components}.
To compute the fundamental posets, we use \cref{fundamental}.
The poset 
$ \pi_1(\bDelta^n, (i,j)) $ is the poset of sequences 
$i < k_1 < ... < k_m < j$ for $m \geq 1$
in $\pi_0(\bDelta^n)=\Delta^n$ corresponding to sequences of atomic morphisms. There is a (necessarily) unique atomic morphism $a \to b$ in $\bDelta^n$ if and only if $b=a+1$.
Hence $ \pi_1(\bDelta^n, (i,j)) $ is the poset of sequences 
$i < k_1 < ... < k_m < j$ for $m \geq 1$ such that
$k_{j+1}=k_j+1$ for every $1 \leq k < m$.
This poset is precisely the poset of subposets
$  \{ S \subset \{i+1,...,j-1 \}\}=(\bD^1)^{\times (j-i-1)} .$
\end{proof}

\begin{notation}Let $n \geq 0.$
Let $\bS^n$ denote the set of permutations of the totally ordered set $\{1<\ldots<n\}$ with $n$-elements.
We equip $\bS^n$ with a partial order generated by setting $\rho\leq\sigma$ whenever $\sigma=\tau_{i, \rho(i)}\circ\rho$ for some $1\leq i\leq n$ such
that $\rho(i) < \rho(i + 1)$.
By convention, $\bS^0$ is the singleton set and $\bS^{-\infty}$ is the empty set.
\end{notation}

\begin{remark}
Under this partial ordering, there are minimal and maximal elements: the minimal element is the identity permutation and the maximal element is the permutation which reverses the ordering.
\end{remark}

\begin{notation}
We will identify objects of $\cube^n$ with $n$-tuples $\epsilon=(\epsilon_1,\ldots\epsilon_n)$ such that $\epsilon_i\in\{0,1\}=\partial\bD^1$ for every $1\leq i\leq n$.
For each ordered pair of objects $(\delta,\epsilon)\in\cube^n\times\cube^n$, let $|\epsilon-\delta|=-\infty$ if $\delta_i>\epsilon_i$ for some $1\leq i\leq n$, and otherwise let
\[
|\epsilon-\delta|=\sum_{i=1}^n(\epsilon_i-\delta_i).
\]
Identifying this set with set of subsets $\S\subset\{1,\ldots,n\}$, we obtain a partial ordering on $(\partial\bD^1)^{\times n}$, which we may identify with the poset $(\bD^1)^{\times n}$.
\end{notation}

\begin{corollary}\label{orientalcube} Let $n \geq 0.$
There is a canonical equivalence of posets
\[
\pi_0(\cube^n)\cong(\bD^1)^{\times n}.
\]
For each nondegenerate $1$-cell $(\delta<\epsilon):\bD^1\to\pi_0(\cube^n)$, there is a pullback square of the form
\[
\xymatrix{
S(\bS^{|\epsilon-\delta|})\ar[r]\ar[d] & \tau_{\leq 1}(\cube^n)\ar[d]\\
\bD^1\ar[r]^{(\delta<\epsilon)} & \tau_{\leq 0}(\cube^n).
}
\]
In particular, there are canonical equivalences
$$ \pi_1(\cube^n, (\delta, \epsilon))= \bS^{|\epsilon-\delta|}$$
for all $0$-dimensional oriented basepoints $(\delta,\epsilon):\partial\bD^1\to\cube^n$.    
\end{corollary}

\begin{proof}
The $\pi_0$ computation follows from \cref{components}.
To compute the fundamental posets, we use \cref{fundamental}.
Specifically, the poset $ \pi_1(\cube^n, (\delta, \epsilon))$ is the poset of 
sequences $\delta < \delta_1 < ... < \delta_m < \epsilon $ for $m \geq 1$
in $\pi_0(\cube^n)= (\bD^1)^{\times n} $ corresponding to sequences of atomic morphisms. In $\cube^n$ there is a (necessarily) unique atomic morphism $\delta' \to \delta''$ if and only if $\delta'' \geq \delta'$ in $\pi_0(\cube^n)= (\bD^1)^{\times n}$ and $|\delta''-\delta'|=1.$

Hence the poset $ \pi_1(\cube^n, (\delta, \epsilon))$ is the poset of 
sequences $\delta < \delta_1 < ... < \delta_m < \epsilon $ for $m \geq 1$
in $\pi_0(\cube^n)= (\bD^1)^{\times n} $ such that $ |\delta_{i+1}-\delta_i|=1$ for $ 1 \leq i < m.$
This poset is precisely the poset $ S^{|\epsilon-\delta|}.$
\end{proof}

\section{\mbox{Filtrations of $\infty$-categories}}

\subsection{Filtered dense subcategories}

In topology, the filtration of a cell complex $X$ by $n$-dimensional skeleta often allows one to apply inductive techniques to construct maps  out of $X$, or analyze homological invariants using exact sequences and obstruction theory.
We consider analogous filtrations on the category of $\infty$-categories.
In practice, these arise primarily from the filtration of certain small dense subcategory $\mD\subset\infty\Cat$, such as the orientals or the cubes, by categorical dimension.


Let \[
\mD_0\subset\mD_1\subset\mD_2\subset\cdots
\]
be a sequence of embeddings of small $\infty$-categories and 
$f:\mD:= \colim_{n \geq 0} \mD_n \to\infty\Cat$ a dense functor.
Let $i_n:\mD_n\to\mD$ and $i^m_n:\mD_m\to\mD_n$ be the embeddings, $f_{n!}=(f_n)_!:\mP(\mD_n)\rightleftarrows\infty\Cat: f_n^*$ the unique left adjoint extension of $f_n:=f\circ i_n:\mD_n\to\mD\to\infty\Cat$.

\begin{lemma}\label{dense}

Let \[
\mD_0\subset\mD_1\subset\mD_2\subset\cdots
\]
be a sequence of embeddings of small $\infty$-categories and 
$f:\mD:= \colim_{n \geq 0} \mD_n \to\infty\Cat$ a dense functor.
For every $\infty$-category $X$ there is a canonical equivalence
$$ \colim_{n\geq 0} f_{n!}(f_n^*(X)) \simeq X.$$
\end{lemma}

\begin{proof}

Since $f$ is dense, the restricted Yoneda embedding $\infty\Cat \to \mP(\mD) \simeq \lim_{n\geq 0} \mP(\mD_n), X \mapsto \{f_n^*(X)\}_{n \geq 0}$ is fully faithful. 
Hence for every $Y \in \mD$ there is a canonical equivalence
$$ \Map_\mD(X,Y) \simeq \lim_{n\geq 0} \Map_{\mP(\mD_n)}(f_n^*(X),f_n^*(Y)) \simeq \lim_{n\geq 0} \Map_{\infty\Cat}(f_{n!}(f_n^*(X)),Y) \simeq $$$$ \Map_{\infty\Cat}(\colim_{n\geq 0} f_{n!}(f_n^*(X)),Y) $$ representing an equivalence
$ \colim_{n\geq 0} f_{n!}(f_n^*(X)) \simeq X.$
\end{proof}

\begin{remark}

Let \[
\mD_0\subset\mD_1\subset\mD_2\subset\cdots
\]
be a sequence of embeddings of small $\infty$-categories and 
$f:\mD:= \colim_{n \geq 0} \mD_n \to\infty\Cat$ a dense functor.
Let $\mC_n \subset \infty\Cat$ be the essential image of 
$f_n: \mD_n \subset \mD \xrightarrow{f} \infty\Cat.$
We obtain a sequence \[
\mC_0\subset\mC_1\subset\mC_2\subset\cdots
\]
of embeddings of small $\infty$-categories.
The colimit $\mC:= \colim_{n \geq 0} \mC_n$ is the essential image of $f$.

In general, for every dense functor $f: \mD \to \infty\Cat$
the embedding of the essential image $j: \mC \subset \infty\Cat$ is dense. Let $j_n: \mC_n \to \infty\Cat$ be the restriction of $\bj$
to $\mC_n.$

By \cref{dense} for every $\infty$-category $X$ there is are canonical equivalences
$$ \colim_{n\geq 0} f_{n!}(f_n^*(X)) \simeq X, $$
$$ \colim_{n\geq 0} j_{n!}(j_n^*(X)) \simeq X . $$

\end{remark}

\begin{example}
Let $\mD_n=\bDelta\cap n\Cat \subset\infty\Cat\subset\infty\Cat$ denote the full subcategory spanned by the oriented simplices $\bDelta^m$ for $m\leq n$.
Then $\mD=\colim\mD_n= \bDelta$ is a dense full subcategory of $\infty\Cat$.
\end{example}

\begin{example}
Let $\mD_n=\cube\cap n\Cat \subset\infty\Cat\subset\infty\Cat$ denote the full subcategory spanned by the oriented cubes $\cube^m$ for $m\leq n$.
Then $\mD=\colim\mD_n=\cube$ is a dense full subcategory of $\infty\Cat$.
\end{example}

\begin{example}
Let $\mD_n=\Theta_{\leq n} =\Theta\cap n\Cat \subset\infty\Cat$ denote the full subcategory spanned by those theta which are $n$-categories. Then $\mD=\colim\mD_n= \Theta$ is a dense full subcategory of $\infty\Cat$.

\end{example}

\subsection{The skeletal filtration}

\begin{notation}

We consider the pullback
$$ \prod_{n\in\bN}\mP(\bD^1)\times_{\prod_{n\in\bN}\mP(\bD^0)} \mP(\Delta)$$
of the functor
$ \mP(\Delta) \to \prod_{n\in\bN}\mP(\bD^0), X \mapsto \{X_n\}_{n \geq 0}$
along evalutation at the target.
We denote objects of the pullback as $(X,\mE)$, where $X \in \mP(\Delta)$ and $\mE= \{ \mE_n\}_{n \geq 0}$ is a family
of maps $\mE_n \to \X_n.$

Let $$ \mP(\Delta)^{\#} \subset \prod_{n\in\bN}\mP(\bD^1)\times_{\prod_{n\in\bN}\mP(\bD^0)} \mP(\Delta)$$
spanned by the object $(X,\mE)$ such that for every $n \geq 0$
the fiber of the map $\mE_n \to \X_n$ over every degenerate $n$-simplex of $X$ is contractible.


\end{notation}

\begin{notation}

The projection of the pullback restricts to a forgetful functor 
$\mP(\Delta)^{\#}  \to \mP(\Delta).$
This forgetful functor admits a left and a right adjoint.

\begin{enumerate}[\normalfont(1)]\setlength{\itemsep}{-2pt}
\item The left adjoint equips a simplicial space $X$ with
the family of embeddings $\mE_n \subset X_n$ of the subspace of degenerate $n$-simplices.

\item The right adjoint equips a simplicial space $X$ with
the family of identities $\mE_n = X_n$.

\item The left and right adjoint are both fully faithful since the respective unit and counit are equivalences.

\end{enumerate}

We embed $\mP(\Delta)$ into $\mP(\Delta)^{\#}$ via the left adjoint, which we therefore suppress notationally.

\end{notation}

\begin{remark}
By \cref{fact}, the embedding $$\mP(\bD^1)_\mathrm{mono}\to\mP(\bD^1)$$ admits a left adjoint whose unit is inverted by evaluation at the target, corresponding to the essentially surjective and fully faithful factorization in $\mS$ (\cref{fullsurj}).
This implies that the full subcategory of $\mP(\Delta)^{\#}$
spanned by the pairs $(X,\mE)$ such that the map $X_n \to \mE_n$ is an embedding for every $n \geq 0$, is reflective.
\end{remark}

\begin{notation}

Let $\n \geq 0$.
Let $(\Delta^\n)^t $
be the pair $(\Delta^\n, \mE),$ 
where $\mE_m$ is the subspace of degenerate $m$-simplices of $\Delta^n$ for every $m \neq n$ and $\mE_n$ is the subspace of
degenerate $n$-simplices together with the identity of $[n].$
\end{notation}

\begin{notation}
Let $$\Delta^+ \subset \mP(\Delta)^{\#}$$ be the full subcategory  
spanned by $\Delta^\n$ and $(\Delta^\n)^t.$
We often denote $\Delta^\n$ by $[n]$ and $(\Delta^\n)^t$ by
$[n]^t$ when we view them as object of $\Delta^+.$
\end{notation}

%

\begin{lemma}

The restricted nerve functor
$\gamma: \mP(\Delta)^{\#} \to \mP(\Delta^+)$ is an equivalence.

\end{lemma}

\begin{proof}
For every $(X,\mE) \in \mP(\Delta)^{\#}$
there are natural equivalences
$$ \gamma(X)(\Delta^n) = \Map_{\mP(\Delta)^{\#}}(\Delta^n, (X,\mE)) \simeq \Map_{\mP(\Delta)}(\Delta^n, X) \simeq X_n $$
and 
$$ \gamma(X)((\Delta^n)^t) = \Map_{\mP(\Delta)^{\#}}((\Delta^n)^t, (X,\mE)) \simeq \mE_n.$$

Thus the functors $$  \Map_{\mP(\Delta)^{\#}}(\Delta^n,-), \ \Map_{\mP(\Delta)^{\#}}((\Delta^n)^t, -): \mP(\Delta)^{\#} \to \mS $$
preserve small colimits and the objects
$ \Delta^n, (\Delta^n)^t $ for $\n \geq 0 $
corepresent a family of jointly conservative small colimits
preserving functors.
This implies that the restricted nerve functor
$$\mP(\Delta)^{\#} \to \mP(\Delta^+)$$ is conservative and its left adjoint is fully faithful and so an equivalence.
\end{proof}

\begin{notation}

Let $n \geq 0$.
Let $\Delta^+_{\leq n} \subset \Delta^+$ be the full subcategory spanned by 
$\Delta^\m$ and $(\Delta^\m)^t$ for $m \leq n.$
\end{notation}

\begin{notation}

Let $n \geq 0$ and $X \in \mP(\Delta^+) \simeq \mP(\Delta)^{\#}$.
Let $\sk^{\#}_n(X) \in \mP(\Delta^+) \simeq \mP(\Delta)^{\#}$ be the left Kan extension of the restriction of
$X$ to $\Delta^+_{\leq n}.$

\end{notation}

\begin{remark}

The counit is a map $\sk^{\#}_n(X) \to X$ in $\mP(\Delta^+) \simeq \mP(\Delta)^{\#}$.
The formula for left Kan extensions gives that for every $m \geq 0$
there is a canonical equivalence
$$ \sk^{\#}_n(X) \simeq  \colim_{T \to X, T \in \Delta^+_{\leq n}} T. $$
    
\end{remark}

\begin{remark}
    
Since $\Delta^+$ is the union of all $\Delta^+_{\leq n}$ for $n \geq 0$, for every $X \in \mP(\Delta^+) \simeq \mP(\Delta)^{\#}$
the canonical map
$$ \colim_{n \geq 0}\sk^{\#}_n(X) \to X $$
is an equivalence.

\end{remark}

\begin{construction}\label{joinalg}
Let $(\Delta_{\geq -1}, \ast)$ be the monoidal $\infty$-category of possibly empty totally ordered sets endowed with the join.
By \cite{gepner2025oriented} there is a lax monoidal embedding
$$ (\infty\Cat, \star) \to (\mP(\bDelta), \star) \xrightarrow{R}(\mP(\bDelta_{\geq -1}), \star), $$
where the first embedding is the restricted Yoneda-embedding and
the second embedding is right Kan extension.
The monoidal structure on the right is Day-convolution 
induced by the join monoidal structure on $\bDelta_{\geq -1}$,
the full subcategory of $\infty\Cat$ spanned by the orientals and the empty $\infty$-category.
Therefore an associative algebra structure on an
$\infty$-category $\mC$ with respect to join corresponds to
an associative algebra structure on $R(\Map_{\infty\Cat}(-,\mC)) \in \mP(\bDelta_{\geq -1})$
with respect to Day-convolution, which is the same as a lax monoidal structure on the functor $$R(\Map_{\infty\Cat}(-,\mC)): \bDelta_{\geq -1}^\op \to \mS.$$

We apply this to $\mC= \bD^0.$
Then an associative algebra structure on $\bD^0$ with respect to join corresponds to a lax monoidal structure on the constant functor $\bDelta_{\geq -1}^\op \to \mS$ with value the final object, the tensor unit of $(\mS, \times).$
The constant functor $\bDelta_{\geq -1}^\op \to \mS$ with value the final object factors as monoidal functors $\bDelta_{\geq -1}^\op \to \bD^0 \to \mS$.
So we obtain a canonical associative algebra structure on $\bD^0$ in $(\infty\Cat, \star).$
    
\end{construction}

\begin{notation}

Let $\mC$ be a monoidal category.
By \cite[Proposition 2.2.4.9.]{lurie.higheralgebra} for every associative algebra $A$ in $\mC$ there is a unique 
monoidal functor $(\Delta_{\geq -1}, \ast) \to \mC$ sending
$[0]$ to $A.$
So by \cref{joinalg} there is a unique 
monoidal functor $\bDelta^{(-)}:(\Delta_{\geq -1}, \ast) \to (\infty\Cat, \ast)$ sending $[0]$ to $\bD^0$ and so sending
$[n]$ to $\bDelta^n.$

\end{notation}

\begin{notation}

The functor 
$\Delta \xrightarrow{\bDelta^{(-)}} \infty\Cat$
induces an adjunction
$$ \mP(\Delta) \rightleftarrows \infty\Cat. $$
The right adjoint canonically lifts to
a functor $ \N: \infty\Cat \to \mP(\Delta)^{\#} $
that sends an $\infty$-category $\mC $ to
the space $\Map_{\infty\Cat}(-,\mC)$ equipped with for every $n \geq 0$
the map $$\Map_{\infty\Cat}(\tau_{n-1}\bDelta^\n,\mC) \to \Map_{\infty\Cat}(\bDelta^\n,\mC).$$

The functor $ \N: \infty\Cat \to \mP(\Delta)^{\#} $ admits a left adjoint $| -|: \mP(\Delta)^{\#} \to \infty\Cat $ since it preserves small limits and is accessible as the forgetful functor $\mP(\Delta)^{\#} \to \mP(\Delta) $ preserves limits and colimits.
\end{notation}

\begin{remark}
By adjointness the left adjoint $| - |: \mP(\Delta)^{\#} \to \infty\Cat $ extends the left adjoint $ \mP(\Delta) \to \infty\Cat $
along the free functor. So $|[n]| \simeq \bDelta^n$.
Moreover $| (\Delta^n)^t | \simeq \tau_{n-1}\bDelta^n$
because for every $\mC \in \infty\Cat$ there is a canonical equivalence
$$ \Map_{\infty\Cat}(|(\Delta^n)^t |,\mC) \simeq  \Map_{\mP(\Delta)^{\#}}((\Delta^n)^t,\gamma(\mC)) \simeq \Map_{\infty\Cat}(\tau_{n-1}\bDelta^\n,\mC). $$
\end{remark}


 


The following result is due to Loubaton \cite[Theorem 3.3.2.5.]{loubaton2024complicialmodelinftyomegacategories}:

\begin{theorem}\label{compl}
The functor $\N:\infty\Cat\to \mP(\Delta)^{\#}$ is fully faithful.
\end{theorem}

Next we recall the notion of the $n$-skeleton $\sk_n X$ of a simplicial space $X:\Delta^{\op}\to\mS$.

\begin{definition}
Let $X:\Delta^{\op}\to\mS$ be a simplicial space and fix an integer $n\geq -1$.
The $n$-skeleton $$\sk_n X\subset X$$ is the simplicial subspace of $X$ consisting of those $m$-simplices $\Delta^m\to X$ such that $\Delta^m\to X$ factors as $\Delta^m\to\Delta^\ell\to X$ for some $\ell\leq n$.
\end{definition}

\begin{remark}
There are evident maps $\sk_{m} X\to\sk_n X$ for any $m\leq n$.
These maps are monomorphisms since both are subobjects of $X$.
\end{remark}


\begin{remark}
By construction the canonical map $\colim_{n\geq 0}\sk_n X \to X$ is an equivalence.
\end{remark}

\begin{definition}
A map $\Delta^n\to X$ is {\em nondegenerate} if it does not factor as $\Delta^n\to\Delta^m\to X$ for any $m<n$. We write $$\Map^\nd(\Delta^n,X)\subset\Map(\Delta^n,X)$$ for the full subspace on nondegenerate $n$-simplices of $X$.
\end{definition}

\begin{notation}
Let $\partial\Delta^n:=\sk_{n-1}\Delta^n$.
\end{notation}

\begin{proposition}\label{push}
For every simplicial space $X$ and $n\geq 0$ the canonical commutative square
\[
\xymatrix{
\underset{\Map^\nd(\Delta^n,X)}{\coprod}\partial\Delta^n\ar[r]\ar[d] & \sk_{n-1}X\ar[d]\\
\underset{\Map^\nd(\Delta^n,X)}{\coprod}\Delta^n\ar[r] & \sk_n X
}
\]
is a pushout of simplicial spaces.
\end{proposition}

\begin{proof}
Since both vertical maps are monomorphisms, it suffices to show that the commutative square induces an equivalence on complements of the vertical maps. This holds if the commutative square induces a bijection of equivalence classes on complements of the vertical maps.
To see this, we observe that an $m$-simplex $\Delta^m\to\sk_n X$ which does not factor through the $n-1$-skeleton of $X$ corresponds exactly to a nondegenerate $n$-simplex of $X$.
\end{proof}

    


\begin{proposition}
For every $n\geq 0$ the canonical diagram
\[
\coprod_{0\leq j < k\leq n}\Delta^{n-2}\rightrightarrows\coprod_{0\leq i\leq n}\Delta^{n-1}\to\partial\Delta^n
\]
is a coequalizer of simplicial spaces.
\end{proposition}

\begin{proof}
Let $Q$ be the coequalizer of the two parallel morphisms
$$\coprod_{0\leq j < k\leq n}\Delta^{n-2}\rightrightarrows\coprod_{0\leq i\leq n}\Delta^{n-1}.$$
We proof that the canonical map $Q\to\partial\Delta^n$ is an equivalence. 
By the Yoneda-lemma we have to see that for every $m \geq 0$ the map $\Map(\Delta^m,Q)\to\Map(\Delta^m,\partial\Delta^n)$ is an equivalence. This is equivalent to say that the induced diagram of spaces
\[
\coprod_{0\leq j < k\leq n}\Map(\Delta^m,\Delta^{n-2})\rightrightarrows\coprod_{0\leq i\leq n}\Map(\Delta^m,\Delta^{n-1})\to\Map(\Delta^m,\partial\Delta^n)
\]
is a coequalizer.

To check this, we observe that this diagram is a diagram of simplicial sets. This is because 
$\Delta$ is a $(1,1)$-category and the simplicial space $\partial\Delta^n$ is a subobject of the simplicial space $\Delta^n$.
Note that the above diagram is a coequalizer diagram in the category of simplicial sets
as it is the defining property of the simplicial set $\partial\Delta^n$.
Therefore it suffices to show that the space $\Map(\Delta^m,Q)$ is discrete.

The coequalizer $Q$, viewed as an object over $\Delta^n$, admits a finite filtration $Q=\colim_{0\leq i\leq n} Q_i$, where $Q_0\cong\Delta^{n-1}\to\Delta^n$ is the inclusion of the face opposite the $0$ vertex and $Q_{i+1}$ is obtained from $Q_i$ as the pushout
\[
\xymatrix{
\Delta^{n-1}\times_{\Delta^n} Q_i\ar[r]\ar[d] & Q_i\ar[d]\\
\Delta^{n-1}\ar[r] & Q_{i+1}},
\]
where $\Delta^{n-1}\to\Delta^n$ is the inclusion of the face opposite $i+1$.
Since the face inclusions $\Delta^{n-1}\to \Delta^n$ are monomorphisms, so are the basechanges along $Q_i\to\Delta^n$.
Thus the top horizontal map of the commutative square is a monomorphism.
This reduces to showing that a cobasechange of discrete spaces along monomorphism remains discrete.
Every monomorphism of discrete spaces is a summand inclusion $X \to X + X'$.
And the basechange of an arbitrary map of spaces $X\to Y$ along the canonical map $X\to X+X'$ of spaces is the space $Y+X'$, which is discrete whenever $Y$ and $X'$ are discrete. 
\end{proof}


\begin{definition}
Let $(X,\mE) \in \mP(\Delta)^{\#}.$
The $n$-skeleton $\sk_n (X,\mE)\to (X,\mE)$ of $(X,\mE)$
is $$(\sk_n(X),\mE \cap \sk_n(X)) \to (X,\mE).$$

\end{definition}

\begin{definition}
Let $X$ be an $\infty$-category.
The $n$-skeleton $\sk_n X\to X$ of $X$
is $| \sk_n(\N(X)) | \to | \N(X) | \simeq X.$
    
\end{definition}



\begin{lemma}\label{filsim}

Let $(X,\mE) \in \mP(\Delta)^{\#}.$
The canonical map $$\colim_{n \geq 0}\sk_n(X,\mE) \to (X,\mE)$$
in $\mP(\Delta)^{\#}$ is an equivalence.
\end{lemma}

\begin{proof}

The canonical map $\colim_{n \geq 0}\sk_n(X,\mE) \to (X,\mE)$
in $\mP(\Delta)^{\#}$ lies over the map
$$\colim_{n \geq 0}\sk_n(X) \to X $$
in $\mP(\Delta),$
which is an equivalence.
So it remains to see that 
the latter map induces an equivalence after applying the 
functor $ \Map_{\mP(\Delta)^{\#}}((\Delta^m)^t,-)$
for every $m \geq 0.$
This induced map identifies with the canonical equivalence
$$\colim_{n \geq 0}(\sk_n(X)_m \times_{X_m} \mE_m) \simeq 
(\colim_{n \geq 0}\sk_n(X))_m \times_{X_m} \mE_m
\simeq \mE_m.$$   
\end{proof}

\begin{corollary}

Let $X$ be an $\infty$-category.
The canonical functor $\colim_{n \geq 0}\sk_n(X) \to X$
is an equivalence.
\end{corollary}

\begin{proof}

The canonical functor $\colim_{n \geq 0}\sk_n(X) \to X$
identifies with the realization
of the functor the map $\colim_{n \geq 0}\sk_n(N(X)) \to \N(X)$ in $ \mP(\Delta)^{\#}$ by \cref{compl}. 
The latter map is an equivalence by \cref{filsim}.
\end{proof}

\begin{proposition}
Let $n \geq 0$ and $X \in \mP(\Delta^+) \simeq \mP(\Delta)^{\#}$. The canonical map
$$\sk^{\#}_n(X) \to \sk_n(X)$$
becomes an equivalence after realization.

\end{proposition}

\begin{proof}
The canonical embedding $\sk_n(X) \to X$ in $\mP(\Delta)^{\#}$
induces an equivalence after restriction to 
$\Delta^+_{\leq n}$.
The inverse of this equivalence extends to a map
$\sk^{\#}_n(X) \to \sk_n(X)$ in $\mP(\Delta^+) \simeq  \mP(\Delta)^{\#}.$
By the formula for left Kan extensions the latter map identifies with the canonical map
$$\colim_{T \to X, T \in \Delta^+_{\leq n}} T \to \sk_n(X)$$ in $\mP(\Delta^+) \simeq  \mP(\Delta)^{\#}.$

The realization of this map identifies with the functor
$$\colim_{T \to X, T \in \Delta^+_{\leq n}} |T| \to |\sk_n(X)|. $$
By definition of the realization $\mP(\Delta) \to \infty\Cat $ there is a canonical equivalence
$$ |\sk_n(X)| \simeq \colim_{T \to \sk_n(X), T \in \Delta^+} |T| \simeq \colim_{m \geq n}\colim_{T \to \sk_n(X), T \in \Delta^+_{\leq m}} |T| $$$$ \simeq \colim_{T \to \sk_n(X), T \in \Delta^+_{\leq n}} |T| \simeq \colim_{T \to X, T \in \Delta_{\leq n}^+} |T| . $$
\end{proof}

\begin{notation}
We write $\mE_n^\nd \subset \mE_n$ for the subspace of nondegenerate $n$-simplices of $X$.
\end{notation}

\begin{corollary}

For every stratified simplicial space $(X, \mE)$ and $n\geq 0$, the canonical commutative square
\[
\xymatrix{
\underset{\mE^\nd_\n}{\coprod}\partial\Delta^n\ar[r] \coprod \underset{\Map^\nd(\Delta^n,X)\setminus \mE^\nd_\n}{\coprod}\partial\Delta^n\ar[r] \ar[d] & \sk_{n-1}(X, \mE) \ar[d]\\
\underset{\mE^\nd_\n}{\coprod}(\Delta^n)^t \coprod \underset{\Map^\nd(\Delta^n,X) \setminus \mE^\nd_\n}{\coprod}\Delta^n \ar[r] & \sk_n (X, \mE)
}
\]
is a pushout of stratified simplicial spaces.
    
\end{corollary}

\begin{proof}

The commutative square of the statement lies over a pushout square of simplicial spaces by \cref{push}.
Therefore it suffices to show every thin simplex $\alpha \in 
\mE \cap \sk_n X$ is hit by a thin simplex of the pushout.
If $\alpha$ is degenerate, $\alpha \in \sk_{n-1}X$
and so $\alpha \in \mE \cap \sk_{n-1}X$.
So we can assume that $\alpha$ is non-degenerate.
In particular, $\alpha$ is a simplex of dimension smaller or equal $\n.$ If $\alpha$ is of dimension smaller $\n$, it belongs to
$\mE \cap \sk_{n-1}X$. So we can assume that $\alpha$ is an $\n$-simplex. In this case $\alpha$ is in the image of the map
$$ \underset{\mE^\nd_\n}{\coprod}(\Delta^n)^t \to \sk_n X. $$  
\end{proof}

Applying the realization functor
$$ | - |: \mP(\Delta)^{\#} \to \infty\Cat $$
we obtain the following:

\begin{theorem}\label{skeleta}

For every $\infty$-category $X$ and $n\geq 0$ the canonical commutative square
\[
\xymatrix{
\underset{\Map^\nd_{\infty\Cat}(\tau_{n-1}\bDelta^n,X)}{\coprod}\partial\bD^n\ar[r] \coprod \underset{\Map_{\infty\Cat}^\nd(\bDelta^n,X)\setminus \Map^\nd_{\infty\Cat}(\tau_{n-1}\bDelta^n,X)}{\coprod}\partial\bD^n\ar[r] \ar[d] & \sk_{n-1}(X) \ar[d]\\
\underset{\Map^\nd_{\infty\Cat}(\tau_{n-1}\bDelta^n,X)}{\coprod}\bD^{n-1} \coprod \underset{\Map_{\infty\Cat}(\bDelta^n,X) \setminus \Map^\nd_{\infty\Cat}(\tau_{n-1}\bDelta^n,X)}{\coprod}\bD^n \ar[r] & \sk_n(X)
}
\]
is a pushout of $\infty$-categories.
    
\end{theorem}

\begin{proof}

For every $\infty$-category $X$ and $n\geq 0$ the canonical commutative square
\[
\xymatrix{
\underset{\Map^\nd_{\infty\Cat}(\tau_{n-1}\bDelta^n,X)}{\coprod}\partial\bDelta^n\ar[r] \coprod \underset{\Map_{\infty\Cat}^\nd(\bDelta^n,X)\setminus \Map^\nd_{\infty\Cat}(\tau_{n-1}\bDelta^n,X)}{\coprod}\partial\bDelta^n\ar[r] \ar[d] & \sk_{n-1}(X) \ar[d]\\
\underset{\Map^\nd_{\infty\Cat}(\tau_{n-1}\bDelta^n,X)}{\coprod}\tau_{\n-1}(\bDelta^n) \coprod \underset{\Map_{\infty\Cat}(\bDelta^n,X) \setminus \Map^\nd_{\infty\Cat}(\tau_{n-1}\bDelta^n,X)}{\coprod}\bDelta^n \ar[r] & \sk_n(X)
}
\]
is a pushout of $\infty$-categories.
The following canonical squares are puhout squares:
\[
\xymatrix{
\partial\bD^{n} \ar[r]\ar[d] & \partial\bDelta^n \ar[d]\\
\bD^{n}\ar[r] & \bDelta^n,}
\xymatrix{
\bD^{n} \ar[r]\ar[d] & \bDelta^n \ar[d]\\
\bD^{n-1}\ar[r] & \tau_{n-1}\bDelta^n.}
\]
\end{proof}

\begin{corollary}\label{wedgesphere}
Let $n\geq 0$ and let $X$ be an $\infty$-category.
The cofiber of the canonical functor $$\sk_{n-1}X\to\sk_n X$$ is a wedge of spheres and topological spheres.

\end{corollary}

\begin{proof}

For every $n \geq 0$ the cofiber of the functor $\partial\bD^n \to \bD^n$ is the categorical sphere $B^n(\bN).$
The cofiber of the functor $\partial\bD^n \to \bD^n \to \bD^{n-1}$ is 
$\tau_{n-1 }B^n(\bN) \simeq B^n(\bZ). $
We apply \cref{skeleta} to deduce the result. Here we use that for every functors $X_1 \to Y_1,..., X_n \to Y_n$ for $n \geq 1$
the cofiber of the functor $X_1 \coprod ... \coprod X_n \to Y_1 \coprod ... \coprod Y_n $ is $(Y_1/X_1) \vee ... \vee (Y_n /X_n).$
\end{proof}

\begin{corollary}\label{skeletaconnect}
Let $n\geq 1$ and let $X$ be an $\infty$-category.
The canonical functor $$\sk_{n}X\to\sk_{n+1} X$$ is $n$-connective.

\end{corollary}

\begin{proof}

This follows from \cref{skeleta} since the functors
$\partial\bD^{n+1} \subset \bD^{n+1}$ and $\partial\bD^{n+1} \subset \bD^{n}$ are $n$-connective and $n$-connective functors are the left class of a factorization system (\cref{mainfact}) and so closed under cobasechange.
\end{proof}

\subsection{Skeletal obstruction theory}

\begin{definition}Let $X, Y$ be $\infty$-categories, $n \geq 0$
and $F:\sk_n(X) \to Y$ a functor.
The $\infty$-category of extensions of $F$ to $\sk_{n+1}(X)$
is the fiber $$ \Fun(\sk_{n+1}(X), Y) \to \Fun(\sk_n(X), Y) $$
over $F.$
    
\end{definition}

\begin{corollary}Let $X, Y$ be $\infty$-categories, $n \geq 0$
and $F: \sk_n(X) \to Y$ a functor.
There is a canonical equivalence 
$$ \Fun(\sk_{n+1}(X), Y) \times_{\Fun(\sk_{n}(X), Y)} \{ F \} \simeq $$$$ \underset{\alpha \in \Map^\nd_{\infty\Cat}(\tau_{n-1}\bDelta^n,X)}{\prod} \Mor^n_{Y}(F \circ \alpha_{| \partial\bD^n}) \times \underset{\alpha \in \Map_{\infty\Cat}^\nd(\bDelta^n,X)\setminus \Map^\nd_{\infty\Cat}(\tau_{n-1}\bDelta^n,X)}{\prod} \mathrm{Aut}(\Mor^{n-1}_{Y}(F \circ \alpha_{| \partial\bD^n})). $$
    
\end{corollary}

\begin{proof}
By \cref{skeleta} there is a canonical equivalence 
$$ \Fun(\sk_{n+1}(X), Y) \times_{\Fun(\sk_{n}(X), Y)} \{ F \} \simeq $$
$$\underset{\Map^\nd_{\infty\Cat}(\tau_{n-1}\bDelta^n,X)}{\prod}  \Fun(\bD^n, Y) \times \underset{\Map_{\infty\Cat}^\nd(\bDelta^n,X)\setminus \Map^\nd_{\infty\Cat}(\tau_{n-1}\bDelta^n,X)}{\prod} \Fun(\bD^{n-1}, Y) $$$$ \times_{\underset{ \Map^\nd_{\infty\Cat}(\tau_{n-1}\bDelta^n,X)}{\prod} \Fun(\partial\bD^n, Y) \times \underset{\Map_{\infty\Cat}^\nd(\bDelta^n,X)\setminus \Map^\nd_{\infty\Cat}(\tau_{n-1}\bDelta^n,X)}{\prod} } \{ \{F \circ \alpha_{| \partial\bD^n}\}_{\alpha} \} \simeq  $$
$$ \underset{\alpha \in\Map^\nd_{\infty\Cat}(\tau_{n-1}\bDelta^n,X)}{\prod}  \Fun_{\partial\bD^n}(\bD^n, Y) \times \underset{\alpha \in\Map_{\infty\Cat}^\nd(\bDelta^n,X)\setminus \Map^\nd_{\infty\Cat}(\tau_{n-1}\bDelta^n,X)}{\prod} \Fun_{\partial\bD^n}(\bD^{n-1}, Y) \simeq  $$
$$ \underset{\alpha \in\Map^\nd_{\infty\Cat}(\tau_{n-1}\bDelta^n,X)}{\prod}  \Mor^n_{Y}(F \circ \alpha_{| \partial\bD^n}) \times \underset{\alpha \in\Map_{\infty\Cat}^\nd(\bDelta^n,X)\setminus \Map^\nd_{\infty\Cat}(\tau_{n-1}\bDelta^n,X)}{\prod} \mathrm{Aut}(\Mor^{n-1}_{Y}(F \circ \alpha_{| \partial\bD^n})).$$  
\end{proof}

\begin{notation}

Let $X$ be an $\infty$-category, $n \geq 0$ and
${Z} $ an oriented base point of dimension $n.$
Let $$ \pi'_n(Y,{Z}) \subset \pi_n(Y,{Z}) $$
be the partially ordered subset of elements represented by
functors $\bD^n \to Y$ under $\partial\bD^n $
corresponding to functors $\bD^1 \to \Fun_{\partial\bD^{n-1}}(\bD^{n-1}, Y) $ under $\partial\bD^1 $ that factor trough $*/ \partial \bD^1= B\bZ.$

\end{notation}

\begin{corollary}\label{obstruction}
Let $X, Y$ be $\infty$-categories, $n \geq 0$
and $F:\sk_n(X) \to Y$ a functor.
There is a canonical equivalence 
$$ \tau_{\leq 0}(\Fun(\sk_{n+1}(X), Y) \times_{\Fun(\sk_{n}(X), Y)} \{ F \}) \simeq $$$$ \underset{\alpha \in \Map^\nd_{\infty\Cat}(\tau_{n-1}\bDelta^n,X)}{\prod} \pi_n(Y,F \circ \alpha_{| \partial\bD^n}) \times \underset{\alpha \in \Map_{\infty\Cat}^\nd(\bDelta^n,X)\setminus \Map^\nd_{\infty\Cat}(\tau_{n-1}\bDelta^n,X)}{\coprod} \pi'_n(Y,F \circ \alpha_{| \partial\bD^n}). $$
    
\end{corollary}

\bibliographystyle{plain}

\bibliography{bib}

\end{document}